\documentclass[notitlepage,12pt]{article}
\usepackage[tbtags]{amsmath}
\usepackage{amsfonts}
\usepackage{amsthm}
\usepackage{color,xcolor}
\usepackage{rotating}
\usepackage{graphicx,amssymb}
\usepackage{mathrsfs}
\usepackage{epstopdf}
\usepackage{enumerate}
\usepackage{enumitem}
\usepackage{pdflscape}
\usepackage{geometry}
\usepackage{natbib}
\usepackage{changepage}
\usepackage{mathtools} 
\usepackage{mathabx}  
\usepackage[colorlinks=true,linkcolor=blue,urlcolor=black,citecolor=blue,bookmarksopen=true]{hyperref}
\usepackage{bookmark}
\usepackage{afterpage}

\newcommand{\Date}[1]{\def\@Date{#1}}
\def\today{\number\day~\ifcase\month\or
	January\or February\or March\or April\or May\or June\or
	July\or August\or September\or October\or November\or December\fi~\number\year}

\theoremstyle{theorem}
\newtheorem{lemma}{Lemma}
\newtheorem{proposition}{Proposition}
\newtheorem{corollary}{Corollary}
\newtheorem{assumption}{Assumption}
\newtheorem{definition}{Definition}
\newtheorem{theorem}{Theorem}
\theoremstyle{definition}
\newtheorem{example}{Example}
\newtheorem{remark}{Remark}
\def\Bka{{\it Biometrika}}

\def\VAR{\textsc{var}}
\def\VARX{\textsc{varx}}
\def\BMSB{\textsc{bmsb}}
\def\DGP{\textsc{DGP}}
\def\T{{ \mathrm{\scriptscriptstyle T} }}

\DeclareMathOperator*{\argmin}{arg\,min}

\DeclareMathOperator*{\diag}{diag}

\DeclareMathOperator*{\var}{var}
\DeclareMathOperator*{\cond}{cond}
\DeclareMathOperator*{\smb}{sb}
\DeclareMathOperator*{\kl}{KL}
\DeclareMathOperator*{\vect}{vec}
\DeclareMathOperator*{\tr}{tr}
\DeclareMathOperator*{\rank}{rank}
\def\P{{\mathrm{pr}}}

\numberwithin{equation}{section}
\begin{document}

\title{\vspace{-15mm} \bf Finite Time Analysis of Vector Autoregressive Models under Linear Restrictions}

\author{Yao Zheng$^{\dag}$ \qquad  Guang Cheng$^\ddag$\\
	$^\dag$Department of Statistics, University of Connecticut,\\
	Storrs, CT 06269, USA\\
	$^{\ddag}$Department of Statistics, Purdue University, \\
	West Lafayette, IN 47906, USA\\
	yao.zheng@uconn.edu \ \ chengg@purdue.edu 
}

\maketitle

\begin{abstract}
This paper develops a unified finite-time theory for the ordinary least squares estimation of possibly unstable and even slightly explosive vector autoregressive models under linear restrictions, with the applicable region  $\rho(A)\leq 1+c/n$, where $\rho(A)$ is the spectral radius of the transition matrix $A$ in the \VAR(1) representation, $n$ is the time horizon and $c>0$ is a universal constant. The linear restriction framework encompasses various existing models such as banded/network vector autoregressive models. We show that the restrictions reduce the error bounds via not only the reduced dimensionality but also  a scale factor resembling the asymptotic covariance matrix of the estimator in the fixed-dimensional setup: as long as the model is correctly specified, this scale factor is decreasing in the number of restrictions. It is revealed that the phase transition from slow to fast error rate regimes is determined by the smallest singular value of $A$, a measure of the least excitable mode of the system. The minimax lower bounds are derived across different regimes. The developed non-asymptotic theory not only bridges the theoretical gap between stable and unstable regimes but precisely characterizes the effect of restrictions and its interplay with model parameters. Simulations support our theoretical results. 
\end{abstract}
 

\section{Introduction\label{sec:intro}} 

The vector autoregressive model \citep{Sims1980} is arguably the most fundamental model for multivariate time series \citep{Lutkepohl2005, Tsay2013}.
Applications of the model and its variants can be found in almost any field that involves learning the temporal dependency: economics and finance \citep{Wu_Xia2016}, energy forecasting \citep{Dowell_Pinson2016}, psychopathology \citep{Bringmann2013}, neuroscience \citep{Gorrostieta2012} and reinforcement learning \citep{Recht2018}, etc.

Consider the vector autoregressive model of order one, \VAR(1), in the following form:
\begin{equation*}
X_{t+1}= A X_t + \eta_t,
\end{equation*}
where $X_t\in\mathbb{R}^d$ is the observed time series, $A\in\mathbb{R}^{d\times d}$ is the unknown transition matrix, and $\eta_t\in\mathbb{R}^d$ is the innovation. In modern applications, the dimension $d$ is often relatively large. However, since the number of unknown parameters increases as  $d^2$, leading to problems such as over-parametrization, the  model cannot provide reliable estimates or forecasts without further restrictions \citep{Stock_Watson2001}.  A classical approach to dimensionality reduction for vector autoregressive models, which recently has enjoyed a resurgence of interest, advocates the incorporation of prior knowledge into  modelling.  For example, motivated by the fact that in spatio-temporal studies, it is often sufficient to collect information from `neighbors', \cite{Guo_Wang_Yao2016} proposed the banded vector autoregressive model, where the nonzero entries of $A$ are assumed to form a narrow band along the main diagonal, after arranging the $d$ components of $X_t$ by geographic location. To analyse users' time series over large social networks, the network vector autoregressive model of \cite{Zhu_Pan_Li_Liu_Wang2017}  used the follower-followee adjacency matrix to determine the zero-nonzero pattern of $A$, together with equality restrictions to further reduce the dimensionality. In fact, the above models can both be incorporated by the general framework of linear restrictions 
\begin{equation}\label{eq:lr}
\mathcal{C}\text{vec}(A^\T)=\mu,
\end{equation}
where $\mathcal{C}$ is a prespecified restriction matrix, $\mu$ is a known constant vector, and $A^\T$ is the transpose of $A$. This form of restrictions is traditionally well known by time series modellers; see  books on multivariate time series analysis  such as \cite{Reinsel1993}, \cite{Lutkepohl2005} and \cite{Tsay2013}. 

Meanwhile, drawing inspiration from recent developments in high-dimensional regression, another well studied approach concerns the penalized estimation, where the modeller is agnostic to  the locations of nonzero coordinates in $A$ while assuming certain sparsity \citep{Davis_Zang_Zheng2015, Han_Lu_Liu2015, Basu_Michailidis2015}, or the directions of low-dimensional projections while, e.g., assuming a low rank structure of $A$ \citep{Ahn_Reinsel1988,Negahban_Wainwright2011}. Note that in the former case, once the locations of nonzero coordinates are identified, the model can be formulated as an instance of \eqref{eq:lr}. Although we focus on fully known restrictions, the framework of \eqref{eq:lr} allows us to study the theoretical properties of a much richer variety of restriction patterns in this paper.  

On the other hand, in the literature on large vector autoregressive models, there has been an almost exclusive focus on stable processes. Technically, this means requiring that the spectral radius $\rho(A)<1$, or often more stringently, that the spectral norm $\lVert A \rVert_2<1$.  However, the analysis of stable processes typically cannot be carried over to unstable processes. In this paper, we provide a novel finite-time (non-asymptotic) analysis of the ordinary least squares estimator for stable, unstable and even slightly explosive vector autoregressive models within the general framework of \eqref{eq:lr}. 

Our analysis sheds new light on the phase transition phenomenon of the ordinary least squares estimator across different stability regimes. This is made possible by adopting the non-asymptotic, non-mixing approach of \cite{Simchowitz2018}. Resting upon a generalization of Mendelson's \citeyearpar{Mendelson2014} small-ball method, this approach is particularly attractive because: (i) it unifies stable and unstable cases, whereas in asymptotic theory these two cases would require substantially different techniques; and (ii) in contrast to existing non-asymptotic methods, it can well capture the fundamental  trait that the estimation will be more accurate as $\rho(A)\rightarrow1$. While relaxing the normality assumption in \cite{Simchowitz2018}, we precisely characterize the impact of imposing restrictions on the estimation error.  More importantly, we reveal for the first time that the phase transition from slow to fast error rates depends on the the smallest singular value of $A$, a measure of the least excitable mode of the system.  In addition, we expand the applicable region  $\rho(A)\leq 1$ in the above paper to $\rho(A)\leq 1+c/n$, where  $c>0$ is a universal constant, so  slightly explosive processes are also included.
Compared to \cite{Guo_Wang_Yao2016} which focused on the case with $\lVert A \rVert_2<1$, our assumption on the innovation distribution is stronger, and our error rate for the stable regime is larger by a logarithmic factor which, however, can be dropped under the normality assumption; see  $\S$ \ref{subsec:var_ub} for details.  Note that although \cite{Zhu_Pan_Li_Liu_Wang2017} relied on even milder assumptions on the innovations than \cite{Guo_Wang_Yao2016},  they assumed that the number of unknown parameters is fixed, and hence their theoretical analysis is incomparable to ours; see Example \ref{ex:networkAR} in $\S$ \ref{subsec:var_example}.

Throughout we denote by $\lVert\cdot\rVert$ the Euclidean norm and by $\mathcal{S}^{d-1}=\{\omega\in\mathbb{R}^d: \lVert \omega \rVert=1\}$ the unit sphere in $\mathbb{R}^d$. For a real matrix $A=(a_{ij})$, we let $\lambda_{\max}(A)$ and $\lambda_{\min}(A)$ (or $\sigma_{\max}(A)$ and  $\sigma_{\min}(A)$) be its largest and smallest eigenvalues (or singular values), respectively; additionally, we let $\rho(A)=\vert \lambda_{\max}(A) \vert, \lVert A\rVert_2=\sup_{\lVert\omega\rVert=1} \lVert A\omega\rVert$ and  $\lVert A\rVert_F=(\sum_{i,j}a_{ij}^2)^{1/2}$ be the spectral radius, spectral norm and Frobenius norm of $A$, respectively. For $x\in\mathbb{R}$, let $\lfloor x\rfloor=\max\{k\in\mathbb{Z}: k\leq x\}$ and $\lceil x\rceil=\min\{k\in\mathbb{Z}: k\geq x\}$, where $\mathbb{Z}$ is the set of integers. We write $A\succ0$ (or $A\succeq 0$) if $A$ is a positive definite (or positive semidefinite) matrix. Moreover, for any real symmetric matrices $A$ and $B$, we write $A\prec B$ (or $A \preceq B$) if $B-A\succ0$ (or $B-A\succeq0$), and write $A\nprec B$ (or $A \npreceq B$) if $A\prec B$ (or $A \preceq B$)  does not hold.  For any quantities $X$ and $Y$, we write $X \gtrsim Y$ if there exists a universal constant  $c>0$ independent of $(n,d,m,R,k,\sigma,\delta)$, whose meaning will become clear later, such that $X \geq c Y$. 

\section{Linearly restricted stochastic regression \label{sec:Reg_upperbd}}
\subsection{Problem formulation\label{subsec:formulation}}
Consider a sequence of time-dependent covariate-response pairs $\{(X_t, Y_t)\}_{t=1}^n$ following
\begin{equation}\label{eq:model_regression}
Y_t=A_*X_t+ \eta_t, 
\end{equation}
where $Y_t, \eta_t \in\mathbb{R}^q, X_t\in\mathbb{R}^d$, $A_*\in\mathbb{R}^{q\times d}$, and $E(\eta_t)=0$.  In particular, \eqref{eq:model_regression} becomes the \VAR(1) model  when $Y_t=X_{t+1}$. In model \eqref{eq:model_regression}, the process $\{X_t, t=1,2,\ldots\}$ is adapted to  the filtration \[\mathcal{F}_{t}=\sigma\{\eta_1,\ldots,\eta_{t-1},X_1,\ldots, X_{t}\}.\]

Let $\beta_*=\text{vec}(A_*^\T)\in\mathbb{R}^{N}$, where $N=q d$. The parameter space of the linearly restricted model can be defined as 
\[\mathcal{L} = \{\beta\in\mathbb{R}^N: \mathcal{C}\beta=\mu\},\]
where $\mathcal{C}$ is a known  $(N-m)\times N$ matrix of rank $N-m$, representing $N-m$ independent restrictions, and $\mu\in\mathbb{R}^{N-m}$ is a known constant vector which may simply be set to zero in practice.  Let $\mathcal{C}_{+}$ be an $m\times N$ complement of $\mathcal{C}$ such that $\mathcal{C}_{\text{full}}=(\mathcal{C}_{+}^{\T}, \mathcal{C}^\T)^\T$ is invertible, with its inverse partitioned into two blocks as $\mathcal{C}_{\text{full}}^{-1}=(R, R_{+})$, where $R$ is the  matrix of the first $m$ columns of $\mathcal{C}_{\text{full}}^{-1}$. Additionally, define $\gamma=R_{+}\mu$. Then it holds  $\mathcal{C}\gamma=\mathcal{C}R_{+}\mu=\mu$.  Note that if $\mathcal{C}\beta=\mu$, then $\beta=\mathcal{C}_{\text{full}}^{-1}\mathcal{C}_{\text{full}}\beta=R\mathcal{C}_{+}\beta+R_{+}\mathcal{C}\beta=R\theta+\gamma$, where $\theta=\mathcal{C}_{+}\beta$. Conversely, for any $\theta\in\mathbb{R}^m$, if $\beta=R\theta+\gamma$, then $\mathcal{C}\beta=\mathcal{C}R\theta+\mathcal{C}\gamma=\mu$. Thus, we have 
\[
\mathcal{L} =\{R\theta+\gamma: \theta\in\mathbb{R}^m\},
\]
i.e., the linear space spanned by columns of the restriction matrix $R$, shifted by the constant vector $\gamma$. This immediately implies that, given $(R, \gamma)$, there exists a unique unrestricted parameter $\theta_*\in\mathbb{R}^m$ such that $\beta_*=R\theta_*+\gamma$. Note that  $\gamma=0$ if and only if $\mu=0$. Moreover, the unrestricted model corresponds to the special case where $R=I_N$ and $\gamma=0$. 

The following examples illustrate how the linear restrictions can be encoded by  $(R,\gamma)$ or $(\mathcal{C}, \mu)$, where without loss of generality we set $\mu=\gamma=0$. Let $\beta_{*i}$ denote the $i$th entry of $\beta_*$. 

\begin{example}[Zero restriction]
	The restriction $\beta_{*i}=0$ may be encoded by  setting the $i$th row of $R$ to zero, or  by setting a row of $\mathcal{C}$ to $(0,\ldots, 0, 1, 0,\ldots, 0)\in\mathbb{R}^N$, where the $i$th entry is one.
\end{example}

\begin{example}[Equality restriction]
	Consider the restriction $\beta_{*i}-\beta_{*j}=0$. Suppose that the value of $\beta_{*i}=\beta_{*j}$ is $\theta_{*k}$, the $k$th entry of $\theta_*$. Then this restriction may be encoded by setting both the $i$th and $j$th rows of $R$ to $(0,\ldots, 0, 1, 0,\ldots, 0)\in\mathbb{R}^m$, where the $k$th entry is one. Alternatively, we may set a row of $\mathcal{C}$ to the $1\times N$ vector $c(i,j)$ whose $\ell$th element is defined as $[c(i,j)]_\ell=1(\ell=i)-1(\ell=j)$, where $1(\cdot)$ is the indicator function.
\end{example}

Define $n\times q$ matrices $Y=(Y_1,\ldots, Y_n)^\T$ and $E=(\eta_1,\ldots, \eta_n)^\T$, and the $n\times d$ matrix $X=(X_1,\ldots, X_n)^\T$. Then \eqref{eq:model_regression} has the matrix form $Y=XA_*^\T +E$.  Let $y=\text{vec}(Y)$,  $\eta=\text{vec}(E)$, $Z=(I_q \otimes X)R$, and $\widetilde{y}=y-(I_q \otimes X)\gamma$.
By vectorization and reparameterization, we have
\[
\widetilde{y}=(I_q \otimes X)(\beta_*-\gamma)+\eta=Z\theta_*+\eta.
\]

As a result, the ordinary least squares estimator of $\beta_*$ for the linearly restricted model is
\begin{equation}\label{eq:ols_theta}
\widehat{\beta}=R\widehat{\theta}+\gamma, \quad \widehat{\theta}=\argmin_{\theta\in\mathbb{R}^m}\lVert \widetilde{y}- Z \theta\rVert^2.
\end{equation}
Notice that $Z\in\mathbb{R}^{q n\times m}$. To ensure the feasibility of \eqref{eq:ols_theta}, we need $q n \geq m$.  Let $R=(R_1^\T, \ldots, R_q^\T)^\T$ and $\gamma=(\gamma_1^\T, \ldots, \gamma_q^\T)^\T$, where $R_i$ are $d\times m$ matrices and $\gamma_i$ are $d\times 1$ vectors. Then, $A_*=(I_q \otimes \theta_*^\T)\widetilde{R}+G$, where
\begin{equation*}
\widetilde{R}=(R_1,\ldots, R_q)^\T \in\mathbb{R}^{m q\times d}, \quad G=(\gamma_1,\ldots, \gamma_q)^\T \in\mathbb{R}^{q\times d}.
\end{equation*}
Consequently, the ordinary least squares estimator of $A$ is
$\widehat{A}=(I_q \otimes \widehat{\theta}^\T)\widetilde{R}+G$.

\subsection{General upper bounds analysis \label{subsec:sr_ub}}
To derive  upper estimation error bounds for the stochastic regression model in $\S$ \ref{subsec:formulation}, we begin by introducing a key technical ingredient, namely the block martingale small ball condition \citep{Simchowitz2018}. As a generalization of Mendelson's \citeyearpar{Mendelson2014} small-ball method to time-dependent data, this condition can be viewed as a non-asymptotic stability assumption for controlling the lower tail behavior of the Gram matrix $X^\T X$, or $Z^\T Z$ in our context. 

\begin{definition}[Block Martingale Small Ball Condition]\label{def:BMSB} (i) 	For a real-valued time series $\{X_t, t=1,2,\ldots\}$ adapted to the filtration $\{\mathcal{F}_t\}$, we say that $\{X_t\}$ satisfies the $(k,\nu,\alpha)$-\BMSB\ condition if there exist an integer $k\geq1$ and constants $\nu>0$ and $\alpha\in(0,1)$ such that, for every integer $s\geq 0$,	$k^{-1}\sum_{t=1}^{k} \P(|X_{s+t}|\geq \nu \mid \mathcal{F}_s) \geq \alpha$ with probability one. (ii) For a time series $\{X_t, t=1,2,\ldots\}$ taking values in $\mathbb{R}^d$, we say that  $\{X_t\}$ satisfies the $(k,\Gamma_{\smb},\alpha)$-\BMSB\ condition if there exists  $0\prec\Gamma_{\smb}\in\mathbb{R}^{d\times d}$ such that,  for every $\omega\in\mathcal{S}^{d-1}$, the real-valued time series $\{\omega^\T X_t, t=1,2,\ldots\}$ satisfies the $\{k,  (w^\T\Gamma_{\smb}w)^{1/2}, \alpha\}$-\BMSB\ condition.	
\end{definition}

The value of the probability $\alpha$ is unimportant for our purpose as long as it exists. The thresholding matrix $\Gamma_{\smb}$, or $\nu$ in the univariate case,   captures  the  average cumulative excitability over any size-$k$ block; e.g., if $\{X_t\}$ is a mean-zero vector autoregressive process, then $\Gamma_{\smb}$ will scale proportionally no less than $k^{-1}E(\sum_{t=1}^{k}X_{s+t}X_{s+t}^\T \mid \mathcal{F}_s)$, which is constant for all $s$. Since every time point is associated with a new shock to the  process,  $E(X_{s+t} X_{s+t}^\T \mid \mathcal{F}_s)$  will increase as $t$ increases. Consequently, $\Gamma_{\smb}$  will be monotonic increasing in $k$; see Lemma \ref{lem_bmsb} in $\S$ \ref{subsec:verify}. 	
Moreover, by aggregating all the size-$k$ blocks, $\Gamma_{\smb}$  will essentially become the lower bound of $n^{-1}\sum_{t=1}^{n}X_tX_t^\T$.  Note that for the ordinary least squares estimation, a larger lower bound on the Gram matrix will yield a sharper estimation error bound. Thus, a larger block size $k$ is generally preferred; see Theorem \ref{thm3} in $\S$ \ref{sec:VAR_upperbd} for details.

Let $\Gamma_{\smb}$ and $\overline{\Gamma}$ be  $d\times d$ positive definite matrices, and denote
\begin{equation}\label{notation:Gammaline_R}
\underline{\Gamma}_{R}=R^\T (I_q \otimes \Gamma_{\smb}) R, \quad \overline{\Gamma}_R=R^\T (I_q \otimes \overline{\Gamma}) R.
\end{equation}
In our theoretical analysis, properly rescaled matrices $\underline{\Gamma}_{R}$ and $\overline{\Gamma}_R$ will serve as lower and upper bounds of the Gram matrix $Z^\T Z$, respectively, and the covering numbers derived from them will give rise to the quantity $\log \det(\overline{\Gamma}_R\underline{\Gamma}_{R}^{-1})$ in Theorem \ref{thm1} below.
The regularity conditions  underlying our upper bound analysis are listed as follows:

\begin{assumption}\label{assum:bmsb}
	The covariates process $\{X_t\}_{t=1}^n$ satisfies the $(k,\Gamma_{\smb},\alpha)$-\BMSB\ condition.
\end{assumption}

\begin{assumption}\label{assum:upper_matrix}
	For any $\delta\in(0,1)$, there exists $\overline{\Gamma}_R$ defined as in \eqref{notation:Gammaline_R} such that  $\P(Z^\T Z \npreceq n\overline{\Gamma}_R)\leq \delta$, where $Z=(I_q \otimes X)R$, and  $\overline{\Gamma}_R$ is dependent on $\delta$.
\end{assumption}

\begin{assumption}\label{assum:error_subG}
	For every integer $t\geq 1$, $\eta_t\mid \mathcal{F}_{t}$ is mean-zero and $\sigma^2$-sub-Gaussian.  
\end{assumption}

\begin{theorem}\label{thm1} Let $\{(X_t, Y_t)\}_{t=1}^n$ be generated by the linearly restricted stochastic regression model.
	Fix $\delta\in(0,1)$. Suppose that Assumptions \ref{assum:bmsb}--\ref{assum:error_subG} hold, $0\prec\Gamma_{\smb}\preceq\overline{\Gamma}$, and 
	\begin{equation}\label{thm1_condition}
	n \geq \frac{9k}{\alpha^2} \left \{ m \log\frac{27}{\alpha} + \frac{1}{2}\log \det(\overline{\Gamma}_R\underline{\Gamma}_{R}^{-1})+\log q+\log\frac{1}{\delta}\right \}.
	\end{equation}
	Then, with probability at least $1-3\delta$, we have
	\begin{equation*}
	\lVert \widehat{\beta} - \beta_* \rVert \leq  \frac{9\sigma}{\alpha} \left [\frac{\lambda_{\max}(R \underline{\Gamma}_R^{-1} R^\T)}{n}\left \{12m \log\frac{14}{\alpha}+9\log\det(\overline{\Gamma}_R\underline{\Gamma}_{R}^{-1}) +6\log\frac{1}{\delta} \right \} \right]^{1/2}.
	\end{equation*}
\end{theorem}

Note that $\lVert \widehat{\beta} - \beta_* \rVert=\lVert \widehat{A} - A_* \rVert_F$. The result in Theorem \ref{thm1} is new even for the unrestricted stochastic regression, where $R=I_N$. For  vector autoregressive processes, we will specify the matrices $\underline{\Gamma}_R$ and $\overline{\Gamma}_R$ in $\S$ \ref{subsec:verify}, where $\underline{\Gamma}_R$ will depend on the block size $k$ through $\Gamma_{\smb}$. As $k$ increases, $\Gamma_{\smb}$ will become larger and hence the factor $\lambda_{\max}(R \underline{\Gamma}_R^{-1} R^\T)$ will become smaller, resulting in a sharper error bound.
However, $k$ cannot be too large due to condition \eqref{thm1_condition}. Specifically, this condition arises from applying  the Chernoff bound technique to lower bound the Gram matrix via aggregation of all the size-$k$ blocks, since the probability guarantee of the Chernoff bound will degrade as the number of blocks decreases; see Lemma \ref{lem_A2} in the supplementary material. 	
Therefore, to apply Theorem \ref{thm1} to vector autoregressive processes, a crucial step  will be to derive a feasible region for $k$ that guarantees condition \eqref{thm1_condition}; see $\S$ \ref{subsec:k_condition} for details. 

Similarly,  we can obtain an analogous upper bound for $\widehat{A}-A_*$ in the spectral norm:

\begin{proposition}\label{prop1} Let $\{(X_t, Y_t)\}_{t=1}^n$ be generated by the linearly restricted  stochastic regression model. Fix $\delta\in(0,1)$. Then,  under the conditions of Theorem \ref{thm1}, with probability at least $1-3\delta$, we have
	\begin{equation*}
	\lVert \widehat{A} - A_* \rVert_2 \leq \frac{9\sigma}{\alpha} \left [\frac{\lambda_{\max}\left (\sum_{i=1}^{q}R_i \underline{\Gamma}_R^{-1} R_i^\T \right)}{n}\left \{12m \log\frac{14}{\alpha}+9\log\det(\overline{\Gamma}_R\underline{\Gamma}_{R}^{-1}) +6\log\frac{1}{\delta} \right \} \right]^{1/2}.
	\end{equation*}
\end{proposition}

\section{Linearly restricted vector autoregression\label{sec:VAR_upperbd}}
\subsection{Representative examples\label{subsec:var_example}}
We begin by illustrating how  the formulation in $\S$ \ref{sec:Reg_upperbd} can be used to study vector autoregressive models. Four representative examples will be discussed: \VAR($p$) model, banded vector autoregressive model, network vector autoregressive model, and pure unit root process. 

Consider  the \VAR($1$) model, i.e., model \eqref{eq:model_regression} with $Y_t=X_{t+1}\in\mathbb{R}^d$:
\begin{equation}\label{var1_true}
X_{t+1}=A_*X_t+ \eta_t, 
\end{equation}
subject to $\beta_* = R\theta_*+\gamma$, where  $\beta_*=\text{vec}(A_*^\T)\in\mathbb{R}^{d^2}$, $R=(R_1^\T, \ldots, R_d^\T)^\T\in\mathbb{R}^{d^2\times m}$ with $R_i$  being $d\times m$ matrices, $\theta_*\in\mathbb{R}^m$,  and $\gamma=(\gamma_1^\T, \ldots, \gamma_d^\T)^\T\in\mathbb{R}^{d^2}$ with $\gamma_i\in\mathbb{R}^d$. 

\setcounter{example}{0}
\begin{example}[\VAR($p$) model\label{ex:VARp}] Interestingly, vector autoregressive models of order $p<\infty$ can be viewed as linearly restricted \VAR($1$) models. Consider the \VAR($p$) model
	\begin{equation}\label{model:var_p}
	Z_{t+1}=A_{*1}Z_t+A_{*2}Z_{t-1}+\cdots+A_{*p}Z_{t-p+1}+\varepsilon_t,
	\end{equation}
	where $Z_t, \varepsilon_t \in\mathbb{R}^{d_0}$, and $A_{*i}\in\mathbb{R}^{d_0\times d_0}$ for $i=1,\ldots, p$. Denote $d=d_0 p$, $X_t=(Z_t^\T, Z_{t-1}^\T, \ldots, Z_{t-p+1}^\T)^\T \in\mathbb{R}^d$, $\eta_t=(\varepsilon_t^\T, 0, \ldots, 0)^\T\in\mathbb{R}^d$, and 
	\begin{equation}\label{eq:var_p_A}
	A_*=\left(\begin{array}{cccc}
	A_{*1} & \cdots & A_{*p-1} & A_{*p}\\ 
	I_{d_0}& \cdots & 0        & 0     \\
	\vdots & \ddots & \vdots   & \vdots\\
	0      & \cdots & I_{d_0}  & 0     
	\end{array}\right) \in\mathbb{R}^{d\times d}.
	\end{equation}
	As a result, \eqref{model:var_p} can be written  into the \VAR($1$) form in \eqref{var1_true}. As shown in  \eqref{eq:var_p_A},  all entries in the last $d-d_0$ rows of  $A_*$ are restricted to either zero or one. The restriction $\beta_{*i}=1$ can be encoded in $(\mathcal{C},\mu)$ in the same way as Example 1 in $\S$ \ref{subsec:formulation} but with the $i$th entry of $\mu$ set to one. Thus,  with or without restrictions, the \VAR($p$) model can be studied by the same method as that for a linearly restricted \VAR(1) model. Note that the special structure of the innovation $\eta_t$ that some entries of $\eta_t$ are fixed at zero will not pose any extra difficulty.
\end{example}

In the following  examples, we consider \VAR(1) models with various structures for $A_*=(a_{*ij})_{d\times d}$, and set $\gamma=0$ so that the restrictions are in the form of $R\theta=0$:

\begin{example}[Banded vector autoregression\label{ex:bandedAR}] \cite{Guo_Wang_Yao2016} proposed the vector autoregressive model with the following zero restrictions:
	\begin{equation}\label{eq:banded}
	a_{*ij}=0,\hspace{5mm} \vert i-j\vert > k_0,
	\end{equation}
	where the integer  $1\leq k_0 \leq \lfloor (d-1)/2 \rfloor$ is called the bandwidth parameter. Let $b_{*i}\in\mathbb{R}^d$ be the transpose of the $i$th row of $A_*$. Hence, $\beta_*=(b_{*1}^\T,\ldots, b_{*d}^\T)^\T$. Note that the restrictions  are imposed on each $b_{*i}$ separately.
	As a result, the $b_{*i}$'s are determined by non-overlapping subsets of entries in $\theta_*$; that is, we can write $b_{*i}=R_{(i)}\vartheta_{*i}$, where $R_{(i)}\in\mathbb{R}^{d\times m_i}$, $\vartheta_{*i}\in\mathbb{R}^{m_i}$, $\theta_*=(\vartheta_{*1}^\T, \ldots, \vartheta_{*d}^\T)^\T\in\mathbb{R}^m$, and $m=\sum_{i=1}^{d}m_i$. In this case, $R$ is a block diagonal matrix:
	\begin{equation*}
	R=\left(\begin{array}{*{3}c}
	R_{(1)} &  			& 0\\
	& \ddots 	&  \\
	0		& 			& R_{(d)}
	\end{array} \right) \in \mathbb{R}^{d^2\times m},
	\end{equation*}
	and \eqref{eq:banded} can be encoded in $R$ as follows: (1) $m_i=k_0+i$ and $R_{(i)}=(I_{m_i},0)^\T$ if $1\leq i \leq k_0+1$; (2) $m_i=2k_0+1$ and $R_{(i)}=(0_{m_i\times (i-k_0-1)}, I_{m_i}, 0_{m_i\times (d-i-k_0)})^\T$ if $k_0+1<i<d-k_0$; and (3) $m_i=k_0+1+d-i$ and $R_{(i)}=(0,I_{m_i})^\T$ if $d-k_0\leq i \leq d$.
\end{example}

\begin{example}[Network vector autoregression\label{ex:networkAR}] Consider the network model in \cite{Zhu_Pan_Li_Liu_Wang2017}. Let us drop the individual effect and the intercept to ease the notations. This model assumes that all diagonal entries of $A_*$ are equal: $a_{*ii}=\theta_{*1}$ for $1\leq i\leq d$. For the off-diagonal entries,  the zero-nonzero pattern is known and completely determined by the social network: $a_{*ij}\neq 0$ if and only if individual $i$ follows individual $j$. Moreover, all nonzero off-diagonal entries are assumed to be equal: $a_{*ij}=\theta_{*2}$ if $a_{*ij}\neq 0$, for $1\leq i\neq j \leq d$. This model is actually very parsimonious, with only $m=2$, while the network size $d$ can be extremely large. To incorporate the above restrictions, we may define the $d^2\times 2$ matrix $R$ as follows: for  $i=1,\ldots, d^2$, the $i$th row of $R$ is  $(1,0)$ if $\beta_{*i}$ corresponds to a diagonal entry of $A_*$,  $(0,1)$ if $\beta_{*i}$ corresponds to a nonzero off-diagonal entry of $A_*$, and $(0,0)$ if $\beta_{*i}$ corresponds to a zero off-diagonal entry of $A_*$. 
\end{example}

\begin{example}[Pure unit root process\label{ex:pure_unit_root}] Another simple but important case is $A_*=\rho I_d$ with $\rho\in\mathbb{R}$. Then, the smallest true model has $m=1$,  and the corresponding restrictions, $a_{*11}=\cdots=a_{*dd}$ and $a_{*ij}=0$ for $1\leq  i\neq j \leq d$, can be imposed by setting $R=(e_1^\T, \ldots, e_d^\T)^\T \in\mathbb{R}^{d^2}$, where $e_i$ is the $d\times 1$ vector with all elements zero except the $i$th being one. When $\rho=1$, the underlying model becomes the pure unit root process, a classic example of unstable vector autoregressive processes \citep{Hamilton1994}. In particular, the problem of testing $A_*=I_d$, or unit root testing in panel data, has been extensively studied in the asymptotic literature; see \cite{Chang2004} and \cite{Zhang_Pan_Gao2018} for studies in low and high dimensions, respectively. Note that \cite{Zhang_Pan_Gao2018} focused on asymptotic distributions of the largest eigenvalues of the sample covariance matrix of the pure unit root process under  $\lim_{n,d\rightarrow\infty}d/n=0$. It  does not involve parameter estimation and hence cannot be directly compared to this paper.	
\end{example}

The stochastic regression can also incorporate possibly time-dependent exogenous inputs such as individual effects \citep{Zhu_Pan_Li_Liu_Wang2017} and observable factors \citep{Zhou_Bose_Fan_Liu2018}, leading to the class of \VARX\ models \citep[see, e.g.,][]{Wilms_Basu_Bien_Matteson2017}. Since \VARX\ models can be analyzed similarly to vector autoregressive models, we do not pursue the details in this paper.

\subsection{Verification of Assumptions  \ref{assum:bmsb}--\ref{assum:error_subG} in Theorem \ref{thm1}\label{subsec:verify}}

In light of the generalisability to \VAR($p$) models via the \VAR($1$) representation,  to apply the general results in \S\ 2.2 to linearly restricted vector autoregressive models, it suffices to restrict our attention to the \VAR($1$) model in \eqref{var1_true} from now on. 

Following the notations  in $\S$ \ref{sec:Reg_upperbd}, let  $Y=(X_2,\ldots, X_{n+1})^\T$, $Z=(I_d\otimes X)R$, and $A_*=(I_d\otimes \theta_*^\T)\widetilde{R}+G$, where $\widetilde{R}=(R_1,\ldots, R_d)^\T$ and $G=(\gamma_1,\ldots, \gamma_d)^\T$. Note that $q=d$ for the \VAR($1$) model.  In addition, $\{X_t\}$ is adapted to the filtration  
\[\mathcal{F}_{t}=\sigma\{\eta_1,\ldots,\eta_{t-1}\}.\]
The following conditions on $\{X_t\}$ will be invoked in our analysis:

\begin{assumption}\label{assum:var}
	(i) The process $\{X_t\}$ starts at $t=0$ with $X_0=0$; (ii) the innovations $\{\eta_t\}$ are independent and identically distributed with $E(\eta_t)=0$ and $\var(\eta_t)=\Sigma_\eta=\sigma^2 I_d$;  (iii) there is a universal constant $C_0>0$ such that, for every $\nu\in\mathbb{R}^{d}$ with $\nu^\T \Sigma_\eta \nu\neq0$, the density of $\nu^\T \eta_t/(\nu^\T \Sigma_\eta \nu)^{1/2}$ is bounded from above by $C_0$ almost everywhere; and (iv) $\{\eta_t\}$ are $\sigma^2$-sub-Gaussian. 
\end{assumption}

Under Assumption \ref{assum:var}(i),  we can simply write $X_t$ in the finite-order moving average form, $X_t=\sum_{s=0}^{t-1}A_*^{s}\eta_{t-s-1}$ for any $t\geq1$. Then, by Assumption \ref{assum:var}(ii), $\var(X_t)=\sigma^2\Gamma_t$, where 
\begin{equation}\label{eq:Gramian}
\Gamma_t=\sum_{s=0}^{t-1} A_*^s (A_*^\T)^s
\end{equation}
is called the finite-time controllability Gramian \citep{Simchowitz2018}. Note that $\var(X_t)<\infty$ for whatever $A_*$.  By contrast, the typical setup in  asymptotic theory of stable processes assumes that  $\{X_t\}$ starts at $t=-\infty$. In this case, $\{X_t\}_{t\in\mathbb{Z}}$ has the infinite-order moving average form $X_t=\sum_{s=0}^{\infty}A_*^s\eta_{t-s-1}$, so $\var(X_t)=\sigma^2\sum_{s=0}^{\infty} A_*^s (A_*^\T)^s=\sigma^2\lim_{t\rightarrow\infty} \Gamma_{t}<\infty$
if and only if $\rho(A_*)<1$.   Thus, an important benefit of Assumption \ref{assum:var}(i) is that it allows us to  capture the possibly explosive behavior of $\var(X_t)$ over any finite time horizon, and derive upper bounds of the Gram matrix over different stability regimes; see Lemmas \ref{lem_gram1} and \ref{lem_gram} in this subsection.  

The condition $\Sigma_\eta=\sigma^2 I_d$ in Assumption \ref{assum:var}(ii) is imposed for simplicity. However, we can easily extend all the proofs in this paper to the general case with any symmetric matrix $\Sigma_\eta\succeq0$: we only need to rederive all results  with the role of $\Gamma_t$ replaced by $\sum_{s=0}^{t-1} A_*^s \Sigma_\eta (A_*^\T)^s$.
Assumption \ref{assum:var}(iii) is used to establish the block martingale small ball condition, i.e., Assumption \ref{assum:bmsb}, for vector autoregressive processes, and it allows us to  lower bound the small ball probability by leveraging Theorem 1.2 in \cite{Rudelson_Vershynin2015}   on densities of sums of independent random variables; see  also Remark  \ref{remark:normal}. Essentially,  Assumption \ref{assum:var}(iii) only requires that the distribution of any one-dimensional projection of the innovation is well spread on the real line.   Examples of such distributions include multivariate normal and  multivariate $t$ \citep{Kotz_Nadarajah2004} distribution and, more generally, elliptical distributions \citep{Fang_Kotz_Ng1990} with the consistency property in \cite{Kano1994}.  Lastly, it is clear that Assumptions \ref{assum:var}(ii) and (iv) guarantee Assumption \ref{assum:error_subG}.

\begin{remark}
	In asymptotic theory, stable and unstable processes require substantially different techniques, and results derived under $\rho(A_*)<1$ typically cannot be carried over to  unstable processes. For example, the convergence rate of the ordinary least squares estimator for fixed-dimensional unstable \VAR(1) processes  is $n$ instead of $n^{1/2}$, and the limiting distribution is no longer  normal \citep{Hamilton1994}. 
\end{remark}

\begin{remark}
	The controllability Gramian $\Gamma_t$ is interpretable even without Assumption \ref{assum:var}(i). By recursion, we have $X_{s+t}=\sum_{\ell=0}^{t-1}A_*^{\ell}\eta_{s+t-\ell-1}+A_*^tX_{s}$ for any time point $s$ and  duration $t\geq1$. As a result, $\var(X_{s+t}\mid \mathcal{F}_{s})=\sum_{\ell=0}^{t-1} A_*^\ell \Sigma_\eta (A_*^\T)^\ell$, and it simply becomes $\sigma^2\Gamma_t$ if $\Sigma_\eta=\sigma^2 I_d$.  Note that $\var(X_{s+t}\mid \mathcal{F}_{s})$, or equivalently $\Gamma_t$, is a partial sum of a geometric sequence due to the autoregressive structure. Roughly speaking, larger $A_*$ means more persistent impact of $\eta_t$.  
\end{remark}

In the following, we present three lemmas for the linearly restricted vector autoregressive model. Lemma \ref{lem_bmsb} establishes the block martingale small ball condition by specifying  $\Gamma_{\smb}$, while  Lemmas \ref{lem_gram1} and \ref{lem_gram}  verify Assumption \ref{assum:upper_matrix} by providing two possible specifications of $\overline{\Gamma}$. 

Some additional notations to be used in Lemma \ref{lem_gram} are introduced as follows: Denote by $\Sigma_X$  the covariance matrix of the  $d n\times 1$ vector $\vect(X^\T)=(X_1^\T, \dots, X_n^\T)^\T$, so the $(t,s)$th $d\times d$ block of $\Sigma_X$ is $E(X_tX_s^\T)$, for $1\leq t,s\leq n$. Then define
\begin{equation}\label{xi}
\xi=\xi(m,d,n,\delta) =2 \left \{\frac{\lambda_{\max}(\Gamma_n)  \psi(m,d,\delta) \lVert \Sigma_X \rVert_2}{\sigma^2 n} \right \}^{1/2} +  \frac{2 \psi(m,d,\delta)\lVert \Sigma_X \rVert_2}{\sigma^2 n},
\end{equation}
where $\psi(m,d,\delta)=C_1 \{m\log 9+\log d+\log(2/\delta)\}$, and $C_1>0$ is a universal constant.

\begin{lemma}\label{lem_bmsb} 
	Suppose that $\{X_t\}_{t=1}^{n+1}$ follows $X_{t+1}=A_*X_t+\eta_t$ for $t=0,1,\dots, n$. Under Assumptions \ref{assum:var}(ii) and (iii), for any $1\leq k \leq \lfloor n/2 \rfloor$, $\{X_t\}_{t=1}^n$ satisfies the $(2k,\Gamma_{\smb},1/10)$-\BMSB\ condition, where $\Gamma_{\smb}=\sigma^2\Gamma_{k}/ (4C_0)^2$. 
\end{lemma}

\begin{lemma}\label{lem_gram1}
	Let $\{X_t\}_{t=1}^{n+1}$ be generated by the linearly restricted vector autoregressive model.  Under Assumptions \ref{assum:var}(i) and (ii), for any $\delta\in(0,1)$, it holds  $\P(Z^\T Z \npreceq n\overline{\Gamma}_R)\leq \delta$, where  $\overline{\Gamma}_R=R^\T (I_d\otimes \overline{\Gamma}) R$, with $\overline{\Gamma}=\sigma^2 m\Gamma_{n}/\delta$.
\end{lemma}

\begin{lemma}\label{lem_gram}
	Let $\{X_t\}_{t=1}^{n+1}$ be generated by the linearly restricted vector autoregressive model. Under Assumptions \ref{assum:var}(i) and (ii), if $\{\eta_t\}$ are normal, then for any $\delta\in(0,1)$, it holds  $\P(Z^\T Z \npreceq n\overline{\Gamma}_R)\leq \delta$, where  $\overline{\Gamma}_R=R^\T (I_d\otimes \overline{\Gamma}) R$, with $\overline{\Gamma}=\sigma^2 \Gamma_{n}+ \sigma^2 \xi I_d$, and $\xi=\xi(m,d,n,\delta)$ is defined in \eqref{xi}.
\end{lemma}	

\begin{remark}\label{remark:normal}
	Unlike \cite{Simchowitz2018}, by leveraging \cite{Rudelson_Vershynin2015}, we establish the  block martingale small ball condition without the normality assumption. If Assumption \ref{assum:var}(ii) is relaxed to the general $\var(\eta_t)=\Sigma_\eta\succeq 0$, by a straightforward extension of the proof of Lemma \ref{lem_bmsb}, we can show that Lemma \ref{lem_bmsb} holds with $\Gamma_{\smb}=\sum_{\ell=0}^{k-1} A_*^\ell \Sigma_\eta (A_*^\T)^\ell/ (4C_0)^2$.  
\end{remark}

\begin{remark}\label{remark:hansonw}
	Lemma \ref{lem_gram1} is  a simple consequence of the Markov inequality and the property that $\lambda_{\max}(\cdot)\leq \tr(\cdot)$, so no distributional assumption on $\eta_t$ is required. However,  Lemma \ref{lem_gram} relies on the Hanson-Wright inequality \citep{Vershynin2018}, where the normality assumption is invoked. Although adopting the $\overline{\Gamma}$ in  Lemma \ref{lem_gram}  can eliminate a factor of $\log m$ in the resulting estimation error bounds, the $\overline{\Gamma}$ in Lemma \ref{lem_gram1} actually leads to sharper bounds under certain conditions on $A_*$; see $\S$ \ref{subsec:k_condition} and $\S$ \ref{subsec:var_ub} for details. 
\end{remark}

By Lemma \ref{lem_bmsb}, for any $1\leq k \leq \lfloor n/2 \rfloor$, the matrix  $\underline{\Gamma}_{R}$ in Theorem \ref{thm1} can be specified as
\begin{equation}\label{Gamma_under_R}
\underline{\Gamma}_{R}= \sigma^2 R^\T (I_d\otimes \Gamma_k) R/(4C_0)^2.
\end{equation}
By Lemmas \ref{lem_gram1} and \ref{lem_gram}, the matrix $\overline{\Gamma}_R$ in Theorem \ref{thm1} can be chosen as
$\overline{\Gamma}_R=\overline{\Gamma}_{R}^{(1)}$ or $\overline{\Gamma}_{R}^{(2)}$, where
\begin{equation}\label{Gamma_over_R}
\overline{\Gamma}_{R}^{(1)}=\sigma^2 m R^\T (I_d\otimes \Gamma_n) R/\delta, \quad 
\overline{\Gamma}_{R}^{(2)}=\sigma^2 R^\T (I_d\otimes \Gamma_n) R + \sigma^2 \xi(m,d,n,\delta) R^\T R;
\end{equation}
recall the definitions of $\underline{\Gamma}_{R}$ and $\overline{\Gamma}_R$  in \eqref{notation:Gammaline_R}, where $q=d$ for the \VAR($1$) model.
Furthermore, observe that $\underline{\Gamma}_{R}$ in \eqref{Gamma_under_R} and  the two $\overline{\Gamma}_R$'s in \eqref{Gamma_over_R}, which serve as lower and upper bounds of $Z^\T Z$, respectively, are all related to the controllability Gramian $\Gamma_{t}$, and $0\prec I_d \preceq\Gamma_{k}\preceq\Gamma_{n}$. 

\subsection{Feasible region of $k$\label{subsec:k_condition}}

With $\underline{\Gamma}_{R}$ and $\overline{\Gamma}_R$ chosen as in \eqref{Gamma_under_R} and \eqref{Gamma_over_R}, the term $\log \det(\overline{\Gamma}_R\underline{\Gamma}_{R}^{-1})$ in condition \eqref{thm1_condition} in Theorem \ref{thm1} is intricately  dependent on both $k$ and $n$. Thus, in order to apply the theorem to model \eqref{var1_true}, we need to verify the existence of the block size $k$ satisfying  \eqref{thm1_condition}. This boils down to deriving an explicit upper bound of $\log \det(\overline{\Gamma}_R\underline{\Gamma}_{R}^{-1})$ free of $k$. 

By \eqref{eq:Gramian} and \eqref{Gamma_under_R}, it is easy to show that  $\det(\underline{\Gamma}_{R})$ is monotonic increasing in $k$. Then, as $\overline{\Gamma}_R$ is free of $k$, $\log \det(\overline{\Gamma}_R\underline{\Gamma}_{R}^{-1})$ is maximized when $k=1$. As a result, we can first upper bound $\log \det(\overline{\Gamma}_R\underline{\Gamma}_{R}^{-1})$ by its value at $k=1$ to get rid of its dependence on $k$. That is,
\begin{equation}\label{logdet1} 
\log \det(\overline{\Gamma}_R\underline{\Gamma}_{R}^{-1}) \leq \log \det\{\overline{\Gamma}_R (\sigma^2 R^\T R)^{-1} (4C_0)^2\}.
\end{equation}

Now it suffices to upper bound the right-hand side of \eqref{logdet1}, where $\overline{\Gamma}_R$ can be chosen from $\overline{\Gamma}_{R}^{(1)}$ and $\overline{\Gamma}_{R}^{(2)}$ in \eqref{Gamma_over_R}.
As shown in the supplementary material,
\begin{equation}\label{logdet} 
\log \det\{\overline{\Gamma}_R (\sigma^2 R^\T R)^{-1} (4C_0)^2\} \lesssim \begin{cases}
m\log(m/\delta)+\kappa, & \text{if}\quad \overline{\Gamma}_R=\overline{\Gamma}_{R}^{(1)}\\
m\log \{2\max(1,\xi)\}+\kappa, & \text{if}\quad \overline{\Gamma}_R=\overline{\Gamma}_{R}^{(2)}
\end{cases},
\end{equation}
where $\xi=\xi(m,d,n,\delta)$ is defined as in \eqref{xi}, and
\begin{equation}\label{kappa}
\kappa =\log \det\left \{R^\T (I_d\otimes \Gamma_n) R (R^\T R)^{-1} \right \}.
\end{equation}

Obviously, without imposing normality, we can only choose $\overline{\Gamma}_{R}=\overline{\Gamma}_{R}^{(1)}$;  see Lemmas \ref{lem_gram1} and \ref{lem_gram} in $\S$ \ref{subsec:verify}. However, if $\{\eta_t\}$ are normal,  $\overline{\Gamma}_R$ can be set to whichever of $\overline{\Gamma}_{R}^{(1)}$ and $\overline{\Gamma}_{R}^{(2)}$  delivers the sharper upper bound. Note that both $\kappa$ and $\xi$ depend on $n$, and their growth rates with respect to $n$ depend on the magnitude of $A_*$.  This will ultimately affect the choice between $\overline{\Gamma}_{R}^{(1)}$ and $\overline{\Gamma}_{R}^{(2)}$; e.g., if $\kappa$ is the dominating term in both upper bounds in \eqref{logdet}, we will be indifferent between the two.  Assumptions \ref{assum:explosive}--\ref{assum:stable2} below summarize the three cases of $A_*$ we consider:

\begin{assumption}\label{assum:explosive}
	It holds $\rho(A_*)\leq 1+c/n$, where $c>0$ is a universal constant.
\end{assumption}

\begin{assumption}\label{assum:stable1}
	It holds $\rho(A_*)\leq \bar{\rho} <1$ and $\lVert A_*\rVert_2 \leq C$, where $\bar{\rho}, C >0$ are universal constants.
\end{assumption}

\renewcommand{\theassumption}{\arabic{assumption}$^\prime$}
\setcounter{assumption}{5}
\begin{assumption}\label{assum:stable2}
	It holds $\rho(A_*)\leq \bar{\rho} <1$, $\mu_{\min}(\mathcal{A})=\inf_{\lVert z \rVert=1} \lambda_{\min}(\mathcal{A}^*(z)\mathcal{A}(z))\geq \mu_1$, and  $\lVert A_*^t\rVert_2 \leq C \varrho^t$ for any integer $1\leq t\leq n$, where  $\bar{\rho}, \mu_1, C >0$ and $\varrho\in(0,1)$ are universal constants, and $\mathcal{A}(z)=I_d-A_*z$ for any complex number $z$.
\end{assumption}

Assumption  \ref{assum:explosive} is the most general case among the above three, and Assumption \ref{assum:stable1} is weaker than Assumption \ref{assum:stable2}.  Notably, Assumption \ref{assum:stable2} does not require $\lVert A_*\rVert_2 < 1$ because $C$ may be greater than one. \cite{Guo_Wang_Yao2016} assumed  $\lVert A_*^t\rVert_2 \leq \varrho^t$ for $\varrho\in(0,1)$ and any positive integer $t$, while it is unclear if this can be relaxed to $\lVert A_*^t\rVert_2 \leq C \varrho^t$ as in Assumption \ref{assum:stable2}. We need  $\mu_{\min}(\mathcal{A})$ to be bounded away from zero in order to derive a sharp upper bound on $\lVert \Sigma_X \rVert_2$; see Remark \ref{remark:xi} below.  This condition is also necessary for the estimation error rates derived in \cite{Basu_Michailidis2015} for a similar reason. In particular, if $A_*$ is diagonalizable, it is shown in Proposition 2.2 therein that $\mu_{\min}(\mathcal{A})\geq \left[ \{1-\rho(A_*)\}/\cond(S) \right]^2$, where $S$ is defined in \eqref{eq:jordan} below.

\begin{remark}\label{remark:kappa}
	From \eqref{kappa}, $\kappa$ is dependent on $n$ through $\Gamma_n$.  If $\rho(A_*)<1$,  it holds
	$\Gamma_{n}\preceq\Gamma_{\infty}=\lim_{n\rightarrow\infty}\Gamma_{n}<\infty$, and then
	$\kappa \leq \log \det\left \{R^\T (I_d\otimes \Gamma_\infty) R ( R^\T R)^{-1} \right \}$,
	an upper bound free of $n$. By contrast, if $\rho(A_*)\geq 1$, $\Gamma_{\infty}$ no longer exists, and we need to carefully control the growth rate of $\Gamma_{n}$; see Lemma \ref{lem_Gamma_n} in the supplementary material. This is achieved via the Jordan decomposition of $A_*$ in \eqref{eq:jordan} below, and the mildest condition we need is Assumption \ref{assum:explosive}. The upper bound of $\kappa$ under Assumption \ref{assum:explosive} or \ref{assum:stable1} is given by Lemma \ref{lem_kappa} in the supplementary material.
\end{remark}

\begin{remark}\label{remark:xi}
	For $\xi$ defined in \eqref{xi}, it is worth noting that $\lVert \Sigma_X \rVert_2$ depends on $n$, as $\Sigma_X=[E(X_tX_s^\T)]_{1\leq t,s\leq n}$, where $E(X_tX_s^\T)=\sigma^2  A_*^{t-s} \Gamma_s$ for $1\leq s\leq t\leq n$ under Assumptions \ref{assum:var}(i) and (ii). Unlike $\kappa$ discussed in Remark \ref{remark:kappa}, even under Assumption \ref{assum:stable1},  $\lVert \Sigma_X \rVert_2$ is not guaranteed to be bounded by a constant free of $n$; indeed, we need Assumption \ref{assum:stable2} for this purpose, since $\lVert \Sigma_X \rVert_2$ is affected by not only the growing diagonal blocks $\sigma^2\Gamma_1, \dots, \sigma^2\Gamma_n$ but also the growing off-diagonal blocks; see Lemma \ref{lem_Sigma_X} in the supplementary material for details.  The upper bound of $\xi$ under Assumption \ref{assum:explosive} or \ref{assum:stable2} is given by  Lemma \ref{lem_xi} in the supplementary material.
\end{remark}

Let the Jordan decomposition of $A_*$ be 
\begin{equation}\label{eq:jordan}
A_*=SJS^{-1},
\end{equation}
where $J$ has $L$ blocks with sizes $1\leq b_1,\ldots, b_L \leq d$, and both $J$ and $S$ are $d\times d$ complex matrices. Let 
$b_{\max}=\max_{1\leq \ell \leq L} b_\ell,$
and denote the condition number of $S$ by
$\cond(S)=\left \{\lambda_{\max}(S^* S)/\lambda_{\min}(S^*S)\right \}^{1/2}$,
where $S^*$ is the conjugate transpose of $S$.

The following proposition, which follows from \eqref{logdet1}, \eqref{logdet} and upper bounds of $\kappa$ and $\xi$ under Assumptions \ref{assum:explosive}, \ref{assum:stable1} or \ref{assum:stable2}, is proved in $\S$ \ref{sec:S3} of the supplementary material.

\begin{proposition}\label{prop_G}
	For any $A_*\in\mathbb{R}^{d\times d}$, under Assumption \ref{assum:explosive}, we have $$\log \det(\overline{\Gamma}_R\underline{\Gamma}_{R}^{-1}) \lesssim m \left[\log \{d\cond(S)/\delta\}+ b_{\max} \log n\right]$$ for $\overline{\Gamma}_R=\overline{\Gamma}_R^{(1)}$ or $\overline{\Gamma}_R^{(2)}$. Moreover, if Assumption \ref{assum:stable1} holds, then  $\log \det(\overline{\Gamma}_R^{(1)}\underline{\Gamma}_{R}^{-1}) \lesssim m \log (m/\delta)$. Furthermore,  if Assumption \ref{assum:stable2} holds and  $n  \gtrsim m +\log(d/\delta)$, then  $\log \det(\overline{\Gamma}_R^{(2)}\underline{\Gamma}_{R}^{-1}) \lesssim m$. 
\end{proposition}

The condition $n  \gtrsim m +\log(d/\delta)$ in Proposition \ref{prop_G} is not stringent, because it is necessary for condition \eqref{thm1_condition} in Theorem \ref{thm1}, where $q=d$ for the $\VAR(1)$ model. By Proposition \ref{prop_G},  $\overline{\Gamma}_{R}^{(2)}$ yields a sharper upper bound of $\log \det(\overline{\Gamma}_R\underline{\Gamma}_{R}^{-1})$ than does $\overline{\Gamma}_{R}^{(1)}$  only under Assumption \ref{assum:stable2}. Thus, we shall always set $\overline{\Gamma}_{R}=\overline{\Gamma}_{R}^{(1)}$ unless Assumption \ref{assum:stable2} holds and  $\{\eta_t\}$ are normal.
As a result, under Assumption \ref{assum:var}, the feasible region of $k$ that is sufficient for condition \eqref{thm1_condition} in Theorem \ref{thm1} is
\begin{equation}\label{eq:k_sufficient}
k \lesssim \begin{cases}
\frac{n}{m\left [\log\{d\cond(S)/\delta\}+ b_{\max} \log n \right ]}, & \text{if Assumption \ref{assum:explosive} holds}\\
\frac{n}{m\log(m/\delta)+\log d}, & \text{if Assumption \ref{assum:stable1} holds}\\
\frac{n}{m+\log(d/\delta)}, & \text{if Assumption \ref{assum:stable2} holds and $\{\eta_t\}$ are normal}
\end{cases}.
\end{equation}

Because the upper bound of $\log \det(\overline{\Gamma}_R\underline{\Gamma}_{R}^{-1})$ and the feasible region of $k$ are both dependent on the assumption on $A_*$ and whether  $\{\eta_t\}$ are normal, the resulting estimation error bounds will vary slightly  under different conditions; see Theorems \ref{thm2} and \ref{thm3} in $\S$ \ref{subsec:var_ub}.

\subsection{Analysis of upper bounds in vector autoregression\label{subsec:var_ub}}

We focus on the upper bound analysis of $\lVert \widehat{\beta} - \beta_* \rVert$; nevertheless, from Proposition \ref{prop1} we can readily obtain analogous results for $\lVert \widehat{A} - A_* \rVert_2$,  which are omitted here.  For simplicity, denote 
\begin{equation*}
\Gamma_{R,k}=R \left \{R^\T (I_d\otimes \Gamma_k) R\right \}^{-1} R^\T.
\end{equation*}

The first theorem follows directly from Theorem \ref{thm1}, Lemmas \ref{lem_bmsb}--\ref{lem_gram} and Proposition \ref{prop_G}.
\begin{theorem}\label{thm2}
	Let $\{X_t\}_{t=1}^{n+1}$ be generated by the linearly restricted vector autoregressive model. Fix $\delta\in(0,1)$. For any $1\leq k \leq \lfloor n/2 \rfloor$ satisfying \eqref{eq:k_sufficient}, under Assumption \ref{assum:var}, we have the following results: (i) If Assumption \ref{assum:explosive} holds, with probability at least $1-3\delta$,
	\begin{equation*}
	\lVert \widehat{\beta} - \beta_* \rVert \lesssim \left (\lambda_{\max}(\Gamma_{R,k})\frac{ m \left[\log \{d\cond(S)/\delta\}+ b_{\max} \log n\right]}{n} \right )^{1/2}.
	\end{equation*}
	(ii) If Assumption \ref{assum:stable1} holds, with probability at least $1-3\delta$,
	\[
	\lVert \widehat{\beta} - \beta_* \rVert \lesssim \left \{\lambda_{\max}(\Gamma_{R,k})\frac{m \log(m/\delta)}{n}\right \}^{1/2}.
	\]
	(iii) If Assumption \ref{assum:stable2} holds and $\{\eta_t\}$ are normal,  with probability at least $1-3\delta$,
	\[
	\lVert \widehat{\beta} - \beta_* \rVert \lesssim \left \{\lambda_{\max}(\Gamma_{R,k})\frac{m +\log(1/\delta)}{n}\right \}^{1/2}.
	\]		
\end{theorem}

To gain an intuitive understanding of the factor $\lambda_{\max}(\Gamma_{R,k})$ in Theorem \ref{thm2}, consider the asymptotic distribution of $\widehat{\beta}$ under the assumptions that $\rho(A_*)<1$  and that $d, m$ and $A_*$ are all fixed: 
\begin{equation}\label{eq:limit_dist}
n^{1/2}(\widehat{\beta} - \beta_*)\rightarrow N\big(0, \underbrace{R \{R^\T (I_d\otimes \Gamma_\infty) R\}^{-1} R^\T}_{\Gamma_{R,\infty}}\big)
\end{equation}
in distribution as $n\rightarrow\infty$, where $\Gamma_{\infty}=\lim_{k\rightarrow\infty}\Gamma_{k}$; see  \cite{Lutkepohl2005}. Thus, $\lambda_{\max}(\Gamma_{R,k})$ in Theorem \ref{thm2} resembles the limiting covariance matrix $\Gamma_{R,\infty}$. However, it is noteworthy that  by adopting a non-asymptotic approach, Theorem \ref{thm2} retains  the dependence of the estimation error on $\Gamma_{R,k}$ across stable, unstable and slightly explosive regimes.  Also note that, similarly to \eqref{eq:limit_dist}, the error bounds in Theorem \ref{thm2} are free of $\sigma^2$, as the scaling effect of $\sigma^2$ on $\eta_t$ is canceled out by that on $X_t$ due to the autoregressive structure. 

As a special case, $b_{\max}=1$ if $A_*$ is diagonalizable. Moreover, if $A_*=\rho I_d$, then $b_{\max}=1$, $\cond(S)=1$, $\Gamma_k=\gamma_k(\rho) I_d$, with $\gamma_k(\rho)=\sum_{s=0}^{k-1}\rho^{2s}$, and thus
\begin{equation}\label{eq:lambda_max_equal}
\lambda_{\max}(\Gamma_{R,k})=\gamma_k^{-1}(\rho) \lambda_{\max}\{R (R^\T R)^{-1} R^\T\}=\gamma_k^{-1}(\rho) \lambda_{\max}\{(R^\T R)^{-1} R^\T R\}=\gamma_k^{-1}(\rho),
\end{equation}
where the second equality is due to the fact that, for any matrices $A\in\mathbb{R}^{N\times m}$ and $B\in\mathbb{R}^{m\times N}$, $AB$ and $BA$ have the same nonzero eigenvalues \citep[Theorem 1.3.20,][]{Horn_Johnson1985}. 

By Theorem \ref{thm2}, The linear restrictions affect the error bounds through both the factor $\lambda_{\max}(\Gamma_{R,k})$ and the explicit rate function of $m$ and $n$. To further illustrate this, suppose that 
\[\beta_*=R\theta_*+\gamma=R^{(1)}R^{(2)}\theta_*+\gamma,\] 
where $R^{(1)}\in\mathbb{R}^{d^2\times\widetilde{m}}$ has rank $\widetilde{m}$, and $R^{(2)}\in\mathbb{R}^{\widetilde{m}\times m}$ has rank $m$, with $\widetilde{m}\geq m+1$. Then $\mathcal{L}^{(1)}=\{R^{(1)}\theta+\gamma:\theta\in\mathbb{R}^{\widetilde{m}}\}\supseteq\mathcal{L} =\{R\theta+\gamma: \theta\in\mathbb{R}^m\}$. By an argument similar to that in \citet[p.~199]{Lutkepohl2005}, we can show that $\Gamma_{R,k}\preceq \Gamma_{R^{(1)},k}$, so
\begin{equation}\label{eq:lambda_max}
\lambda_{\max}(\Gamma_{R,k}) \leq \lambda_{\max}(\Gamma_{R^{(1)},k}). 
\end{equation}
Note that the parameter space $\mathcal{L}^{(1)}$ has fewer restrictions than $\mathcal{L}$. Therefore, with fewer restrictions,  the effective model size will increase from $m$ to $\widetilde{m}$, and meanwhile $\lambda_{\max}(\Gamma_{R,k})$ will increase to $\lambda_{\max}(\Gamma_{R^{(1)},k})$, both leading to deterioration of the error bound. 

\begin{remark}
	The preservation of the factor $\lambda_{\max}(\Gamma_{R,k})$ in Theorem \ref{thm2} is achieved by bounding $Z^\T Z$ and $Z^\T \eta$ simultaneously through the Moore-Penrose pseudoinverse $Z^\dagger$, and note that  $Z^\dagger=(Z^\T Z)^{-1}Z^\T$ if $Z^\T Z\succ 0$; see
	also \cite{Simchowitz2018}.  This key advantage is not enjoyed by the non-asymptotic analyses in \cite{Basu_Michailidis2015} and \cite{Faradonbeh_Tewari_Michailidis2018}.  In their analyses,  $X^\T X$ and $X^\T E$, or  $Z^\T Z$ and $Z^\T \eta$ in our context, were bounded separately. This would not only break down $\Gamma_{R,k}$, but also cause degradation of the error bound as $\rho(A_*)\rightarrow1$ due to the inevitable involvement of the condition number of $X^\T X$ in the resulting error bound. 
\end{remark}

\begin{remark}
	If $A_*=\rho I_d$, then \eqref{eq:lambda_max} becomes an equality. However, the equality generally does not hold even for diagonal matrices $A_*$. For example, if $A_*=\diag(\rho_1,\rho_2)\in\mathbb{R}^{2\times 2}$, where $\vert\rho_1\vert>\vert\rho_2\vert$, then $\Gamma_k=\diag\{\gamma_k(\rho_1), \gamma_k(\rho_2)\}$. Let $R=(1,0,0,0)^\T=R^{(1)}R^{(2)}$, where $R^{(1)}=(I_2,0)^\T\in\mathbb{R}^{4\times 2}$ and $R^{(2)}=(1,0)^\T$. Consequently, \eqref{eq:lambda_max} is a strict inequality:
	$\lambda_{\max}(\Gamma_{R,k}) = \gamma_k^{-1}(\rho_1) < \gamma_k^{-1}(\rho_2)=\lambda_{\max}(\Gamma_{R^{(1)},k})$.
\end{remark}

The next theorem sharpens the error bounds in Theorem \ref{thm2} by utilizing the largest possible $k$, since $\lambda_{\max}(\Gamma_{R,k})$ is monotonic decreasing in $k$. The dependence of $\lambda_{\max}(\Gamma_{R,k})$ on $A_*$ will be captured by $\sigma_{\min}(A_*)$, a measure of the least excitable mode of the underlying dynamics.

\begin{theorem}\label{thm3}
	Suppose that the conditions of Theorem \ref{thm2} hold. Fix $\delta\in(0,1)$, and let $c_1>0$ be a universal constant.
	
	(i) Under Assumption \ref{assum:explosive}, if
	\begin{equation}\label{eq:sigmin_slow1}
	\sigma_{\min}(A_*)\leq 1-\frac{c_1m\left [\log\{d\cond(S)/\delta\}+ b_{\max} \log n \right ]}{n},\tag{A.1}
	\end{equation}	
	then, with probability at least $1-3\delta$, 
	\begin{equation}\label{eq:slow1}
	\lVert \widehat{\beta} - \beta_* \rVert \lesssim \left (\frac{\{1-\sigma_{\min}^2(A_*)\}  m \left[\log \{d\cond(S)/\delta\}+ b_{\max} \log n\right]  }{n}\right )^{1/2}, \tag{S.1}
	\end{equation}	
	and if inequality \eqref{eq:sigmin_slow1} holds in the reverse direction, then, with probability at least $1-3\delta$, 
	\begin{equation}\label{eq:fast1}
	\lVert \widehat{\beta} - \beta_* \rVert  \lesssim  \frac{m \left[\log \{d\cond(S)/\delta\}+ b_{\max} \log n\right]}{n}. \tag{F.1}
	\end{equation}
	
	(ii) Under Assumption \ref{assum:stable1}, if
	\begin{equation}\label{eq:sigmin_slow2}
	\sigma_{\min}(A_*)\leq 1-\frac{c_1\{m\log(m/\delta)+\log d\}}{n},\tag{A.2}
	\end{equation}
	then, with probability at least $1-3\delta$, 
	\begin{equation}\label{eq:slow2}
	\lVert \widehat{\beta} - \beta_* \rVert \lesssim  \left [\frac{\{1-\sigma_{\min}^2(A_*)\}m \log(m/\delta)}{n}\right ]^{1/2}, \tag{S.2}
	\end{equation}
	and if inequality \eqref{eq:sigmin_slow2} holds in the reverse direction, then, with probability at least $1-3\delta$, 
	\begin{equation}\label{eq:fast2}
	\lVert \widehat{\beta} - \beta_* \rVert \lesssim \frac{m \log(m/\delta)+\log d}{n}. \tag{F.2}
	\end{equation}	
	
	(iii) Under Assumption \ref{assum:stable2}, if $\{\eta_t\}$ are normal and
	\begin{equation}\label{eq:sigmin_slow3}
	\sigma_{\min}(A_*)\leq 1-\frac{c_1\{m+\log(d/\delta)\}}{n},\tag{A.3}
	\end{equation}
	then, with probability at least $1-3\delta$, 
	\begin{equation}\label{eq:slow3}
	\lVert \widehat{\beta} - \beta_* \rVert \lesssim \left [\frac{ \{1-\sigma_{\min}^2(A_*)\} \{m +\log(1/\delta)\}}{n}\right ]^{1/2}. \tag{S.3}
	\end{equation}	
	and if inequality \eqref{eq:sigmin_slow3} holds in the reverse direction, then, with probability at least $1-3\delta$, 
	\begin{equation}\label{eq:fast3}
	\lVert \widehat{\beta} - \beta_* \rVert \lesssim \frac{m +\log(d/\delta)}{n}. \tag{F.3}
	\end{equation}			
\end{theorem}

Theorem \ref{thm3} reveals an interesting phenomenon of phase transition  from slow to fast error rate regimes, i.e., from about $O\{(m/n)^{1/2}\}$ as in \eqref{eq:slow1}--\eqref{eq:slow3} to about $O(m/n)$ as in \eqref{eq:fast1}--\eqref{eq:fast3}, up to logarithmic factors.  Within the slow rate regime, the estimation error decreases as  $\sigma_{\min}(A_*)$ increases. Note that the slow rates in \eqref{eq:slow1}--\eqref{eq:slow3} differ from each other only by logarithmic factors, and so do the fast rates in \eqref{eq:fast1}--\eqref{eq:fast3}.  Moreover, the point at which the transition occurs is dependent on $\sigma_{\min}(A_*)$ instead of $\rho(A_*)$; see conditions \eqref{eq:sigmin_slow1}--\eqref{eq:sigmin_slow3}. Since $\sigma_{\min}(A_*)\leq \rho(A_*)$, conditions \eqref{eq:sigmin_slow1}--\eqref{eq:sigmin_slow3} may be mild as long as $\rho(A_*)$ is not too large. However, the fast rates require the opposite of \eqref{eq:sigmin_slow1}--\eqref{eq:sigmin_slow3},  which cannot be directly inferred from $\rho(A_*)$.

\begin{remark}\label{remark:phases}
	For the special case of $A_*=\rho I_d$ with $\rho\in\mathbb{R}$, Assumptions \ref{assum:stable1} and \ref{assum:stable2} both simply reduce to $\vert\rho\vert<1$, and Assumption \ref{assum:explosive} to $\vert\rho\vert\leq 1+O(1/n)$. Also, $\cond(S)=1$ and $b_{\max}=1$. Thus, under Assumption \ref{assum:var}, by \eqref{eq:fast1}, \eqref{eq:slow2} and \eqref{eq:fast2}, the following holds with high probability:
	\[
	\lVert \widehat{\beta} - \beta_* \rVert \lesssim \begin{cases}
	O[\{(1-\rho^2)m\log m/n\}^{1/2}], & \text{if} \quad \vert\rho\vert\leq 1-O\{(m\log m+\log d)/n\}\\
	O\{(m\log m+\log d)/n\},  & \text{if} \quad 1-O\{(m\log m+\log d)/n\} \leq \vert\rho\vert <1\\
	O\{m\log(d n)/n\}, & \text{if} \quad 1\leq \vert\rho\vert \leq 1+O(1/n)
	\end{cases}.
	\]	
	Moreover, if $\{\eta_t\}$ are normal, by \eqref{eq:slow3} and \eqref{eq:fast3}, we can eliminate all  factors of $\log m$ in the above results; that is, every $m\log m$ will be replaced by $m$.
\end{remark}

\begin{remark}\label{remark:Michailidis}
	\cite{Faradonbeh_Tewari_Michailidis2018} derived error bounds for the ordinary least squares estimator of explosive unrestricted vector autoregressive processes when (1)  $|\lambda_{\min}(A_*)|>1$ or (2)  $A_*$ has no unit eigenvalue. In contrast to case (1), we focus on slightly explosive processes with $\rho(A_*)\leq 1+O(1/n)$.  Moreover, the no unit root requirement of case (2) may be quite restrictive; e.g., it excludes the case of $|\rho|=1$ in Remark \ref{remark:phases}. Thus, the conditions in Theorem \ref{thm3} may be more reasonable in practice.
\end{remark}

\begin{remark}\label{remark:Guo}
	\cite{Guo_Wang_Yao2016} obtained the error rate $O_p\{(m/n)^{1/2}\}$ for the banded vector autoregressive model under weaker conditions on $\{\eta_t\}$ yet stronger conditions on $A_*$ than those in this paper; see Theorem 2 therein. Particularly, they required $\lVert A_*\rVert_2<1$. Note that this rate matches \eqref{eq:slow3}. In view of the lower bounds to be presented in $\S$\ref{sec:lowerbd}, we  conjecture that the rate in \eqref{eq:slow2} is larger than the actual rate by a factor of  $\log m$. 	
\end{remark}

\section{Analysis of lower bounds \label{sec:lowerbd}}

For any $\theta\in\mathbb{R}^m$, let $\beta=R\theta+\gamma$, and the corresponding transition matrix is denoted by $A(\theta)=(I_d\otimes \theta^\T)\widetilde{R}+G$, where $R$, $\gamma$, $\widetilde{R}$ and $G$ are defined as in $\S$ \ref{subsec:var_example}. As $\beta$ is completely determined by $\theta$, it is more convenient to index the probability law of the  model by the unrestricted parameter $\theta$. Thus, we denote by $\P_{\theta}^{(n)}$ the distribution of the sample $(X_1,\ldots, X_{n+1})$ on  $(\mathcal{X}^{n+1}, \mathcal{F}_{n+1})$, where $\mathcal{X}=\mathbb{R}^{d}$ and $\mathcal{F}_{n+1}=\sigma\{\eta_1,\ldots,\eta_{n}\}$.	For any fixed $\bar{\rho}>0$, we write the subspace of $\theta$ such that the spectral radius of $A(\theta)$ is bounded above by $\bar{\rho}$ as
\[
\Theta(\bar{\rho}) = \{\theta\in \mathbb{R}^m:  \rho \{A(\theta)\} \leq \bar{\rho} \}.
\]
The corresponding linearly restricted subspace of $\beta$ is denoted by
$\mathcal{L}(\bar{\rho})=\{R\theta+\gamma: \theta\in\Theta(\bar{\rho})\}$.

The minimax rate of estimation over $\beta\in \mathcal{L}(\bar{\rho})$, or $\theta\in\Theta(\bar{\rho})$, is provided by the next theorem.

\begin{theorem}\label{thm_lower}
	Suppose that $\{X_t\}_{t=1}^{n+1}$ follow the vector autoregressive model $X_{t+1}=AX_t+\eta_t$ with linear restrictions defined as in $\S$ \ref{sec:VAR_upperbd}. In addition, Assumptions \ref{assum:var}(i) and (ii) hold, and  $\{\eta_t\}$ are normally distributed. Fix  $\delta\in(0,1/4)$ and $\bar{\rho}>0$.  Let $\gamma_n(\bar{\rho})=\sum_{s=0}^{n-1}\bar{\rho}^{2s}$. Then, for any $\epsilon\in(0, \bar{\rho}/4]$, we have
	\[
	\inf_{\widehat{\beta}}\sup_{\theta\in \Theta(\bar{\rho})} \P_{\theta}^{(n)} \left \{ \lVert \widehat{\beta}-\beta\rVert \geq \epsilon \right \} \geq \delta,
	\]
	where the infimum is taken over all estimators of $\beta$ subject to $\beta\in\{R\theta+\gamma: \theta\in\mathbb{R}^m\}$,
	for any $n$ such that 
	\begin{equation*}
	n \gamma_n(\bar{\rho}) \lesssim \frac{m+\log(1/\delta)}{\epsilon^2}.
	\end{equation*}
\end{theorem}

As a result, we have the following minimax rates of estimation across different values of $\bar{\rho}$.

\begin{corollary}\label{cor_lower}
	For the linearly restricted vector autoregressive model in Theorem \ref{thm_lower}, the minimax rates of estimation over $\beta\in \mathcal{L}(\bar{\rho})$  are given as follows:
	
	(i) $\{{(1-\bar{\rho}^2)m}/{n}\}^{1/2}$,  if $0<\bar{\rho}< (1-1/n)^{1/2}$; 
	
	(ii) $m^{1/2}/n$, if $(1-1/n)^{1/2} \leq \bar{\rho} \leq 1+c/n$ for a fixed $c>0$; and
	
	(iii) $\bar{\rho}^{-n}\{(\bar{\rho}^2-1)m/n\}^{1/2}$, if $\bar{\rho}>1+c/n$.
\end{corollary}

While the phase transition in Corollary \ref{cor_lower} depends on $\rho(A_*)$ instead of  $\sigma_{\min}(A_*)$, we may still compare the lower bounds to the upper bounds in Theorem \ref{thm3}.  For case (i)  in Corollary \ref{cor_lower}, since $\sigma_{\min}(A_*)\leq \rho(A_*)<(1-1/n)^{1/2}<1$, we may expect that condition \eqref{eq:sigmin_slow2} will  hold in most cases, and hence the upper bound $O\{(m \log m/n)^{1/2}\}$ differs from the lower bound by a factor of $\log m$. However, for case (ii), the upper bound may lie in either the slow or fast rate regime, depending on the magnitude of $\sigma_{\min}(A_*)$, whereas the lower bound lies in the fast rate regime. As shown by our first experiment in $\S$ \ref{sec:simul}, the transition from slow to fast error rates actually depends on $\sigma_{\min}(A_*)$ instead of $\rho(A_*)$. This also suggests that the results in Theorem \ref{thm3} are sharp in the sense that they correctly capture the transition behavior.

\begin{remark}\label{remark:phases_all}
	If $A_*=\rho I_d$ with $\rho\in\mathbb{R}$ and $\{\eta_t\}$ are normal, in view of Remark \ref{remark:phases} and Corollary \ref{cor_lower},  a more straightforward comparison of the upper and lower bounds can be made:
	\begin{center}
		\begin{tabular}{l|@{\hskip 1mm}c@{\hskip 1mm}c}
			\multicolumn{1}{c|@{\hskip 1mm}}{Range of $\lvert \rho \rvert$}& Lower bound                            & Upper bound\\\hline
			$(0,1-O\left \{(m+\log d)/n\right \}]$     & $\Omega[\{{(1-\rho^2)m}/{n}\}^{1/2}]$   & $O[\{(1-\rho^2)m/n\}^{1/2}]$\\
			$[1-O\{(m+\log d)/n\},(1-1/n)^{1/2})$       & $\Omega[\{{(1-\rho^2)m}/{n}\}^{1/2}]$   & $O\{(m+\log d)/n\}$\\
			$[(1-1/n)^{1/2},1)$                               & $\Omega(m^{1/2}/n)$               & $O\{(m+\log d)/n\}$\\
			$[1, 1+O(1/n)]$                                  & $\Omega(m^{1/2}/n)$               & $O\{m\log(dn)/n\}$\\
			$(1+O(1/n), \infty)$                             & $\Omega[\lvert\rho\rvert^{-n}\{(\rho^2-1)m/n\}^{1/2}]$& $-$
		\end{tabular}
	\end{center}
	See Figure \ref{fig5} for an illustration of the theoretical bounds and actual rates suggested by simulation results in $\S$ \ref{sec:simul}. Note that the actual rates and the theoretical upper and lower bounds exactly match when $0<\vert \rho \vert\leq 1-O\{(m+\log d)/n\}$. In addition, the suggested actual rate is  $m/n$ for $1-O\{(m+\log d)/n\}< \vert \rho \vert< 1+O(1/n)$ and  even faster for $\vert \rho \vert$ beyond this range. 
\end{remark}

\begin{figure}[h]
	\centering
	\includegraphics[width=\textwidth]{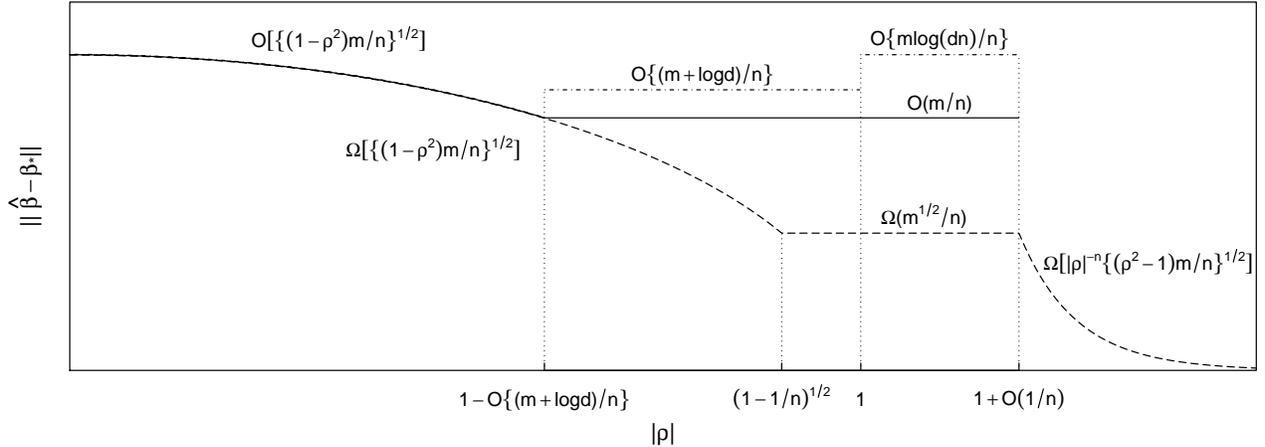}
	\caption{Illustration of theoretical upper (dot-dash) and lower (dashes) bounds and the actual rates (solid) suggested by simulation results in $\S$ \ref{sec:simul} for the $\VAR(1)$ model with $A_*=\rho I_d$ and normal innovations.}
	\label{fig5}
\end{figure}

\begin{remark}\label{remark:Han}
	\cite{Han_Xu_Liu2015} considered the estimation  of a class of copula-based stationary vector autoregressive processes which includes the Gaussian \VAR($1$) process as a special case, and extended the theoretical properties to $\VAR(p)$ processes by arguments similar to those in Example \ref{ex:VARp}. Under a strong-mixing condition on the process and  the low-rank assumption $\rank(A_*)\leq r$, the proposed estimator $\widetilde{A}$ was proved to attain the minimax error rate $\lVert \widetilde{A}-A_* \rVert_F=O\{(dr/n)^{1/2}\}$.  While we consider different restrictions and estimation method, the  rate $\{{(1-\bar{\rho}^2)m}/{n}\}^{1/2}$ in Corollary \ref{cor_lower} for the stable regime  resembles that in \cite{Han_Xu_Liu2015} if we regard $dr$  as the effective model size $m$ of the low-rank $\VAR(1)$ model. However, similarly to our upper analysis in $\S$ \ref{sec:VAR_upperbd}, the factor of $(1-\bar{\rho}^2)^{1/2}$ in our lower bound also reveals that the estimation error may decrease as $A_*$ approaches the stability boundary.
\end{remark}

\section{Simulation experiments\label{sec:simul}}
We conduct three simulation experiments to verify the theoretical results in previous sections, including (i) the estimation error rates, (ii) the transition from slow to fast  rates regimes, and (iii) the effect of the ambient dimension $d$ on the estimation. The data are generated from \VAR$(1)$ processes with $\{\eta_t\}$ drawn independently from $N(0, I_d)$ and the following structures of $A_*$:

\DGP1: Banded structure defined by the  zero restrictions $a_{*ij}=0$ if $\vert i-j\vert > k_0$, where $1\leq k_0 \leq \lfloor (d-1)/2 \rfloor$ is the bandwidth parameter; see Example \ref{ex:bandedAR} in $\S$ \ref{subsec:var_example}. As a result, if all restrictions are imposed, the size of the model is $m=d + (2d - 1)k_0 - k_0^2$.

\DGP2: Group structure defined by equality restrictions as follows. Partition the index set $\mathcal{V}=\{1,\ldots, d\}$ of the coordinates of $X_t$  into $K$ groups of size $b=d/K$ as $\mathcal{V}=\bigcup_{k=1}^K \mathcal{G}_k$, where
\[\mathcal{G}_k=\{(k-1)b+1,\ldots, kb\}\quad (k=1,\ldots, K).\] 
In each row of $A_*$, the off-diagonal entries $a_{*ij}$ with $j$ belonging to the same group are assumed to be equal:  for any $1\leq k\leq K$ and $1\leq i\leq d$, all elements of $\{a_{*ij}, j\in\mathcal{G}_k, j\neq i\}$ are equal. Thus, $m=(K+1)d$, as there are $(K+1)$ free parameters in each row of $A_*$.	

\DGP3: $A_*=\rho I_d$, where $\rho \in\mathbb{R}$. Note that the smallest true model with size $m=1$ results from imposing zero restrictions on all off-diagonal entries of $A_*$ and equality restrictions on all diagonal entries; see Example \ref{ex:pure_unit_root} in $\S$ \ref{subsec:var_example}.  

Throughout the experiments, the $\ell_2$ estimation error $\lVert \widehat{\beta} - \beta_* \rVert$ is calculated by averaging over $1000$ replications. Except for \DGP3, nonzero entries of $A_*$ are generated independently from $U[-1,1]$ and then rescaled such that $\rho(A_*)$ is equal to a certain value. 

The first experiment aims to verify the error rates in Theorem \ref{thm3} and the implication of Theorem \ref{thm2} that the restrictions can reduce the estimation error through both the explicit rate $(m/n)^{1/2}$ and the decrease in the factor  $\lambda_{\max}(\Gamma_{R,k})$. Fixing $d=24$ and $\rho(A_*)=0.2,0.8$ or $1$, we generate data from \DGP1 with $k_0=1$, \DGP2 with $K=2$, and \DGP3. 
For \DGP1 and \DGP3, we fit banded vector autoregressive models with $k_0=1,5$ or $7$ such that $m=70, 156$ or $304$, respectively. For \DGP2, we fit the group-structured model with $K=2,8$ or $12$ such that $m=72, 120$ or $312$, respectively. Note that for \DGP1 and \DGP2,  $\sigma_{\min}(A_*)\leq 0.1$ even when the randomly generated matrix $A_*$ has $\rho(A_*)=1$. However, for \DGP3, it holds  $\sigma_{\min}(A_*)=\rho(A_*)$.  The $\ell_2$ estimation error $\lVert \widehat{\beta} - \beta_* \rVert$ is plotted against $(m/n)^{1/2}$ in Figure \ref{fig1}, where we consider $(m/n)^{1/2}\in\{0.15,0.35,0.55,0.75,0.95\}$.  	Our findings are summarized as follows:

\begin{figure}[t!]
	\centering
	\includegraphics[width=\linewidth]{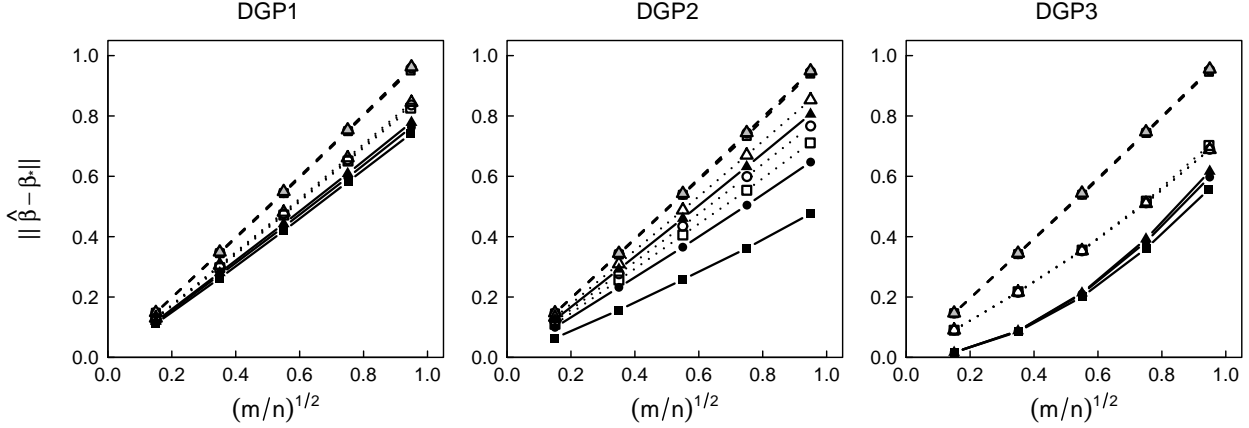}
	\caption{Plots of $\lVert \widehat{\beta} - \beta_* \rVert$ against $(m/n)^{1/2}$ for three data generating processes with  $\rho(A_*)=0.2$ (dashed lines, grey-filled symbols), $0.8$ (dotted lines, unfilled symbols), or $1$ (solid lines, black-filled symbols) and different $m$. \DGP1 and \DGP3 were fitted as banded vector autoregressive models with $m=70$ (squares), $156$ (circles), or $304$ (triangles), and \DGP2 was fitted as grouped vector autoregressive models with $m=72$ (squares), $120$ (circles), or $312$ (triangles).}
	\label{fig1}
\end{figure}

(i) When $\rho(A_*)=0.2$, for all data generating processes, the lines for different $m$ coincide completely with each other and scale perfectly linearly with $(m/n)^{1/2}$. This suggests that the actual error rate is $(m/n)^{1/2}$ when $\sigma_{\min}(A_*)$ lies in the slow rate regime of Theorem \ref{thm3}.

(ii) When $\rho(A_*)=0.8$ or 1, for \DGP1 and \DGP2, although $\lVert \widehat{\beta} - \beta_* \rVert$ is still proportional to $(m/n)^{1/2}$, the three lines for the same $\rho(A_*)$  but different $m$ do not coincide: fixing $\rho(A_*)$, the slope increases as $m$ increases, and the variation in slope is greater as $\rho(A_*)$ is larger.  Note that for \DGP1 and \DGP2, $\sigma_{\min}(A_*)$ is very small.
As  finding (i) suggests that the actual error rate is $(m/n)^{1/2}$ for small $\sigma_{\min}(A_*)$, this extra variation in slope  may be partially explained by the factor $\lambda_{\max}(\Gamma_{R,k})$ in the error bound in Theorem \ref{thm2} due to the effect of $R$. However, $\lVert \widehat{\beta} - \beta_* \rVert$ actually depends on the spectrum of $\Gamma_{R,k}$, and its largest eigenvalue merely serves as an upper bound. When $\rho(A_*)$ is smaller,  we will have more control over the spectrum of $A_*$ and hence that of $\Gamma_{R,k}$. This may explain why the variation in slope is smaller when  $\rho(A_*)$ is smaller.

(iii) For \DGP3 with $\rho(A_*)=0.2$ or $0.8$, the three lines corresponding to different $m$ still completely coincide with each other.  This can be explained by the fact that $\lambda_{\max}(\Gamma_{R,k})$ is independent of $R$ when $A_*=\rho I_d$; see \eqref{eq:lambda_max_equal}.

(iv) For \DGP3 with $\rho(A_*)=1$, in sharp contrast to all other cases, the  error rate  appears to be a quadratic function of $(m/n)^{1/2}$. This  matches the implication of  Theorem \ref{thm3} that when $\sigma_{\min}(A_*)=1$, the error rate falls into the fast rate regime.

(v) Fixing both $m$ and $n$, $\lVert \widehat{\beta} - \beta_* \rVert$ always decreases as $\rho(A_*)$ increases. Moreover, when $\sigma_{\min}(A_*)<1$, fixing $m$,  the lines become less steep as $\rho(A_*)$ increases. Note that $\sigma_{\min}(A_*)$ is larger when $\rho(A_*)$ is, due to our method of generating $A_*$. Thus, this finding can be explained by the factor $\{1-\sigma_{\min}^2(A_*)\}^{1/2}$ in the  error bound  for  the slow rate regime in Theorem \ref{thm3}.

\begin{figure}[t!]
	\centering
	\includegraphics[width=0.9\textwidth]{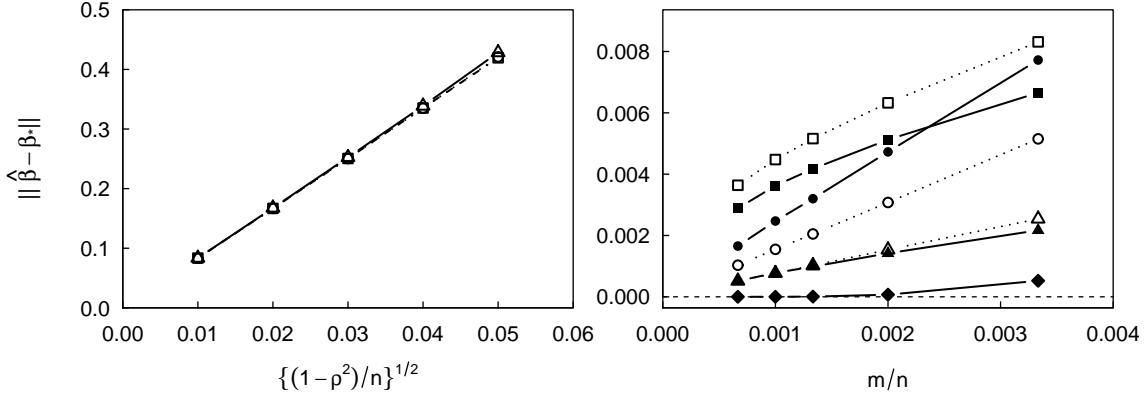}
	\caption{Error rates for \DGP3 as $\rho$ is fixed or approaching one at different rates. Left panel: plot of $\lVert \widehat{\beta} - \beta_* \rVert$ against $\{(1-\rho^2)/n\}^{1/2}$ with $\rho=0.2$ (dashed lines, squares), $0.4$ (dotted lines, circles) or $0.6$ (solid lines, triangles), and $m=70$. Right panel: plot of $\lVert \widehat{\beta} - \beta_* \rVert$ against $m/n$  with $\rho=0.99$ (squares), $1-(m+\log d)/n$ (circles), $1+1/n$ (triangles) or $1.01$ (diamonds), and $m=1$ (solid lines, filled symbols) or $70$ (dotted lines, unfilled symbols). The case of $(m,\rho)=(70,1.01)$ is omitted as the process becomes so explosive that the computation is numerically infeasible.}
	\label{fig2a}
\end{figure}

\begin{figure}[t!]
	\centering
	\includegraphics[width=0.9\textwidth]{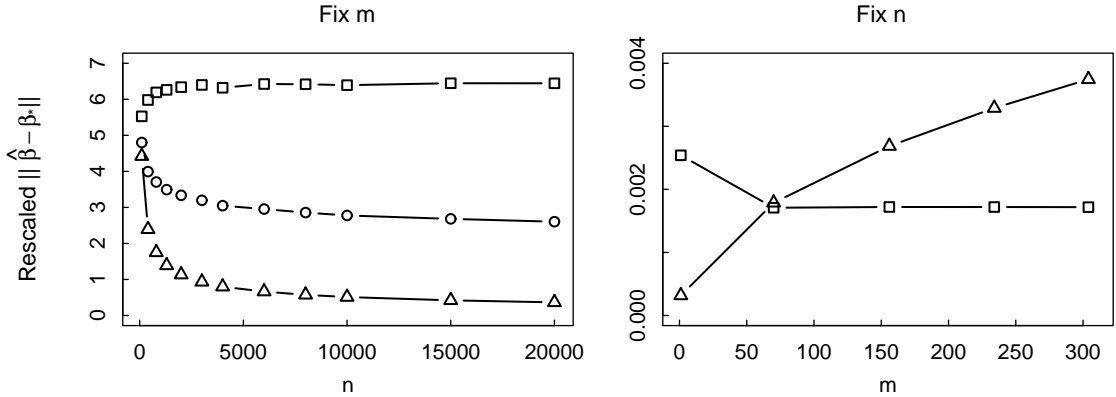}
	\caption{Error rates for \DGP3 when $\rho=1$. Left panel: plot of $\lVert \widehat{\beta} - \beta_* \rVert$ multiplied by $n/8$ (squares), $n/(2\log n)$ (circles) or $n^{1/2}$ (triangles)  against $n$, fixing $m=70$. Right panel: plot of $\lVert \widehat{\beta} - \beta_* \rVert$ divided by $m$  (squares) or $8m^{1/2}$ (triangles) against $m$, fixing $n=400$.}
	\label{fig2b}
\end{figure}

In the second experiment, we focus on \DGP3 to further investigate the error rates and the phase transition. We set $d=24$ and $m=1, 70, 156$ or $304$, where $m=1$ results from fitting the smallest true model, and $m=70, 156$ or $304$ from fitting a banded model with $k_0=1,3$ or $7$, respectively. Figures \ref{fig2a} and \ref{fig2b} display the results, where we have the following findings: 

(i) The combination of results from the first experiment and the left panel of Figure \ref{fig2a} suggests that $\lVert \widehat{\beta} - \beta_* \rVert$ scales as $O[\{(1-\rho^2)m/n\}^{1/2}]$ when $\vert \rho \vert$ is fixed at a level well below one.

(ii) Figure \ref{fig2b} suggests that when $\lvert\rho\rvert=1$ the actual error rate is $m/n$. Specifically, when $m$ is fixed, the left panel shows that $n \lVert \widehat{\beta} - \beta_* \rVert$ becomes stable for $n$ sufficiently large, while $\lVert \widehat{\beta} - \beta_* \rVert$ multiplied by $n^{1/2}$ or $n/\log n$ appears to diminish as $n\rightarrow\infty$. On the other hand, when $n$ is fixed, the right panel shows that $\lVert \widehat{\beta} - \beta_* \rVert/m$ becomes stable for $m$ sufficiently large. 

(iii) The right panel of Figure \ref{fig2a} suggests that the regime of rate $m/n$ is reached as early as $\vert\rho\vert=1-O\{(m+\log d)/n\}$ and maintains even as the process becomes slightly explosive with  $\vert\rho\vert=1+O(1/n)$. This lends support to the boundaries of the fast rate regime suggested by Theorem \ref{thm3}; see Remarks  \ref{remark:phases} and \ref{remark:phases_all}. By contrast, when $\rho=0.99$, the rate appears to be $(m/n)^{1/2}$, similarly to our findings in the first experiment.  On the other hand, when $\rho$ is fixed at a level slightly above one, the rate becomes even faster than $m/n$. This matches the conclusion in Remark \ref{remark:phases_all} that the corresponding lower bound  diminishes at a rate faster than $\lvert\rho\rvert^n$ as $n$ increases.

\begin{figure}[t!]
	\centering
	\includegraphics[width=0.9\textwidth]{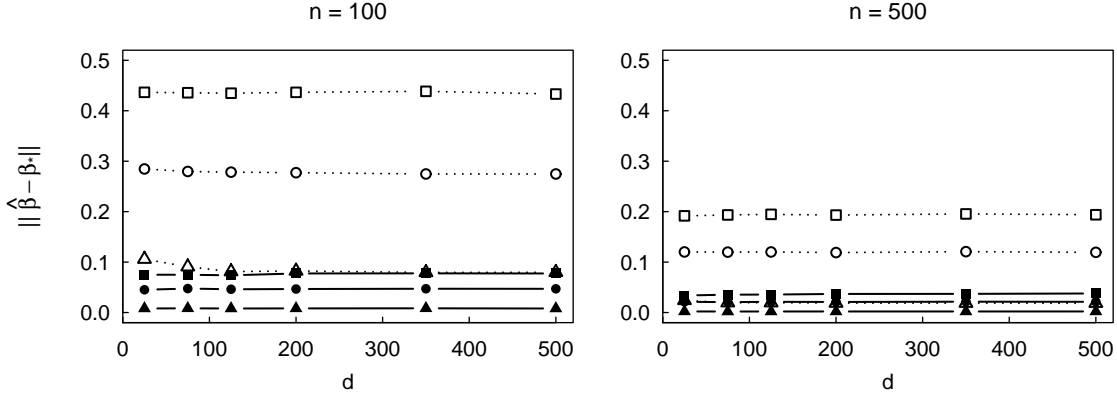}
	\caption{Plots of $\lVert \widehat{\beta} - \beta_* \rVert$ against $d$ for \DGP3  with $n=100$ (left panel) and $500$ (right panel), when $\rho=0.2$ (squares), $0.8$ (circles) or $1$ (triangles), and $m=1$ (solid lines, filled symbols) or $20$ (dotted lines, unfilled symbols).}
	\label{fig3}
\end{figure}

The third experiment aims to check if the ambient dimension $d$  directly affects  the estimation error. We generate data from \DGP3 with $\rho=0.2, 0.8$ or $1$, $n=100$ or $500$ and $d\in[25,500]$. For the estimation, we consider $m=1$ or $20$, where $m=1$ corresponds to the smallest true model, and $m=20$ corresponds to a model subject to (1) $a_{*11}=\cdots=a_{*dd}$ and (2) the restriction that all but $m-1$ of the off-diagonal entries of $A_*$ are zero.  To generate the pattern in (2), we sample the $m-1$ positions uniformly without replacement from all off-diagonal positions of $A_*$. 

Figure \ref{fig3} shows that the estimation error is constant in $d$ for almost all cases. This confirms that $d$ does not affect the estimation error when $\lvert \rho\rvert <1-O\left \{(m+\log d)/n\right \}$; see Remark \ref{remark:phases_all}. Moreover, the extra factor of $\log d$ in the theoretical upper bounds for the other regimes might not be necessary. 	For $(n, m, \rho)=(100, 20, 1)$, the estimation error seems less stable when $d\leq 75$; see the left panel of Figure \ref{fig3}. 	
This might be explained  by the indirect effect of $R$ on $\lambda_{\max}(\Gamma_{R,k})$ in Theorem \ref{thm2}. As $m$ is fixed at $20$, different $d$ corresponds to different $R$. The spectrum of $\Gamma_{R,k}$ may be more sensitive to $R$ when $d$ is smaller, and the resulting impact on the estimation error may be more pronounced when $n$ is smaller. However, as $d$ grows, the restrictions will  become relatively more sparse, so eventually the spectrum of $\Gamma_{R,k}$ will  be stable, and the impact of $d$ will be negligible.

\section{Discussion\label{sec:discuss}}
An interesting future direction is dimensionality reduction for vector autoregressive models with data-driven restrictions.
Such a procedure involves first suggesting possible restrictions based on subject knowledge and then selecting the true restrictions by a data-driven approach. Note that the lasso method \citep{Davis_Zang_Zheng2015, Basu_Michailidis2015} can be viewed as a procedure where zero restrictions are initially suggested for all entries of $A$, and then the true zeros are identified by penalized estimation. Adopting a more general point of view, the modeller can initially suggest the general linear restrictions  \eqref{eq:lr}  instead. This will enable a more flexible and data-driven integration of expert knowledge. On the other hand, if it is known that only zero and equality restrictions are true, yet the locations of the restrictions are unknown, we can select the true restrictions efficiently by  the  delete or merge regressors algorithm proposed by \cite{Maj-Kanska2015} based on the Bayesian information criterion. The consistency of this procedure can be easily extended to vector autoregressive models.

\section*{Acknowledgement}
We thank the editor, associate editor and two referees for their invaluable comments, which have led to substantial improvements of our paper. Cheng's research was partially supported by the U.S. National Science Foundation and the Office of Naval Research, and he wishes to thank the Institute for Advanced Study at Princeton for its hospitality during his visit in Fall 2019.

\section*{Supplementary material}
\label{SM}
Supplementary material available at \Bka\ online includes all technical proofs of this paper.

%

\vspace*{-10pt}

\bibliographystyle{biometrika}
\bibliography{VAR_restricted}

\begin{thebibliography}{}

\bibitem[Ahn and Reinsel, 1988]{Ahn_Reinsel1988}
Ahn, S.~K. and Reinsel, G.~C. (1988).
\newblock Nested reduced-rank autogressive models for multiple time series.
\newblock {\em Journal of the American Statistical Association}, 83:849--856.

\bibitem[Basu and Michailidis, 2015]{Basu_Michailidis2015}
Basu, S. and Michailidis, G. (2015).
\newblock Regularized estimation in sparse high-dimensional time series models.
\newblock {\em The Annals of Statistics}, 43:1535--1567.

\bibitem[Boucheron et~al., 2013]{Boucheron_Lugosi_Massart2013}
Boucheron, S., Lugosi, G., and Massart, P. (2013).
\newblock {\em Concentration inequalities: A nonasymptotic theory of
  independence}.
\newblock Oxford University Press, Oxford.

\bibitem[Bringmann et~al., 2013]{Bringmann2013}
Bringmann, L.~F., Vissers, N., Wichers, M., Geschwind, N., Kuppens, P.,
  Peeters, F., Borsboom, D., and Tuerlinckx, F. (2013).
\newblock A network approach to psychopathology: \textsc{N}ew insights into
  clinical longitudinal data.
\newblock {\em PLoS ONE}, 8:e60188.

\bibitem[Chang, 2004]{Chang2004}
Chang, Y. (2004).
\newblock Bootstrap unit root tests in panels with cross-sectional dependency.
\newblock {\em Journal of Econometrics}, 120:263--293.

\bibitem[Davis et~al., 2015]{Davis_Zang_Zheng2015}
Davis, R.~A., Zang, P., and Zheng, T. (2015).
\newblock Sparse vector autoregressive modeling.
\newblock {\em Journal of Computational and Graphical Statistics},
  25:1077--1096.

\bibitem[Dowell and Pinson, 2016]{Dowell_Pinson2016}
Dowell, J. and Pinson, P. (2016).
\newblock Very-short-term probabilistic wind power forecasts by sparse vector
  autoregression.
\newblock {\em IEEE Transactions on Smart Grid}, 7:763--770.

\bibitem[Fang et~al., 1990]{Fang_Kotz_Ng1990}
Fang, K.-T., Kotz, S., and Ng, K.~W. (1990).
\newblock {\em Symmetric Multivariate and Related Distributions}.
\newblock Chapman and Hall/CRC, New York.

\bibitem[Faradonbeh et~al., 2018]{Faradonbeh_Tewari_Michailidis2018}
Faradonbeh, M. K.~S., Tewari, A., and Michailidis, G. (2018).
\newblock Finite time identification in unstable linear systems.
\newblock {\em Automatica}, 96:342--353.

\bibitem[Gorrostieta et~al., 2012]{Gorrostieta2012}
Gorrostieta, C., Ombao, H., B\'{e}dard, P., and Sanes, J.~N. (2012).
\newblock Investigating brain connectivity using mixed effects vector
  autoregressive models.
\newblock {\em NeuroImage}, 59:3347--3355.

\bibitem[Guo et~al., 2016]{Guo_Wang_Yao2016}
Guo, S., Wang, Y., and Yao, Q. (2016).
\newblock High-dimensional and banded vector autoregressions.
\newblock {\em Biometrika}, 103:889--903.

\bibitem[Hamilton, 1994]{Hamilton1994}
Hamilton, J.~D. (1994).
\newblock {\em Time Series Analysis}.
\newblock Princeton University Press, Princeton.

\bibitem[Han et~al., 2015a]{Han_Lu_Liu2015}
Han, F., Lu, H., and Liu, H. (2015a).
\newblock A direct estimation of high dimensional stationary vector
  autoregressions.
\newblock {\em Journal of Machine Learning Research}, 16:3115--3150.

\bibitem[Han et~al., 2015b]{Han_Xu_Liu2015}
Han, F., Xu, S., and Liu, H. (2015b).
\newblock Rate-optimal estimation of a high-dimensional semiparametric time
  series model.
\newblock Preprint.

\bibitem[Horn and Johnson, 1985]{Horn_Johnson1985}
Horn, R.~A. and Johnson, C.~R. (1985).
\newblock {\em Matrix Analysis}.
\newblock Cambridge University Press, New York.

\bibitem[Kano, 1994]{Kano1994}
Kano, Y. (1994).
\newblock Consistency property of elliptical probability density functions.
\newblock {\em Journal of Multivariate Analysis}, 51:139--147.

\bibitem[Kotz and Nadarajah, 2004]{Kotz_Nadarajah2004}
Kotz, S. and Nadarajah, S. (2004).
\newblock {\em Multivariate t-Distributions and Their Applications}.
\newblock Cambridge University Press.

\bibitem[L\"{u}tkepohl, 2005]{Lutkepohl2005}
L\"{u}tkepohl, H. (2005).
\newblock {\em New Introduction to Multiple Time Series Analysis}.
\newblock Springer-Verlag Berlin Heidelberg.

\bibitem[Maj-Ka\'{n}ska et~al., 2015]{Maj-Kanska2015}
Maj-Ka\'{n}ska, A., Pokarowski, P., and Prochenka, A. (2015).
\newblock Delete or merge regressors for linear model selection.
\newblock {\em Electronic Journal of Statistics}, 9:1749--1778.

\bibitem[Mendelson, 2014]{Mendelson2014}
Mendelson, S. (2014).
\newblock Learning without concentration.
\newblock In {\em JMLR: Workshop and Conference Proceedings}, volume~35, pages
  1--15.

\bibitem[Negahban and Wainwright, 2011]{Negahban_Wainwright2011}
Negahban, S. and Wainwright, M.~J. (2011).
\newblock Estimation of (near) low-rank matrices with noise and
  high-dimensional scaling.
\newblock {\em The Annals of Statistics}, 39:1069--1097.

\bibitem[Recht, 2018]{Recht2018}
Recht, B. (2018).
\newblock A tour of reinforcement learning: \textsc{T}he view from continuous
  control.
\newblock {\em arXiv:1806.09460}.

\bibitem[Reinsel, 1993]{Reinsel1993}
Reinsel, G.~C. (1993).
\newblock {\em Elements of Multivariate Time Series Analysis}.
\newblock Springer-Verlag, New York.

\bibitem[Rudelson and Vershynin, 2015]{Rudelson_Vershynin2015}
Rudelson, M. and Vershynin, R. (2015).
\newblock Small ball probabilities for linear images of high-dimensional
  distributions.
\newblock {\em International Mathematics Research Notices}, 2015:9594--9617.

\bibitem[Simchowitz et~al., 2018]{Simchowitz2018}
Simchowitz, M., Mania, H., Tu, S., Jordan, M., and Recht, B. (2018).
\newblock Learning without mixing: \textsc{T}owards a sharp analysis of linear
  system identification.
\newblock In {\em Proceedings of Machine Learning Research}, volume~75, pages
  439--473.
\newblock 31st Annual Conference on Learning Theory.

\bibitem[Sims, 1980]{Sims1980}
Sims, C.~A. (1980).
\newblock Macroeconomics and reality.
\newblock {\em Econometrica}, 48:1--48.

\bibitem[Stock and Watson, 2001]{Stock_Watson2001}
Stock, J.~H. and Watson, M.~W. (2001).
\newblock Vector autoregressions.
\newblock {\em Journal of Economic Perspectives}, 15:101--115.

\bibitem[Tsay, 2013]{Tsay2013}
Tsay, R.~S. (2013).
\newblock {\em Multivariate Time Series Analysis: With R and Financial
  Applications}.
\newblock John Wiley \& Sons.

\bibitem[Vershynin, 2012]{Vershynin2012}
Vershynin, R. (2012).
\newblock Introduction to the nonasymptotic analysis of random matrices.
\newblock In Eldar, Y. and Kutyniok, G., editors, {\em Compressed Sensing:
  Theory and Applications}, chapter~5, pages 210--268. Cambridge University
  Press.

\bibitem[Vershynin, 2018]{Vershynin2018}
Vershynin, R. (2018).
\newblock {\em High-Dimensional Probability: An Introduction with Applications
  in Data Science}.
\newblock Cambridge University Press, Cambridge.

\bibitem[Wilms et~al., 2017]{Wilms_Basu_Bien_Matteson2017}
Wilms, I., Basu, S., Bien, J., and Matteson, D.~S. (2017).
\newblock Interpretable vector autoregressions with exogenous time series.
\newblock In {\em Symposium on Interpretable Machine Learning, 31st Conference
  on Neural Information Processing Systems (NIPS 2017)}.

\bibitem[Wu and Xia, 2016]{Wu_Xia2016}
Wu, J.~C. and Xia, F.~D. (2016).
\newblock Measuring the macroeconomic impact of monetary policy at the zero
  lower bound.
\newblock {\em Journal of Money, Credit and Banking}, 48:253--291.

\bibitem[Zhang et~al., 2018]{Zhang_Pan_Gao2018}
Zhang, B., Pan, G., and Gao, J. (2018).
\newblock \textsc{CLT} for largest eigenvalues and unit root testing for
  high-dimensional nonstationary time series.
\newblock {\em The Annals of Statistics}, 46:2186--2215.

\bibitem[Zhou et~al., 2018]{Zhou_Bose_Fan_Liu2018}
Zhou, W.-X., Bose, K., Fan, J., and Liu, H. (2018).
\newblock A new perspective on robust \textsc{M}-estimation: \textsc{F}inite
  sample theory and applications to dependence-adjusted multiple testing.
\newblock {\em The Annals of Statistics}, 46:1904--1931.

\bibitem[Zhu et~al., 2017]{Zhu_Pan_Li_Liu_Wang2017}
Zhu, X., Pan, R., Li, G., Liu, Y., and Wang, H. (2017).
\newblock Network vector autoregression.
\newblock {\em The Annals of Statistics}, 45:1096--1123.

\end{thebibliography}


\begin{thebibliography}{8}
\bibitem[Basu and Michailidis, 2015]{Basu_Michailidis2015}
Basu, S. and Michailidis, G. (2015).
\newblock Regularized estimation in sparse high-dimensional time series models.
\newblock {\em The Annals of Statistics}, 43:1535--1567.

\bibitem[Boucheron et~al., 2013]{Boucheron_Lugosi_Massart2013}
Boucheron, S., Lugosi, G., and Massart, P. (2013).
\newblock {\em Concentration inequalities: A nonasymptotic theory of
	independence}.
\newblock Oxford University Press, Oxford.

\bibitem[Horn and Johnson, 1985]{Horn_Johnson1985}
Horn, R.~A. and Johnson, C.~R. (1985).
\newblock {\em Matrix Analysis}.
\newblock Cambridge University Press, New York.

\bibitem[L\"{u}tkepohl, 2005]{Lutkepohl2005}
L\"{u}tkepohl, H. (2005).
\newblock {\em New Introduction to Multiple Time Series Analysis}.
\newblock Springer-Verlag Berlin Heidelberg.

\bibitem[Rudelson and Vershynin, 2015]{Rudelson_Vershynin2015}
Rudelson, M. and Vershynin, R. (2015).
\newblock Small ball probabilities for linear images of high-dimensional
distributions.
\newblock {\em International Mathematics Research Notices}, 2015:9594--9617.

\bibitem[Simchowitz et~al., 2018]{Simchowitz2018}
Simchowitz, M., Mania, H., Tu, S., Jordan, M., and Recht, B. (2018).
\newblock Learning without mixing: \textsc{T}owards a sharp analysis of linear
system identification.
\newblock In {\em Proceedings of Machine Learning Research}, volume~75, pages
439--473.
\newblock 31st Annual Conference on Learning Theory.

\bibitem[Vershynin, 2012]{Vershynin2012}
Vershynin, R. (2012).
\newblock Introduction to the nonasymptotic analysis of random matrices.
\newblock In Eldar, Y. and Kutyniok, G., editors, {\em Compressed Sensing:
	Theory and Applications}, chapter~5, pages 210--268. Cambridge University
Press.

\bibitem[Vershynin, 2018]{Vershynin2018}
Vershynin, R. (2018).
\newblock {\em High-Dimensional Probability: An Introduction with Applications
	in Data Science}.
\newblock Cambridge University Press, Cambridge.

\end{thebibliography}

\clearpage


\newpage
\begin{center}
	{\Large \bf Supplementary Material: Finite Time Analysis of Vector Autoregressive Models under Linear Restrictions}
\end{center}

\begin{abstract}
	This supplementary material contains all technical proofs of the main paper. $\S$ \ref{sec:S1} gives the proofs of Theorem \ref{thm1} and Proposition \ref{prop1}, which rely on three auxiliary lemmas, Lemmas \ref{lem_A1}--\ref{lem_A3}, whose proofs are relegated to $\S$ \ref{subsec:lem_A_proofs}. $\S$ \ref{sec:S2} contains the proofs of Lemmas \ref{lem_bmsb}--\ref{lem_gram}. In $\S$ \ref{sec:S3}, we first verify equation \eqref{logdet} in the main paper and then prove Proposition \ref{prop_G} through four auxiliary lemmas,  Lemmas \ref{lem_kappa}--\ref{lem_Sigma_X}. $\S$ \ref{sec:S4} contains the proof of Theorem \ref{thm3}. Lastly $\S$ \ref{sec:S5}  proves Theorem \ref{thm_lower} and Corollary \ref{cor_lower} after introducing two auxiliary lemmas, Lemmas \ref{lem_A4} and \ref{lem_A5}.
\end{abstract}

\renewcommand{\thesection}{S\arabic{section}}
\renewcommand{\theequation}{S\arabic{equation}}
\renewcommand{\thelemma}{S\arabic{lemma}}
\setcounter{section}{0}
\setcounter{equation}{0}
\setcounter{lemma}{0}
\setcounter{figure}{0}
\setcounter{table}{0}

\section{Proofs of Theorem \ref{thm1} and Proposition \ref{prop1}\label{sec:S1}}
\subsection{Three Auxiliary Lemmas}
The proofs of Theorem \ref{thm1} and Proposition \ref{prop1} rely on three auxiliary lemmas, Lemmas \ref{lem_A1}--\ref{lem_A3}. Lemma \ref{lem_A1} contains key results on covering and discretization.  Lemma \ref{lem_A2} gives a poinwise lower bound of $X^\T X$ through aggregation of all the $\lfloor n/k \rfloor$ blocks of size $k$ using the Chernoff bound. Notice that the probability guarantee in Lemma \ref{lem_A2} will degrade as $k$ increases, since the probability guarantee of the Chernoff bound will degrade as the number of blocks decreases. Lemma \ref{lem_A3} is a multivariate concentration bound for dependent data, respectively.  We state these lemmas first and relegate their proofs to $\S$ \ref{subsec:lem_A_proofs}.

The following notations will be used throughout our proofs: For any integer $d\geq1$ and matrix $0\prec\Gamma\in\mathbb{R}^{d\times d}$, let $\lVert \Gamma^{1/2}(\cdot)\rVert$ be the ellipsoidal vector norm associated to $\Gamma$, i.e., the mapping from  $\omega\in\mathbb{R}^d$ to  $(\omega^\T\Gamma\omega)^{1/2}\in(0,\infty)$. In addition, we denote the corresponding unit ball, or ellipsoid, by $\mathcal{S}_{\Gamma}=\{\omega\in\mathbb{R}^d: \lVert \Gamma^{1/2}\omega\rVert=1\}$. For any set $\mathcal{S}$, we denote its  cardinality, complement and volume  by $\vert \mathcal{S} \vert$, $\mathcal{S}^c$ and $\text{vol}(\mathcal{S})$, respectively.

\begin{lemma}\label{lem_A1}
	Suppose that $Z\in\mathbb{R}^{n\times m}$ and $0\prec \Gamma_{\min} \preceq \Gamma_{\max} \in\mathbb{R}^{m\times m}$. Let $\mathcal{T}$ be a $1/4$-net of $\mathcal{S}_{\Gamma_{\min}}$ in the norm $\lVert \Gamma_{\max}^{1/2}(\cdot)\rVert$. Then, the following holds:
	
	(i) If $\Gamma_{\min}/2 \npreceq Z^\T Z \preceq  \Gamma_{\max}$, then 	$\inf_{\omega\in\mathcal{T}} \omega^\T Z^\T Z \omega < 1$.
	
	(ii) If $\mathcal{T}$ is a minimal $1/4$-net, then
	$\log \vert\mathcal{T} \vert \leq m\log9+(1/2)\log\det(\Gamma_{\max}\Gamma_{\min}^{-1})$.
	
	(iii) If $\Gamma_{\min} \preceq Z^\T Z\preceq \Gamma_{\max}$, then for any $\nu\in\mathbb{R}^{n}$, we have
	\begin{equation*}
	\sup_{\omega\in \mathcal{S}^{m-1}}\frac{\omega^\T Z^\T \nu}{\lVert Z \omega\rVert}=\sup_{\omega\in \mathcal{S}_{\Gamma_{\min}}}\frac{\omega^\T Z^\T \nu}{\lVert Z \omega\rVert} \leq 2 \max_{\omega\in\mathcal{T}}\frac{\omega^\T Z^\T \nu}{\lVert Z \omega\rVert}.
	\end{equation*}
\end{lemma}

\begin{lemma}\label{lem_A2}
	Suppose that the process $\{X_t\}_{t=1}^n$ taking values in $\mathbb{R}^d$ satisfies the $(k,\Gamma_{\smb},\alpha)$-\BMSB\ condition. Let $X=(X_1,\ldots, X_n)^\T$. Then, for any $\omega\in\mathbb{R}^{d}$, we have
	\[
	\P \left(\omega^\T X^\T X \omega \leq  \frac{\alpha^2 k \lfloor n/k\rfloor}{8}  \omega^\T\Gamma_{\smb}\omega \right ) \leq \exp\left (-\frac{\alpha^2 \lfloor n/k \rfloor}{8}\right ).
	\]	
\end{lemma}

\begin{lemma}\label{lem_A3} Let $\{\mathcal{F}_t, t=1,2,\ldots\}$ be a filtration. Suppose that $\{x_t, t=1,2,\ldots\}$ and $\{\eta_t, t=1,2,\ldots\}$ are processes taking values in $\mathbb{R}^q$, and for each integer $t\geq1$, $x_t$ is $\mathcal{F}_t$-measurable, $\eta_t$ is $\mathcal{F}_{t+1}$-measurable, and $\eta_t\mid \mathcal{F}_t$ is mean-zero and $\sigma^2$-sub-Gaussian. Then, for any constants $\beta_-, \beta_+, \gamma>0$, we have 
	\begin{equation}\label{eq:lem_martingale}
	\P\left \{ \frac{ \sum_{t=1}^{n}x_t^\T \eta_t }{\left (\sum_{t=1}^{n}\lVert x_t \rVert^2\right )^{1/2}}\geq\gamma, \; \sum_{t=1}^{n}\lVert x_t \rVert^2 \in [\beta_-, \beta_+]\right \} \leq \frac{\beta_+}{\beta_-} \exp\left (-\frac{\gamma^2}{6\sigma^2}\right ).
	\end{equation}
\end{lemma}

\subsection{Proof of Theorem \ref{thm1}}
Define the $m\times m$ matrices 
\begin{equation}\label{eq:Gammas}
\Gamma_{\max}=n\overline{\Gamma}_R, \hspace{5mm} \Gamma_{\min}=\frac{\alpha^2 k \lfloor n/k\rfloor}{8}\underline{\Gamma}_{R}, \hspace{5mm}	\underline{\Gamma}_{\min}=\Gamma_{\min}/2,
\end{equation}
where
$\overline{\Gamma}_R=R^\T (I_q \otimes \overline{\Gamma}) R$ and $\underline{\Gamma}_{R}=R^\T (I_q \otimes \Gamma_{\smb}) R$. Since $R$ has full column rank, $\overline{\Gamma}_R$ and $\underline{\Gamma}_{R}$ are both  positive definite matrices. Thus, $0\prec \underline{\Gamma}_{\min} \prec \Gamma_{\min} \preceq \Gamma_{\max}$. 

Consider the singular value decomposition
$Z=\mathcal{U} \mathcal{D} \mathcal{V}^\T$,
where $\mathcal{U}\in\mathbb{R}^{q n\times m}$, $\mathcal{D}, \mathcal{V} \in\mathbb{R}^{m\times m}$, and $\mathcal{U}^\T \mathcal{U}=I_m=\mathcal{V}^\T\mathcal{V}$. Let $Z^\dagger$ be the Moore-Penrose pseudoinverse of $Z$, i.e., $Z^\dagger=\mathcal{V} \mathcal{D}^{-}\mathcal{U}^\T$, where the diagonal matrix $\mathcal{D}^{-}$ is defined by taking the reciprocal of each nonzero diagonal entry of $\mathcal{D}$; in particular, $Z^\dagger=(Z^\T Z)^{-1}Z^\T$ if $Z^\T Z\succ 0$. Then, we have $\widehat{\theta}-\theta_*=Z^\dagger\eta$. As a result,
\begin{equation*}
\widehat{\beta} - \beta_*=R (\widehat{\theta}-\theta_*)=RZ^\dagger\eta=R\mathcal{V} \mathcal{D}^{-}\mathcal{U}^\T\eta.
\end{equation*}
Furthermore, since $\underline{\Gamma}_{\min}\succ0$, it holds  on the event $\{Z^\T Z \succeq \underline{\Gamma}_{\min}\}$ that
\begin{align}\label{thm1_decompose}
\lVert\widehat{\beta} - \beta_* \rVert
\leq \lVert R\mathcal{V} \mathcal{D}^{-}\rVert_2\lVert\mathcal{U}^\T\eta\rVert &= \left [\lambda_{\max}\{R (Z^\T Z)^{-1} R^\T\}\right ]^{1/2}\lVert\mathcal{U}^\T\eta\rVert \notag \\
&\leq \left \{\lambda_{\max}(R \underline{\Gamma}_{\min}^{-1} R^\T)\right \}^{1/2}\lVert\mathcal{U}^\T\eta\rVert.
\end{align}
Note that \eqref{thm1_decompose} exploits the self-cancellation effect inside the pseudoinverse $Z^\dagger$: the bound would not be as sharp if  $R$, $Z^\T Z$ and $Z\eta$ were bounded separately.

By \eqref{thm1_decompose} and Assumption \ref{assum:upper_matrix}, i.e., $\P(Z^\T Z \npreceq \Gamma_{\max})\leq \delta$, for any $K>0$, we have
\begin{align}\label{eq:thm1_aeq1}
& \P \left [\lVert \widehat{\beta} - \beta_*\rVert 
\geq K \left \{\lambda_{\max}(R \underline{\Gamma}_{\min}^{-1} R^\T)\right \}^{1/2} \right] \notag\\
&\hspace{5mm} 
\leq \P \left [\lVert \widehat{\beta} - \beta_*\rVert   \geq K \left \{\lambda_{\max}(R \underline{\Gamma}_{\min}^{-1} R^\T)\right \}^{1/2}, \; Z^\T Z \preceq \Gamma_{\max} \right] + \delta \notag\\
&\hspace{5mm} 
\leq \P \left [ \lVert \widehat{\beta} - \beta_*\rVert   \geq K \left \{\lambda_{\max}(R \underline{\Gamma}_{\min}^{-1} R^\T)\right \}^{1/2}, \;  \underline{\Gamma}_{\min} \preceq Z^\T Z \preceq \Gamma_{\max}\right ] \notag\\
& \hspace{10mm} +\P \left ( \underline{\Gamma}_{\min} \npreceq Z^\T Z \preceq \Gamma_{\max}\right )+\delta \notag\\
&\hspace{5mm} 
\leq \P \left ( \lVert \mathcal{U}^\T \eta \rVert \geq K, \; \underline{\Gamma}_{\min} \preceq Z^\T Z \preceq \Gamma_{\max}\right) +\P \left ( \underline{\Gamma}_{\min} \npreceq Z^\T Z \preceq \Gamma_{\max}\right )+\delta.
\end{align}

Notice that condition \eqref{thm1_condition} implies $k \leq n/10$, so that 
\begin{equation}\label{eq:floor_T}
k \lfloor n/k\rfloor \geq n-k \geq (9/10) n.
\end{equation}
As a result, 
\begin{equation}\label{eq:thm1_aeq2}
\left \{\lambda_{\max}(R \underline{\Gamma}_{\min}^{-1} R^\T)\right \}^{1/2}\leq \frac{9}{2\alpha} \left \{\frac{\lambda_{\max}(R \underline{\Gamma}_R^{-1} R^\T)}{n}\right \}^{1/2}.
\end{equation}
In view of \eqref{eq:thm1_aeq1} and \eqref{eq:thm1_aeq2}, to prove this theorem, it remains to show that $Z^\T Z$ is bounded below and $\lVert \mathcal{U}^\T \eta \rVert$ is bounded above, with high probability. Specifically, we will prove that  
\begin{equation}\label{thm1_part1}
\P \left ( \underline{\Gamma}_{\min} \npreceq Z^\T Z \preceq \Gamma_{\max}\right ) \leq \delta
\end{equation}
if condition \eqref{thm1_condition} of the theorem holds, and
\begin{equation}\label{thm1_part2}
\P \left ( \lVert \mathcal{U}^\T \eta \rVert \geq K, \; \underline{\Gamma}_{\min} \preceq Z^\T Z \preceq \Gamma_{\max}\right) \leq \delta
\end{equation}
if we choose
\begin{equation*}
K=2\sigma \left \{12m \log(14/\alpha)+9\log\det(\overline{\Gamma}_R\underline{\Gamma}_{R}^{-1}) +6\log(1/\delta)\right \}^{1/2}.
\end{equation*}

\noindent{\textit{Proof of \eqref{thm1_part1}}:} 
Let $\mathcal{T}$ be a minimal $1/4$-net of $\mathcal{S}_{\Gamma_{\min}}$ in the norm $\lVert \Gamma_{\max}^{1/2}(\cdot)\rVert$. By Lemma \ref{lem_A1}(i), we have 
\begin{equation}\label{eq:thm1_covering}
\P \left (\Gamma_{\min}/2 \npreceq Z^\T Z \preceq  \Gamma_{\max} \right ) 
\leq \P\left( \inf_{\omega\in\mathcal{T}} \omega^\T Z^\T Z \omega < 1\right)
\leq \vert \mathcal{T} \vert \sup_{\omega\in\mathcal{S}_{\Gamma_{\min}}} \P\left(\omega^\T Z^\T Z \omega < 1 \right).
\end{equation}

By Lemma \ref{lem_A1}(ii) and \eqref{eq:floor_T}, we have
\begin{align}\label{eq:thm1_covernum_T}
\log \vert\mathcal{T} \vert &\leq m\log9+(1/2)\log\det(\Gamma_{\max}\Gamma_{\min}^{-1}) \notag\\
&=m\log9+(1/2)m \log\frac{8n}{k\lfloor n/k\rfloor \alpha^2}+(1/2)\log\det(\overline{\Gamma}_R\underline{\Gamma}_{R}^{-1}) \notag\\
&\leq m \log(27/\alpha)+(1/2)\log\det(\overline{\Gamma}_R\underline{\Gamma}_{R}^{-1}).
\end{align}

Note that $Z^\T Z=R^\T (I_q \otimes X^\T X) R=\sum_{i=1}^q R_i^\T X^\T X R_i$, where each $R_i$ is a $d\times m$ block in $R=(R_1^\T, \ldots, R_q^\T)^\T$. Likewise,	
$\Gamma_{\min}=(1/8){\alpha^2 k \lfloor n/k\rfloor}\sum_{i=1}^q R_i^\T \Gamma_{\smb} R_i$. By a change of variables and Lemma \ref{lem_A2}, we have
\begin{align*}
\sup_{\omega\in\mathcal{S}_{\Gamma_{\min}}}\P\left(\omega^\T Z^\T Z \omega < 1 \right) &= \sup_{\omega\in\mathbb{R}^m}\P\left(\omega^\T Z^\T Z \omega < \omega^\T \Gamma_{\min} \omega\right)\\
&=\sup_{\omega\in\mathbb{R}^m}\P\left(\sum_{i=1}^q \omega^\T R_i^\T X^\T X R_i \omega < \frac{\alpha^2 k \lfloor n/k\rfloor}{8} \sum_{i=1}^q \omega^\T R_i^\T   \Gamma_{\smb} R_i \omega \right)\\
&\leq \sum_{i=1}^q \sup_{\omega\in\mathbb{R}^m} \P\left(\omega^\T R_i^\T X^\T X R_i \omega < \frac{\alpha^2 k \lfloor n/k\rfloor}{8} \omega^\T R_i^\T   \Gamma_{\smb} R_i \omega \right)\\
&\leq q \exp\left (-\frac{\alpha^2 \lfloor n/k \rfloor}{8}\right ) \leq q \exp\left (-\frac{\alpha^2 n}{9k}\right ),
\end{align*}
where we used \eqref{eq:floor_T} again in the last inequality.
This, together with \eqref{eq:thm1_covering} and \eqref{eq:thm1_covernum_T}, yields
\[
\P \left (\Gamma_{\min}/2 \npreceq Z^\T Z \preceq  \Gamma_{\max} \right ) \leq \exp\left \{ m \log\frac{27}{\alpha}+\frac{1}{2}\log\det(\overline{\Gamma}_R\underline{\Gamma}_{R}^{-1})+\log q-\frac{\alpha^2 n}{9k}\right \}\leq \delta,
\]
as long as condition \eqref{thm1_condition} of the theorem holds.

\noindent{\textit{Proof of  \eqref{thm1_part2}}:}
Recall that $Z=\mathcal{U}\mathcal{D}\mathcal{V}^\T$ and $\mathcal{U}^\T\mathcal{U}=I_m$. Thus, on the event $\{\underline{\Gamma}_{\min} \preceq Z^\T Z \preceq \Gamma_{\max}\}$, we have
\begin{align*}
\lVert \mathcal{U}^\T \eta \rVert = \sup_{\omega\in \mathbb{R}^{m}\setminus\{0\}} \frac{\omega^\T \mathcal{U}^\T \eta}{\lVert\omega\rVert}
= \sup_{\omega\in \mathbb{R}^{m}\setminus\{0\}} \frac{\omega^\T \underline{\Gamma}_{\min}^{-1/2}\mathcal{V} \mathcal{D} \mathcal{U}^\T \eta}{\lVert \mathcal{D}\mathcal{V}^\T\underline{\Gamma}_{\min}^{-1/2}\omega\rVert}
&= \sup_{\omega\in \mathbb{R}^{m}\setminus\{0\}} \frac{\omega^\T \underline{\Gamma}_{\min}^{-1/2}Z^\T \eta}{\lVert Z\underline{\Gamma}_{\min}^{-1/2}\omega\rVert} \\
&= \sup_{\omega\in \mathcal{S}_{\underline{\Gamma}_{\min}}} \frac{\omega^\T Z^\T \eta}{\lVert Z\omega\rVert},
\end{align*}
where the second equality uses the fact that $\mathcal{D}\mathcal{V}^\T\underline{\Gamma}_{\min}^{-1/2}$ is nonsingular if $Z^\T Z \succeq \underline{\Gamma}_{\min} \succ0$. Then it follows from Lemma \ref{lem_A1}(iii) that, on the event $\{\underline{\Gamma}_{\min} \preceq Z^\T Z \preceq \Gamma_{\max}\}$, we have
\begin{equation*}
\lVert \mathcal{U}^\T \eta \rVert  \leq 2\max_{\omega\in \underline{\mathcal{T}}} \frac{\omega^\T Z^\T \eta}{\lVert Z\omega\rVert},
\end{equation*}
where $\underline{\mathcal{T}}$ is a $1/4$-net of $\mathcal{S}_{\underline{\Gamma}_{\min}}$ in the norm $\lVert \Gamma_{\max}^{1/2}(\cdot)\rVert$. Therefore,
\begin{align}\label{eq:thm1_discrete}
\begin{split}
&\P \left ( \lVert \mathcal{U}^\T \eta \rVert \geq K, \; \underline{\Gamma}_{\min} \preceq Z^\T Z \preceq \Gamma_{\max}\right)\\
&\hspace{5mm} \leq \P \left ( \max_{\omega\in \underline{\mathcal{T}}} \frac{\omega^\T Z^\T \eta}{\lVert Z\omega\rVert} \geq K/2, \; \underline{\Gamma}_{\min} \preceq Z^\T Z \preceq \Gamma_{\max}\right)\\ 
&\hspace{5mm} \leq \vert\underline{\mathcal{T}}\vert 
\sup_{\omega\in \mathcal{S}_{\underline{\Gamma}_{\min}}}
\P \left ( \frac{\omega^\T Z^\T \eta}{\lVert Z\omega\rVert} \geq K/2, \; \underline{\Gamma}_{\min} \preceq Z^\T Z \preceq \Gamma_{\max} \right )\\
&\hspace{5mm} = \vert\underline{\mathcal{T}}\vert 
\sup_{\omega\in \mathcal{S}^{m-1}}
\P \left (\frac{\omega^\T \underline{\Gamma}_{\min}^{-1/2} Z^\T \eta}{\lVert Z \underline{\Gamma}_{\min}^{-1/2} \omega\rVert} \geq K/2, 
\; I_d \preceq \underline{\Gamma}_{\min}^{-1/2}Z^\T Z \underline{\Gamma}_{\min}^{-1/2} \preceq \underline{\Gamma}_{\min}^{-1/2} \Gamma_{\max} \underline{\Gamma}_{\min}^{-1/2} \right )\\
&\hspace{5mm} \leq \vert\underline{\mathcal{T}}\vert 
\sup_{\omega\in \mathcal{S}^{m-1}}
\P \left \{\frac{\omega^\T \underline{\Gamma}_{\min}^{-1/2} Z^\T \eta}{\lVert Z \underline{\Gamma}_{\min}^{-1/2} \omega\rVert} \geq K/2, 
\; 1 \leq \lVert Z\underline{\Gamma}_{\min}^{-1/2}\omega \rVert^2 \leq \lambda_{\max}(\underline{\Gamma}_{\min}^{-1/2} \Gamma_{\max} \underline{\Gamma}_{\min}^{-1/2})
\right \}.
\end{split}
\end{align}

Similarly to \eqref{eq:thm1_covernum_T}, we can show that
\begin{equation}\label{eq:thm1_covernum_T2}
\log \vert\underline{\mathcal{T}} \vert 
\leq m\log9+(1/2)\log\det(\Gamma_{\max}\underline{\Gamma}_{\min}^{-1})
\leq m \log(38/\alpha)+(1/2)\log\det(\overline{\Gamma}_R\underline{\Gamma}_{R}^{-1}).
\end{equation}

Now it remains to derive a pointwise upper bound on the probability in \eqref{eq:thm1_discrete} for any fixed $\omega\in \mathcal{S}^{m-1}$. Let $\eta_{i,t}$ be the $i$th element of $\eta_t$, and denote \[\eta_{(i)}=(\eta_{i,1},\ldots, \eta_{i,n})^\T.\] 
Note that $\eta=(\eta_{(1)}^\T, \ldots, \eta_{(q)}^\T)^\T$.
Fixing $\omega\in \mathcal{S}^{m-1}$, define $x_t=(x_{1,t},\ldots, x_{q,t})^\T$, where $x_{i,t}=X_t^\T R_i\underline{\Gamma}_{\min}^{-1/2}\omega$, and denote
\[
x_{(i)}=(x_{i,1},\ldots, x_{i,n})^\T=XR_i\underline{\Gamma}_{\min}^{-1/2}\omega.
\] 
Then, we have $\omega^\T \underline{\Gamma}_{\min}^{-1/2} Z^\T = \omega^\T \underline{\Gamma}_{\min}^{-1/2}R^\T (I_q \otimes X^\T)=(x_{(1)}^\T, \ldots, x_{(q)}^\T)$.
As a result,
\[\omega^\T \underline{\Gamma}_{\min}^{-1/2} Z^\T \eta= \sum_{i=1}^q x_{(i)}^\T\eta_{(i)}=\sum_{i=1}^q\sum_{t=1}^{n}x_{i,t} \eta_{i,t}=\sum_{t=1}^{n}x_t^\T \eta_t\]
and
\[\lVert Z \underline{\Gamma}_{\min}^{-1/2} \omega\rVert^2=\sum_{i=1}^q \lVert x_{(i)}\rVert^2=\sum_{t=1}^{n}\lVert x_t \rVert^2.\]
Applying Lemma \ref{lem_A3} to $\{x_t\}$ and $\{\eta_t\}$, with $\beta_-=1$ and $\beta_+=\lambda_{\max}(\underline{\Gamma}_{\min}^{-1/2} \Gamma_{\max} \underline{\Gamma}_{\min}^{-1/2})$, we have
\begin{align}\label{eq:thm1_eq1}
&\P \left \{\frac{\omega^\T \underline{\Gamma}_{\min}^{-1/2} Z^\T \eta}{\lVert Z \underline{\Gamma}_{\min}^{-1/2} \omega\rVert} \geq K/2, 
\; 1 \leq \lVert Z\underline{\Gamma}_{\min}^{-1/2}\omega \rVert^2 \leq \lambda_{\max}(\underline{\Gamma}_{\min}^{-1/2} \Gamma_{\max} \underline{\Gamma}_{\min}^{-1/2})
\right \} \notag\\
&\hspace{5mm}= \P\left \{ \frac{ \sum_{t=1}^{n}x_t^\T \eta_t }{\left (\sum_{t=1}^{n}\lVert x_t \rVert^2\right )^{1/2}}\geq K/2, \; \sum_{t=1}^{n}\lVert x_t \rVert^2 \in [\beta_-,\beta_+]\right \}
\leq \frac{\beta_+}{\beta_-} \exp(-\frac{K^2}{24\sigma^2}).
\end{align}
Moreover, by a method similar to that for \eqref{eq:thm1_covernum_T2}, we can show that
\begin{align}\label{eq:thm1_eq2}
\frac{\beta_+}{\beta_-} 
\leq \det(\underline{\Gamma}_{\min}^{-1/2} \Gamma_{\max} \underline{\Gamma}_{\min}^{-1/2})
= \det(\Gamma_{\max} \underline{\Gamma}_{\min}^{-1})\leq \exp\left \{ m \log\frac{9}{2\alpha}+\log\det(\overline{\Gamma}_R\underline{\Gamma}_{R}^{-1}) \right \}.
\end{align}

Combining \eqref{eq:thm1_discrete}--\eqref{eq:thm1_eq2}, we have
\[
\P \left ( \lVert \mathcal{U}^\T \eta \rVert \geq K, \; \underline{\Gamma}_{\min} \preceq Z^\T Z \preceq \Gamma_{\max}\right)
\leq \exp \left[ 2m \log\frac{14}{\alpha}+\frac{3}{2}\log\det(\overline{\Gamma}_R\underline{\Gamma}_{R}^{-1}) - \frac{K^2}{24\sigma^2}\right ] \leq \delta,
\]
if we choose $K$ as mentioned below \eqref{thm1_part2}. This completes the proof of this theorem.

\subsection{Proof of Proposition \ref{prop1}}
Define the matrices $\Gamma_{\max}, \Gamma_{\min}$ and $\underline{\Gamma}_{\min}$ as in \eqref{eq:Gammas}, and consider the singular value decomposition of $Z$ as in the proof of Theorem \ref{thm1}. Note that 
\begin{equation*}
\widehat{A} - A_*= \{I_q \otimes (\widehat{\theta}-\theta_*)^\T \} \widetilde{R}
= \{I_q \otimes (Z^\dagger\eta)^\T \} \widetilde{R}
= (I_q \otimes \eta^\T \mathcal{U})(I_q \otimes \mathcal{D}^{-}\mathcal{V}^\T) \widetilde{R}.
\end{equation*}
Since $\underline{\Gamma}_{\min}\succ0$, it holds  on the event $\{Z^\T Z \succeq \underline{\Gamma}_{\min}\}$ that
\begin{align*}
\lVert\widehat{A} - A_* \rVert_2
&\leq \lVert (I_q \otimes \mathcal{D}^{-}\mathcal{V}^\T) \widetilde{R}\rVert_2 \lVert\mathcal{U}^\T\eta\rVert
= \left(\lambda_{\max}[\widetilde{R}^\T \{I_q \otimes(Z^\T Z)^{-1}\} \widetilde{R}]\right )^{1/2}\lVert\mathcal{U}^\T\eta\rVert \\
&\leq \left[\lambda_{\max}\{\widetilde{R}^\T (I_q \otimes\underline{\Gamma}_{\min}^{-1}) \widetilde{R}\}\right]^{1/2}\lVert\mathcal{U}^\T\eta\rVert.
\end{align*}
Consequently, by a method similar to that for \eqref{eq:thm1_aeq1}, under Assumption \ref{assum:upper_matrix}, we can show that
\begin{align*}
& \P \left (\lVert\widehat{A} - A_* \rVert_2
\geq K \left[\lambda_{\max}\{\widetilde{R}^\T (I_q \otimes\underline{\Gamma}_{\min}^{-1}) \widetilde{R}\}\right]^{1/2} \right)\\
&\hspace{5mm} 
\leq \P \left ( \lVert \mathcal{U}^\T \eta \rVert \geq K, \; \underline{\Gamma}_{\min} \preceq Z^\T Z \preceq \Gamma_{\max}\right) +\P \left ( \underline{\Gamma}_{\min} \npreceq Z^\T Z \preceq \Gamma_{\max}\right )+\delta
\end{align*}
for any $K>0$. Moreover, similarly to \eqref{eq:thm1_aeq2}, we have
\begin{align*}
\left[\lambda_{\max}\{\widetilde{R}^\T (I_q \otimes\underline{\Gamma}_{\min}^{-1}) \widetilde{R}\}\right]^{1/2} 
&\leq \frac{9}{2\alpha} \left [\frac{\lambda_{\max}\{\widetilde{R}^\T (I_q \otimes\underline{\Gamma}_R^{-1}) \widetilde{R}\}}{n}\right ]^{1/2}\\
&=\frac{9}{2\alpha}\left \{\frac{\lambda_{\max}\left (\sum_{i=1}^{q}R_i \underline{\Gamma}_R^{-1} R_i^\T \right)}{n}\right \}^{1/2}.
\end{align*}
Then, along the same lines of the arguments for Theorem \ref{thm1}, we accomplish the proof of this proposition.

\subsection{Proofs of Lemmas \ref{lem_A1}--\ref{lem_A3}\label{subsec:lem_A_proofs}} 
The covering and discretization results in Lemma \ref{lem_A1} are modified from Lemmas 4.1, D.1 and D.2 in \cite{Simchowitz2018}. For clarity, we rewrite the proofs of Lemma \ref{lem_A1}(i)--(ii) to correct any typographical error in their proofs, and present our own proof of Lemma \ref{lem_A1}(iii).  Lemma \ref{lem_A2} establishes a pointwise lower bound on $X^\T X$ via the \BMSB\ condition, which will be strengthened into a union bound in the proof of Theorem \ref{thm1} via Lemma \ref{lem_A1}(i); see also Proposition 2.5 in the above paper.
Finally, as a multivariate generalization of Lemma 4.2(b) in their paper, Lemma \ref{lem_A3} gives a concentration bound on $ \sum_{t=1}^{n}x_t^\T \eta_t / (\sum_{t=1}^{n}\lVert x_t \rVert^2)^{1/2}$. Note that it is crucial to bound this self-normalized process as a whole, instead of bounding the numerator $\sum_{t=1}^{n}x_t^\T \eta_t$ and the denominator $(\sum_{t=1}^{n}\lVert x_t \rVert^2)^{1/2}$ separately; otherwise, the bound would degrade for slower-mixing processes.

\begin{proof}[Proof of Lemma~\ref{lem_A1}] Note that claim (i) will be used to cover $\mathcal{S}^{m-1}$ in terms of $\Gamma_{\min}$ and $\Gamma_{\max}$ for deriving the union upper bound on $Z^\T Z$ in the proof of Theorem \ref{thm1}.  The corresponding covering number is given in claim (ii), which is larger when $\Gamma_{\max}$ is farther away from $\Gamma_{\min}$ as measured by $\log\det(\Gamma_{\max}\Gamma_{\min}^{-1})$.  Claim (iii) is a discretization result for $\omega^\T Z \nu/\lVert Z^\T \omega\rVert$. 
	
	To prove (i), it is equivalent to show that 
	\begin{equation}\label{eq:lem_A1_eq1}
	\mathcal{E}=\{\inf_{\omega\in\mathcal{T}} \omega^\T Z^\T Z \omega \geq 1 \} \cap \{Z^\T Z \preceq \Gamma_{\max}\} \subseteq \{Z^\T Z \succeq  \Gamma_{\min}/2\}.
	\end{equation}
	Since $\mathcal{T}$ is a $1/4$-net of $\mathcal{S}_{\Gamma_{\min}}$ in the norm $\lVert \Gamma_{\max}^{1/2}(\cdot)\rVert$, on the event $\mathcal{E}$, we have
	\begin{align*}
	1/4 \geq \sup_{\omega\in \mathcal{S}_{\Gamma_{\min}}}\inf_{\upsilon\in\mathcal{T}}\lVert \Gamma_{\max}^{1/2}(\omega-\upsilon)\rVert &\geq \sup_{\omega\in \mathcal{S}_{\Gamma_{\min}}}\inf_{\upsilon\in\mathcal{T}}\lVert Z(\omega-\upsilon)\rVert\\
	& \geq \sup_{\omega\in \mathcal{S}_{\Gamma_{\min}}}\inf_{\upsilon\in\mathcal{T}}\left (\lVert Z \upsilon\rVert-\lVert Z \omega\rVert\right )
	= \inf_{\omega\in\mathcal{T}}\lVert Z \omega\rVert-\inf_{\omega\in \mathcal{S}_{\Gamma_{\min}}}\lVert Z \omega\rVert \\
	&\geq 1 -\inf_{\omega\in \mathcal{S}_{\Gamma_{\min}}}\lVert Z \omega\rVert,
	\end{align*}
	where the second and last inequalities are due to  $Z^\T Z \preceq  \Gamma_{\max}$ and $\inf_{\omega\in\mathcal{T}} \omega^\T Z^\T Z \omega \geq 1$, respectively.  As a result,
	\begin{equation*}
	3/4 \leq \inf_{\omega\in \mathcal{S}_{\Gamma_{\min}}}\lVert Z \omega\rVert = \inf_{\omega\in \mathcal{S}^{m-1}}\lVert Z\Gamma_{\min}^{-1/2}\omega\rVert=\left \{\lambda_{\min} (\Gamma_{\min}^{-1/2}Z^\T Z\Gamma_{\min}^{-1/2})\right \}^{1/2}.
	\end{equation*}
	Therefore,
	$Z^\T Z \succeq (9/16) \Gamma_{\min} \succeq \Gamma_{\min}/2$, i.e., \eqref{eq:lem_A1_eq1} holds.
	
	The proof of claim (ii)  is basically the same as that in \cite{Simchowitz2018}, except for some minor corrections. Note that $\vert\mathcal{T} \vert$ is equal to the covering number of the shell of the  ellipsoid $E=\{\omega\in\mathbb{R}^m: \omega^\T \Gamma_{\max}^{-1/2}\Gamma_{\min}\Gamma_{\max}^{-1/2}\omega\leq 1\}$ in the Euclidean norm. Let $B=\{x\in\mathbb{R}^m:\lVert x\rVert \leq 1\}$ be the unit ball in $\mathbb{R}^m$, and denote by $+$  the Minkowski sum. If $\mathcal{T}$ is a minimal $\epsilon$-net of $\mathcal{S}_{\Gamma_{\min}}$ in the norm $\lVert \Gamma_{\max}^{1/2}(\cdot)\rVert$, then it follows from a standard volumetric argument that
	\begin{align*}
	\vert\mathcal{T} \vert 
	\leq \frac{\text{vol}\{E+(\epsilon/2)B\}}{\text{vol}\left \{(\epsilon/2)B \right \}}
	\leq \frac{\text{vol}\left \{(1+\epsilon/2)E\right \}}{\text{vol}\left \{(\epsilon/2)B \right\}}
	&= \frac{(1+\epsilon/2)^m\text{vol}(E )}{(\epsilon/2)^m\text{vol}(B)}\\
	&=\frac{(1+\epsilon/2)^m}{(\epsilon/2)^m \left \{\det(\Gamma_{\max}^{-1/2}\Gamma_{\min}\Gamma_{\max}^{-1/2} ) \right \}^{1/2}}\\
	&=(2/\epsilon+1)^m\left \{\det(\Gamma_{\min}^{-1}\Gamma_{\max})\right \}^{1/2}.
	\end{align*}
	Taking $\epsilon=1/4$ yields the result in (ii). 
	
	Finally, we prove (iii). First note that since $\Gamma_{\min}\succ0$, we have
	\[
	\sup_{\omega\in \mathcal{S}^{m-1}}\frac{\omega^\T Z^\T \nu}{\lVert Z \omega\rVert}
	=\sup_{\omega\in \mathbb{R}^m\setminus\{0\}}\frac{\omega^\T Z^\T \nu}{\lVert Z \omega\rVert}
	=\sup_{\omega\in \mathbb{R}^m\setminus\{0\}}\frac{\omega^\T \Gamma_{\min}^{-1/2} Z^\T \nu}{\lVert Z \Gamma_{\min}^{-1/2}\omega\rVert}
	=\sup_{\omega\in \mathcal{S}_{\Gamma_{\min}}}\frac{\omega^\T Z^\T \nu}{\lVert Z \omega\rVert}.
	\]
	For a fixed $\nu\in\mathbb{R}^n$, define $\phi:\mathbb{R}^m\setminus\{0\}\rightarrow\mathbb{R}$ by 
	\[
	\phi(\omega)=\frac{\omega^\T Z^\T \nu}{\lVert Z\omega\rVert}.
	\]	
	To prove (iii), we will show that for any $\omega\in\mathcal{S}_{\Gamma_{\min}}$, there exist $\omega_0\in\mathcal{T}$ and $u\in\mathbb{R}^d\setminus\{0\}$ such that
	\begin{equation}\label{eq:lem_A1_eq2}
	\phi(\omega) \leq \phi(\omega_0)+(1/2)\phi(u).
	\end{equation}
	Let 
	\[
	u=\frac{\omega}{\lVert Z \omega\rVert}-\frac{\omega_0}{\lVert Z\omega_0\rVert}.
	\]
	Then, $u\neq 0$ as long as $\omega\neq\omega_0$, and we have
	\[\phi(\omega)-\phi(\omega_0)=u^\T Z^\T \nu = \lVert Z u\rVert \phi(u).\] Therefore, to prove \eqref{eq:lem_A1_eq2}, it suffices to show that
	\begin{equation}\label{eq:lem_A1_eq3}
	\lVert Z u\rVert \leq 1/2.
	\end{equation}
	
	Note that 
	\[
	Z u=\frac{Z(\omega-\omega_0)}{\lVert Z \omega\rVert}+\frac{Z\omega_0}{\lVert Z \omega_0\rVert}\frac{\lVert Z \omega\rVert-\lVert Z \omega_0\rVert}{\lVert Z \omega\rVert}.
	\]
	As a result,
	\begin{equation}\label{eq:lem_A1_eq4}
	\lVert Z u\rVert \leq \frac{2\lVert Z(\omega-\omega_0)\rVert }{\lVert Z \omega\rVert} 
	\leq \frac{2\lVert Z(\omega-\omega_0)\rVert }{\inf_{\omega\in \mathcal{S}_{\Gamma_{\min}}}\lVert Z \omega\rVert}.
	\end{equation}
	
	Since $0\prec \Gamma_{\min}\preceq Z^\T Z$, we have
	\[
	\inf_{\omega\in \mathcal{S}_{\Gamma_{\min}}}\lVert Z \omega\rVert = \inf_{\omega\in \mathcal{S}^{m-1}}\lVert Z\Gamma_{\min}^{-1/2}\omega\rVert=\left \{\lambda_{\min} (\Gamma_{\min}^{-1/2}Z^\T Z\Gamma_{\min}^{-1/2})\right \}^{1/2}\geq 1. 
	\]
	Moreover, since $Z^\T Z \preceq \Gamma_{\max}$ and $\mathcal{T}$ is a $1/4$-net of $\mathcal{S}_{\Gamma_{\min}}$ in $\lVert \Gamma_{\max}^{1/2}(\cdot)\rVert$, there is $\omega_0\in\mathcal{T}$ such that
	\[
	\lVert Z(\omega-\omega_0)\rVert \leq \lVert \Gamma_{\max}^{1/2}(\omega-\omega_0)\rVert \leq 1/4.
	\]
	Combining the results above with \eqref{eq:lem_A1_eq4}, we have \eqref{eq:lem_A1_eq3}, and hence \eqref{eq:lem_A1_eq2}. Taking the supremum with respect to $\omega\in \mathcal{S}_{\Gamma_{\min}}$ on both sides of \eqref{eq:lem_A1_eq2}, we have
	\[
	\sup_{\omega\in \mathcal{S}_{\Gamma_{\min}}}\frac{\omega^\T Z^\T \nu}{\lVert Z \omega\rVert} 
	\leq  \max_{\omega_0\in\mathcal{T}}\frac{\omega_0^\T Z^\T \nu}{\lVert Z \omega_0\rVert}+\frac{1}{2}\sup_{u\in \mathbb{R}^d\setminus\{0\}}\frac{u^\T Z^\T \nu}{\lVert Z u\rVert}
	= \max_{\omega\in\mathcal{T}}\frac{\omega^\T Z^\T \nu}{\lVert Z \omega\rVert}+\frac{1}{2}\sup_{\omega\in \mathcal{S}_{\Gamma_{\min}}}\frac{\omega^\T Z^\T \nu}{\lVert Z \omega\rVert},
	\]
	which yields the inequality in (iii).
\end{proof}

\begin{proof}[Proof of Lemma~\ref{lem_A2}]
	This lemma directly follows from Proposition 2.5 of \cite{Simchowitz2018}, where  the Chernoff bound technique is applied to lower bound the Gram matrix via aggregating all the $\lfloor n/k \rfloor$ blocks of size $k$.
\end{proof}

\begin{proof}[Proof of Lemma~\ref{lem_A3}] Along the lines of the proof of Lemma 4.2 in \cite{Simchowitz2018}, we can show that  the left-hand side of \eqref{eq:lem_martingale} is bounded above by $\log\lceil\beta_+/\beta_-\rceil \exp\{-\gamma^2/(6\sigma^2)\}$. Then \eqref{eq:lem_martingale} follows from the fact that $\log\lceil x \rceil < \log(1+x) \leq x$ for $x>0$.
\end{proof}

\section{Proofs of Lemmas \ref{lem_bmsb}--\ref{lem_gram}\label{sec:S2}}
\subsection{Proof of Lemma \ref{lem_bmsb}}
First note that for any $s\in\mathbb{Z}$ and positive integer $t$, 
\[
X_{s+t}=\eta_{s+t-1}+A_*\eta_{s+t-2}+\cdots+A_*^{t-1}\eta_{s}+A_*^tX_s=\sum_{\ell=0}^{t-1}A_*^{\ell}\eta_{s+t-\ell-1}+A_*^tX_s,
\]
where, by Assumption  \ref{assum:var}(ii), $\sum_{\ell=0}^{t-1}A_*^{\ell}\eta_{s+t-\ell-1}$ is independent of $\mathcal{F}_s$, and $A_*^tX_s\in \mathcal{F}_s$.
Then, for any $\omega\in\mathcal{S}^{d-1}$, we have 
\[
\frac{\omega^\T X_{s+t}}{\sigma(\omega^\T \Gamma_t\omega)^{1/2}}=S_{\omega}+c_\omega,
\]
where $c_\omega=\omega^\T A_*^tX_s/\{\sigma(\omega^\T \Gamma_t\omega)^{1/2}\}$, and $S_{\omega}$ can be written as a weighted sum of real-valued independent random variables,
\[
S_{\omega}=\frac{\omega^\T  \sum_{\ell=0}^{t-1}A_*^{\ell}\eta_{s+t-\ell-1}}{\sigma(\omega^\T \Gamma_t\omega)^{1/2}}=\sum_{\ell=0}^{t-1} a_\ell e_{\ell},
\]
with 
\[
a_\ell=\left \{\frac{\omega^\T  A_*^{\ell} (A_*^\T)^{\ell} \omega}{\omega^\T \Gamma_t\omega}\right \}^{1/2}, \quad e_{\ell}=\frac{\omega^\T A_*^{\ell}\eta_{s+t-\ell-1}}{\sigma\{\omega^\T  A_*^{\ell} (A_*^\T)^{\ell} \omega\}^{1/2}}.
\]
Notice that $\sum_{\ell=0}^{t-1}a_\ell^2=1$, and $e_0, \dots, e_{t-1}$ are real-valued independent random variables. Moreover, by Assumption \ref{assum:var}(iii),  the density of each $e_\ell$ is bounded by $C_0$ almost everywhere. Applying Theorem 1.2 in \cite{Rudelson_Vershynin2015}, it follows that the density of $S_{\omega}$ is bounded by $\surd{2} C_0$ almost everywhere. In addition, $S_{\omega}$ is independent of $\mathcal{F}_s$, and  $c_\omega\in\mathcal{F}_s$. Therefore,
\begin{align}\label{eq:lem1_eq1}
\P\left \{\vert \omega^\T X_{s+t}\vert \geq \sigma (\omega^\T \Gamma_{t}\omega)^{1/2} (4C_0)^{-1}\mid \mathcal{F}_s \right \} &= \P\left \{\vert S_{\omega}+c_\omega \vert \geq (4C_0)^{-1}\mid \mathcal{F}_s \right \} \notag \\
& = 1- \P\left \{\vert S_{\omega}+c_\omega \vert \leq (4C_0)^{-1}\mid \mathcal{F}_s \right \} \notag\\
& \geq 1-\surd{2}/2>0.
\end{align}

For any integer $k\geq 1$, by \eqref{eq:lem1_eq1} and the fact that $\Gamma_{t}\succeq \Gamma_{k}$ for  $t\geq k$, we have
\begin{align*}
&\frac{1}{2k}\sum_{t=1}^{2k}\P \left \{\vert \omega^\T X_{s+t}\vert \geq \sigma (\omega^\T \Gamma_{k}\omega)^{1/2} (4C_0)^{-1} \mid \mathcal{F}_s \right \}\\
&\hspace{5mm}
\geq \frac{1}{2k}\sum_{t=k}^{2k}\P \left \{\vert \omega^\T X_{s+t}\vert \geq  \sigma (\omega^\T \Gamma_{k}\omega)^{1/2} (4C_0)^{-1} \mid \mathcal{F}_s \right \}\\
&\hspace{5mm}
\geq \frac{1}{2k}\sum_{t=k}^{2k}\P \left \{\vert \omega^\T X_{s+t}\vert \geq \sigma (\omega^\T \Gamma_{t}\omega)^{1/2} (4C_0)^{-1} \mid \mathcal{F}_s \right \}\\
&\hspace{5mm} \geq \frac{(2k-k+1)(1-\surd{2}/2)}{2k} > \frac{1}{10}.
\end{align*}
Choosing $\Gamma_{\smb} = \sigma^2 \Gamma_{k}/ (4C_0)^2$, we accomplish the proof of this lemma.

\subsection{Proof of Lemma \ref{lem_gram1}}
Since $E(X^\T X)=\sigma^2\sum_{t=1}^{n}\Gamma_{t}\preceq \sigma^2 n\Gamma_{n}$, we have
\[
E(Z^\T Z)=R^\T \{I_d\otimes E(X^\T X)\} R\leq \sigma^2 n R^\T (I_d\otimes \Gamma_{n}) R.
\]
Then, with $\overline{\Gamma}_R = (\sigma^2 m/\delta) R^\T (I_d\otimes \Gamma_n) R$, it follows from the Markov inequality that
\begin{align*}
\P(Z^\T Z \npreceq n\overline{\Gamma}_R)&=\P \left [\lambda_{\max}\{ (n\overline{\Gamma}_R)^{-1/2} Z^\T Z (n\overline{\Gamma}_R)^{-1/2}\}  \geq1 \right ]\\
& \leq E\left [\lambda_{\max}\{ (n\overline{\Gamma}_R)^{-1/2} Z^\T Z (n\overline{\Gamma}_R)^{-1/2}\}  \right ]\\
& \leq \text{tr}\left \{ (n\overline{\Gamma}_R)^{-1/2} E(Z^\T Z) (n\overline{\Gamma}_R)^{-1/2}\right \} \\
& \leq \text{tr}\{(\delta/m)I_m \}=\delta,
\end{align*}
which completes the proof of this lemma.
Note that the factor of $m$ in the definition of  $\overline{\Gamma}_R$  is a consequence of upper bounding  $\lambda_{\max}(\cdot)$ by $\text{tr}(\cdot)$. 

\subsection{Proof of Lemma \ref{lem_gram}}
Recall that $R=(R_1^\T, \dots, R_d^\T)^\T$, where each $R_i$ is a $d\times m$ block. For simplicity,  denote $Q_i=XR_i (R_i^\T R_i)^{-1/2}=(Q_{i,1}, \dots, Q_{i,n})^\T\in\mathbb{R}^{n\times m}$ with $i=1,\dots, d$, where $Q_{i,t}= (R_i^\T R_i)^{-1/2}R_i^\T X_t\in\mathbb{R}^m$. To prove this lemma, it suffices to verify the following two results:
\begin{equation}\label{eq:lem3_eq1}
\P(Z^\T Z \npreceq n\overline{\Gamma}_R) \leq \sum_{i=1}^{d}P \{  \lVert Q_i^\T Q_i - E(Q_i^\T Q_i) \rVert_2 > n\sigma^2\xi \},
\end{equation}
where $\overline{\Gamma}_{R}=\sigma^2 R^\T (I_d\otimes \Gamma_n) R + \sigma^2 \xi R^\T R$, and
\begin{equation}\label{eq:lem3_eq2}
\P \{  \lVert Q_i^\T Q_i - E(Q_i^\T Q_i) \rVert_2 > n\sigma^2\xi \} \leq \delta/d \quad (i=1,\ldots, d).
\end{equation}

We first prove \eqref{eq:lem3_eq1}. Note that $Z^\T Z=R^\T \{I_d\otimes (X^\T X)\} R$ and $E(Z^\T Z)=R^\T \{I_d\otimes E(X^\T X)\} R=\sigma^2 R^\T \{I_d\otimes \sum_{t=1}^{n}\Gamma_t\} R$. Then, since $n\Gamma_n\succeq \sum_{t=1}^{n}\Gamma_t$, we have
\[
n\overline{\Gamma}_R\succeq \sigma^2 R^\T \{I_d\otimes \sum_{t=1}^{n}\Gamma_t\} R+ n\sigma^2\xi R^\T R
=E(Z^\T Z) + n \sigma^2\xi R^\T R.
\]
As a result,
\begin{align}\label{eq:lem3_eq3}
\P(Z^\T Z \preceq n\overline{\Gamma}_R)	&=\P\{Z^\T Z - E(Z^\T Z)\preceq n\sigma^2\xi R^\T R\} \notag \\
&=\P\left [ \sum_{i=1}^{d}R_i^\T \{X^\T X - E(X^\T X)\}R_i \preceq n \sigma^2\xi \sum_{i=1}^{d}R_i^\T R_i \right ] \notag\\
&\geq \P \left \{ \bigcap_{i=1}^d \left [  R_i^\T \{X^\T X - E(X^\T X)\}R_i \preceq n \sigma^2\xi R_i^\T R_i\right ] \right \} \notag\\
&\geq 1- \sum_{i=1}^{d}\P \left[  R_i^\T \{X^\T X - E(X^\T X)\}R_i \npreceq n \sigma^2\xi R_i^\T R_i\right ].
\end{align}
Moreover, for $i=1,\dots, d$, we have
\begin{align*}
&\P \left[  R_i^\T \{X^\T X - E(X^\T X)\}R_i \preceq n\sigma^2\xi R_i^\T R_i\right ]\\
&\hspace{5mm}=\P \left[  (R_i^\T R_i)^{-1/2} R_i^\T \{X^\T X - E(X^\T X)\}R_i (R_i^\T R_i)^{-1/2} \preceq n \sigma^2\xi I_m\right ]\\
&\hspace{5mm}=\P\{ Q_i^\T Q_i - E(Q_i^\T Q_i) \preceq n\sigma^2\xi I_m\}\\
&\hspace{5mm} \geq \P \{  \lVert Q_i^\T Q_i - E(Q_i^\T Q_i) \rVert_2 \leq n\sigma^2\xi \},
\end{align*}
which implies
\begin{equation}\label{eq:lem3_eq4}
\P \left[  R_i^\T \{X^\T X - E(X^\T X)\}R_i \npreceq n \sigma^2\xi R_i^\T R_i\right ] \leq  \P \{  \lVert Q_i^\T Q_i - E(Q_i^\T Q_i) \rVert_2 > n \sigma^2\xi \}.
\end{equation}
Combining \eqref{eq:lem3_eq3} and \eqref{eq:lem3_eq4}, we accomplish the proof of \eqref{eq:lem3_eq1}.

Next we prove \eqref{eq:lem3_eq2}. Let $\mathcal{T}_0$ be a minimal $(1/4)$-net of the sphere $\mathcal{S}^{m-1}$ in the Euclidean norm. It follows from a standard volumetric argument that $\vert \mathcal{T}_0\vert \leq 9^m$. Moreover, by Lemma 5.4 in \cite{Vershynin2012}, we have
\begin{equation}\label{eq:lem3_eq5}
\lVert Q_i^\T Q_i - E(Q_i^\T Q_i) \rVert_2 \leq 2 \max_{\omega\in\mathcal{T}_0}  \vert \omega^\T \{Q_i^\T Q_i - E(Q_i^\T Q_i) \} \omega \vert. 
\end{equation}
Furthermore, since $\{\eta_t\}$ are normal, $\vect(X^\T)=(X_1^\T, \dots, X_n^\T)^\T$ follows the multivariate normal distribution with mean zero and covariance matrix $\Sigma_X$. Then, for any $\omega\in \mathcal{S}^{m-1}$, we have
\[
Q_i \omega =(Q_{i,1}^\T \omega, \dots, Q_{i,n}^\T \omega)^\T= \left [I_n\otimes \{\omega^\T (R_i^\T R_i)^{-1/2}R_i^\T\}\right ]\vect(X^\T)\sim N(0, \Sigma_{\omega}),
\] 
where
\begin{equation*}
\Sigma_{\omega}=\left [I_n\otimes \{\omega^\T (R_i^\T R_i)^{-1/2}R_i^\T\}\right ] \Sigma_X
\left [I_n\otimes \{R_i (R_i^\T R_i)^{-1/2} \omega\}\right ],
\end{equation*}
and hence there exists $z\sim N(0, I_n)$ such that $\omega^\T Q_i^\T Q_i \omega=z^\T \Sigma_{\omega} z$. As a result, it follows from the Hanson-Wright inequality \citep{Vershynin2018} that, for every $\xi>0$, 
\begin{align}\label{eq:lem3_eq6}
\P \left [\vert \omega^\T \{Q_i^\T Q_i - E(Q_i^\T Q_i) \} \omega \vert > n\sigma^2 \xi/2 \right ] 
& = \P \left [\vert z^\T  \Sigma_{\omega} z - E(z^\T  \Sigma_{\omega} z)  \vert > n \sigma^2 \xi/2 \right ] \notag \\
& \leq 2 \exp \left [-\frac{1}{C_1} \min\left \{ \left (\frac{n\sigma^2\xi}{2\lVert \Sigma_\omega \rVert_F}\right )^2, \frac{n\sigma^2 \xi}{2\lVert \Sigma_\omega \rVert_2}\right \}\right ],
\end{align}
where $C_1 >0$ is a universal constant. In view of \eqref{eq:lem3_eq5} and \eqref{eq:lem3_eq6}, we have
\begin{align*}
\P \{ \lVert Q_i^\T Q_i - E(Q_i^\T Q_i) \rVert_2 > n \sigma^2 \xi \} 
& \leq  \P \left [ \max_{\omega\in\mathcal{T}_0}  \vert \omega^\T  \{Q_i^\T Q_i - E(Q_i^\T Q_i) \} \omega \vert > n\sigma^2\xi/2 \right ]\\
& \leq \vert \mathcal{T}_0 \vert \; \P \left [\vert \omega^\T  \{Q_i^\T Q_i - E(Q_i^\T Q_i) \} \omega \vert > n\sigma^2\xi/2 \right ]\\
& \leq  2(9^m)  \exp \left [-\frac{1}{C_1} \min\left \{ \left (\frac{n\sigma^2\xi}{2\lVert \Sigma_\omega \rVert_F}\right )^2, \frac{n\sigma^2 \xi}{2\lVert \Sigma_\omega \rVert_2}\right \}\right ] \\
& \leq \delta/d
\end{align*}
as long as
\begin{equation}\label{eq:lem3_eq7}
\xi \geq \frac{2}{n\sigma^2} \left [ \left \{\psi(m,d,\delta) \right \}^{1/2} \lVert \Sigma_\omega \rVert_F +  \psi(m,d,\delta)\lVert \Sigma_\omega \rVert_2\right ],
\end{equation}
where $\psi(m,d,\delta)=C_1 \{m\log 9+\log d+\log(2/\delta)\}$. 

To prove \eqref{eq:lem3_eq2}, it now remains to choose $\xi$ such that \eqref{eq:lem3_eq7} holds. Note that 
\begin{equation}\label{eq:lem3_eq8}
\lVert \Sigma_\omega \rVert_2 \leq \lambda_{\max} \left [I_n \otimes \{\omega^\T (R_i^\T R_i)^{-1/2}R_i^\T R_i (R_i^\T R_i)^{-1/2}\omega\}\right ] \lVert \Sigma_X \rVert_2= \lVert \Sigma_X \rVert_2.
\end{equation}
Moreover, since 
\begin{align*}
\tr(\Sigma_{\omega}) & =\sigma^2 \sum_{t=1}^{n} \omega^\T (R_i^\T R_i)^{-1/2} R_i^\T \Gamma_{t} R_i (R_i^\T R_i)^{-1/2} \omega \\
& \leq \sigma^2 n \lambda_{\max}(\Gamma_n) \lambda_{\max}\left \{(R_i^\T R_i)^{-1/2} R_i^\T R_i (R_i^\T R_i)^{-1/2}\right \} \\
& = \sigma^2 n \lambda_{\max}(\Gamma_n),
\end{align*}
in light of \eqref{eq:lem3_eq8}, we have
\begin{equation}\label{eq:lem3_eq9}
\lVert \Sigma_\omega \rVert_F = \{\tr(\Sigma_{\omega}^2)\}^{1/2} \leq \{\lVert \Sigma_\omega \rVert_2 \tr(\Sigma_{\omega})\}^{1/2} \leq \left \{\sigma^2 n \lambda_{\max}(\Gamma_n) \lVert \Sigma_X \rVert_2 \right \}^{1/2}.
\end{equation}
Replacing $\lVert \Sigma_\omega \rVert_2$ and $\lVert \Sigma_\omega \rVert_F$ in \eqref{eq:lem3_eq7} by their upper bounds in \eqref{eq:lem3_eq8} and \eqref{eq:lem3_eq9}, respectively, it follows that \eqref{eq:lem3_eq2} holds if we choose $\xi$ as in \eqref{xi} in the main paper.
The proof of this lemma is complete.

\section{Proofs of Equation \eqref{logdet} and Proposition \ref{prop_G}\label{sec:S3}}
\subsection{Proof of Equation \eqref{logdet}}
The proof of Equation \eqref{logdet} relies on the following lemma:

\begin{lemma}\label{lem_logdet}
	Let $A$ and $B$ be  $m\times m$ symmetric positive definite matrices such that $B^{1/2}AB^{1/2}\succeq I_m$. For any $\xi>0$, it holds 
	$\log\det(AB+\xi I_m)\leq m\log \{2\max(1,\xi)\}+\log\det(AB)$.
\end{lemma}

\begin{proof}[Proof of Lemma~\ref{lem_logdet}]
	Let $\lambda_1\geq \lambda_2\geq \cdots \geq \lambda_m\geq 1$ be the eigenvalues of $B^{1/2}AB^{1/2}$. For any $\xi>0$, it can be readily shown that $\lambda_i+\xi\leq 2\max(1,\xi) \lambda_i$ for all $i$. Moreover, by Theorem 1.3.20 in \cite{Horn_Johnson1985}, $AB$ and $B^{1/2}AB^{1/2}$ have the same nonzero eigenvalues.
	Thus, 
	\begin{align*}
	\log\det(AB+\xi I_m)=\sum_{i=1}^{m}\log(\lambda_i+\xi)&\leq \sum_{i=1}^{m} \left [\log \{2\max(1,\xi)\}+\log\lambda_i\right ]\\
	&= m\log \{2\max(1,\xi)\}+\log\det(AB).
	\end{align*}
	The proof of Lemma \ref{lem_logdet} is complete.
\end{proof}

Now we prove Equation \eqref{logdet}. First note that 
\begin{equation}\label{eq:logdet1}
\log \det\{\overline{\Gamma}_R (\sigma^2 R^\T R)^{-1} (4C_0)^2\}=\log \det\{\overline{\Gamma}_R (\sigma^2 R^\T R)^{-1}\}+2m\log(4C_0).
\end{equation}

If $\overline{\Gamma}_{R}=\overline{\Gamma}_{R}^{(1)}=\sigma^2 m R^\T (I_d\otimes \Gamma_n) R/\delta$,  it is easy to see that
\begin{equation}\label{eq:logdet2}
\log \det\{\overline{\Gamma}_R (\sigma^2 R^\T R)^{-1}\}=m\log(m/\delta)+\kappa,
\end{equation}
where $\kappa =\log \det\left \{R^\T (I_d\otimes \Gamma_n) R (R^\T R)^{-1} \right \}$.

On the other hand, if $\overline{\Gamma}_{R}=\overline{\Gamma}_{R}^{(2)}=\sigma^2 R^\T (I_d\otimes \Gamma_n) R + \sigma^2 \xi R^\T R$,  we have
\[
\log \det\{\overline{\Gamma}_R (\sigma^2 R^\T R)^{-1}\}=\log \det\left \{R^\T (I_d\otimes \Gamma_n) R (R^\T R)^{-1}+\xi I_m \right \},
\]
Note that $(R^\T R)^{-1/2} R^\T (I_d\otimes \Gamma_n) R (R^\T R)^{-1/2}\succeq I_m$, since $\Gamma_n\succeq I_d$. Then, applying Lemma \ref{lem_logdet} with $A=R^\T (I_d\otimes \Gamma_n) R$ and $B=(R^\T R)^{-1}$, we have
\begin{equation}\label{eq:logdet3}
\log \det\{\overline{\Gamma}_R (\sigma^2 R^\T R)^{-1}\}\leq m\log \{2\max(1,\xi)\}+\kappa.
\end{equation}
Combining \eqref{eq:logdet1}--\eqref{eq:logdet3}, we accomplish the proof of Equation \eqref{logdet}.

\subsection{Proof of Proposition \ref{prop_G}}
Proposition \ref{prop_G} is a direct consequence of Equations \eqref{logdet1} and \eqref{logdet} in the main paper and the  upper bounds of $\kappa$ and $\xi$ in  Lemmas \ref{lem_kappa} and \ref{lem_xi} below.  Note that the proofs of Lemmas \ref{lem_kappa} and \ref{lem_xi} rely crucially on the intermediate results on upper bounds of $\lambda_{\max}(\Gamma_n)$ and $\lVert \Sigma_X \rVert_2$ as given in 
Lemmas \ref{lem_Gamma_n} and \ref{lem_Sigma_X} below, respectively.  The proofs of Lemmas \ref{lem_kappa}--\ref{lem_Sigma_X}  are collected in $\S$ \ref{subsec:lem_B_proofs}.

Recall that $\Sigma_X=[E(X_tX_s^\T)]_{1\leq t,s\leq n}$, where $E(X_tX_s^\T)=\sigma^2  A_*^{t-s} \Gamma_s$ for $1\leq s\leq t\leq n$ under Assumptions \ref{assum:var}(i) and (ii),
$\kappa =\log \det\left \{R^\T (I_d\otimes \Gamma_n) R (R^\T R)^{-1} \right \}$, and 
\begin{equation*}
\xi=\xi(m,d,n,\delta) =2 \left \{\frac{\lambda_{\max}(\Gamma_n)  \psi(m,d,\delta) \lVert \Sigma_X \rVert_2}{\sigma^2 n} \right \}^{1/2} +  \frac{2 \psi(m,d,\delta)\lVert \Sigma_X \rVert_2}{\sigma^2 n},
\end{equation*}
where $\psi(m,d,\delta)=C_1 \{m\log 9+\log d+\log(2/\delta)\}$, and $C_1>0$ is a universal constant.

As in the main paper, let the Jordan decomposition of $A_*$ be  $A_*=SJS^{-1}$, where $J$ has $L$ blocks with sizes $1\leq b_1,\ldots, b_L \leq d$, and both $J$ and $S$ are $d\times d$ complex matrices. Let 
$b_{\max}=\max_{1\leq \ell \leq L} b_\ell,$
and denote the condition number of $S$ by
$\cond(S)=\left \{\lambda_{\max}(S^* S)/\lambda_{\min}(S^*S)\right \}^{1/2}$,
where $S^*$ is the conjugate transpose of $S$.

\begin{lemma}\label{lem_kappa} For any $A_*\in\mathbb{R}^{d\times d}$, under Assumption \ref{assum:explosive}, 
	\[\kappa \lesssim  m \left[\log \{d\cond(S)\}+ b_{\max} \log n\right].\] 
	Moreover, if Assumption \ref{assum:stable1} holds, then $\kappa \lesssim m$.
\end{lemma}

\begin{lemma}\label{lem_xi} For any $A_*\in\mathbb{R}^{d\times d}$, under Assumption \ref{assum:explosive},
	\[\log \xi \lesssim  \log \{d\cond(S)/\delta\}+ b_{\max} \log n.\] 
	Moreover, if Assumption \ref{assum:stable2} holds and $n \gtrsim m +\log(d/\delta)$, then $\xi \lesssim 1$.
\end{lemma}

\begin{lemma}\label{lem_Gamma_n}
	For any $A_*\in\mathbb{R}^{d\times d}$, under Assumption \ref{assum:explosive}, 
	\[
	\lambda_{\max}(\Gamma_n) \lesssim db_{\max} n^{2b_{\max}-1} \{\cond(S)\}^2.
	\]
	Moreover, if Assumption \ref{assum:stable1} holds, then $\lambda_{\max}(\Gamma_n) \lesssim 1$.
\end{lemma}
\begin{lemma}\label{lem_Sigma_X}
	For any $A_*\in\mathbb{R}^{d\times d}$, under Assumption \ref{assum:explosive}, 
	\[
	\lVert \Sigma_X \rVert_2 \lesssim  d n \sigma^2 \lambda_{\max}(\Gamma_n),
	\]
	where $\Sigma_X$ is the symmetric $dn\times dn$ matrix with its $(t,s)$th $d\times d$ block being $\sigma^2  A_*^{t-s} \Gamma_s$ for $1\leq s \leq t\leq n$.
	Moreover, if Assumption \ref{assum:stable2} holds, then $\lVert \Sigma_X \rVert_2\lesssim \sigma^2$.
\end{lemma}

\subsection{Proofs of Lemmas \ref{lem_kappa}--\ref{lem_Sigma_X}\label{subsec:lem_B_proofs}} 

\begin{proof}[Proof of Lemma~\ref{lem_kappa}]
	Note that
	\begin{align*}
	\kappa&=\log \left [\det\left \{R^\T (I_d\otimes \Gamma_{n}) R \right \}\det\{(R^\T R)^{-1}\} \right ]\\
	& \leq \log \left [ \lambda_{\max}^m(\Gamma_n)\det(R^\T R)\det\{(R^\T R)^{-1}\} \right ]= m\log  \lambda_{\max}(\Gamma_n).
	\end{align*}
	Thus, the upper bound of $\kappa$ follows directly from Lemma \ref{lem_Gamma_n}.
\end{proof}

\begin{proof}[Proof of Lemma~\ref{lem_xi}]
	First consider the case under Assumption \ref{assum:explosive}. Note that $m\leq dn$; see the paragraph below \eqref{eq:ols_theta} in the main paper. Then
	\[
	\log \psi(m,d,\delta)\lesssim  \log m+\log \log (d/\delta) \lesssim  \log (d/\delta)+\log n.
	\]
	This, together with Lemmas \ref{lem_Sigma_X} and \ref{lem_kappa}, leads to the upper bound of $\log \xi$ under Assumption \ref{assum:explosive}.
	
	Suppose that Assumption \ref{assum:stable2} holds. Then it follows from Lemmas \ref{lem_Sigma_X} and \ref{lem_kappa} that
	\[
	\xi \lesssim \{\psi(m,d,\delta)/n\}^{1/2}+\psi(m,d,\delta)/n.
	\]
	Moreover, if $n \gtrsim m +\log(d/\delta)$, then $\psi(m,d,\delta)/n \lesssim 1$, and consequently $\xi \lesssim 1$.  The proof of this lemma is complete.
\end{proof}

\begin{proof}[Proof of Lemma~\ref{lem_Gamma_n}]
	We first prove the conclusion under Assumption \ref{assum:explosive}. 
	By the Jordan normal form of $A_*$, we have 
	\[
	\Gamma_{n}=S\sum_{s=0}^{n-1}J^s S^{-1} (S^{-1})^* (J^*)^s S^*\preceq \{\lambda_{\min}(S^*S)\}^{-1} S \sum_{s=0}^{n-1}J^s (J^*)^s S^*.
	\]
	Hence
	\begin{equation}\label{eq:Gamma_n_eq1}
	\lambda_{\max}(\Gamma_n) \leq \{\cond(S)\}^2 \lambda_{\max}\left \{\sum_{s=0}^{n-1}J^s (J^*)^s \right \}.
	\end{equation}
	
	For $\ell=1,\ldots, L$,  denote by $J_\ell$ the $\ell$th block of $J$ with size $b_\ell$ and diagonal entries $\lambda_\ell$. Note that the $\ell$th block of the block diagonal matrix $\sum_{s=0}^{n-1}J^s (J^*)^s \in\mathbb{R}^{d\times d}$ is $B_\ell = \sum_{s=0}^{n-1}J_\ell^s (J_\ell^*)^s \in\mathbb{R}^{b_\ell\times b_\ell}$. 
	Moreover, the $(i,j)$th entry of $J_\ell^s$ is 
	\[
	(J_\ell^s)_{ij}=\begin{cases}
	\binom{s}{j-i} \lambda_\ell^{s-(j-i)}, &\text{if}\quad 1\leq i\leq j\leq \min(i+s, b_\ell)\\
	0, &\text{otherwise}
	\end{cases},
	\]
	where $\binom{0}{0}=1$. Then the $i$th diagonal entry of $B_\ell$ is
	\begin{equation}\label{eq_Bl}
	(B_\ell)_{ii}= \sum_{s=0}^{n-1} \sum_{j=1}^{b_\ell} (J_\ell^s)_{ij}^2 = \sum_{s=0}^{n-1} \sum_{j=i}^{\min(i+s, b_\ell)} \left \{\binom{s}{j-i} |\lambda_\ell|^{s-(j-i)}\right \}^2.
	\end{equation}
	
	By Assumption \ref{assum:explosive}, $|\lambda_\ell| \leq \rho(A_*) \leq 1+c/n$ for $c>0$. Thus
	\[
	|\lambda_\ell|^{2\{s-(j-i)\}} \leq \left (1+\frac{c}{n}\right )^{2n}.
	\]
	Note that $(1+c/n)^{2n}$ monotonically increases to $\exp(2c)$ as $n\rightarrow\infty$, which implies that  $|\lambda_\ell|^{2\{s-(j-i)\}}$ is uniformly bounded by a universal constant $C_2>0$. Moreover, for $j-i\leq b_\ell-1$ and $s<n$,
	$\binom{s}{j-i}$ in \eqref{eq_Bl} is uniformly bounded above by $n^{b_\ell-1}$. As a result, for any $1\leq i\leq b_\ell$ and $1\leq \ell \leq L$, we have
	\[
	(B_\ell)_{ii} \leq C_2 b_\ell n^{2b_\ell-1}.
	\]
	Notice that the diagonal entries of $\sum_{s=0}^{n-1}J^s (J^*)^s$ are $\{(B_\ell)_{ii}\}_{1\leq i\leq b_\ell, 1\leq \ell \leq L}$. Therefore
	\begin{equation}\label{eq:Gamma_n_eq2}
	\lambda_{\max}\left \{\sum_{s=0}^{n-1}J^s (J^*)^s \right \} \leq d \max_{1\leq i\leq b_\ell, 1\leq \ell \leq L}(B_\ell)_{ii} \leq C_2 d b_{\max} n^{2b_{\max}-1}.
	\end{equation}
	Combining \eqref{eq:Gamma_n_eq1} and \eqref{eq:Gamma_n_eq2}, we obtain the upper bound of $\lambda_{\max}(\Gamma_n)$ under Assumption \ref{assum:explosive} as stated in this lemma.
	
	Next we verify the conclusion under Assumption \ref{assum:stable1}. Since $\rho(A_*)\leq \bar{\rho} <1$, we have $\Gamma_n \preceq \Gamma_{\infty}=\sum_{s=0}^{\infty}A_*^s (A_*^\T)^s<\infty$. Note that $\rho(A_*)=\lim_{s\rightarrow\infty} \lVert A_*^s \rVert_2^{1/s}$. Thus, for any $\epsilon>0$, there exists a positive integer $n_0=n_0(\epsilon)$ such that $\lVert A_*^s \rVert_2^{1/s}<\rho(A_*)+\epsilon$ for all $s\geq n_0$. Taking $\epsilon=\{1-\rho(A_*)\}/2$, we have $\rho(A_*)+\epsilon=(1+\bar{\rho})/2<1$. As a result,
	\begin{align*}
	\lambda_{\max}(\Gamma_n) \leq \lambda_{\max}(\Gamma_{\infty}) \leq \sum_{s=0}^{\infty}\lVert A_*^s \rVert_2^2 &\leq \sum_{s=0}^{n_0-1}\lVert A_* \rVert_2^{2s}+\sum_{s=n_0}^{\infty} \left (\frac{1+\bar{\rho}}{2}\right)^{2s} \\
	&\leq \sum_{s=0}^{n_0-1} C^{2s}+\left \{1-\left (\frac{1+\bar{\rho}}{2}\right )^2\right \}^{-1},
	\end{align*}
	where the last upper bound is a fixed constant.  The proof of this lemma is complete.
\end{proof}

\begin{proof}[Proof of Lemma~\ref{lem_Sigma_X}]
	The result under Assumption \ref{assum:explosive} is straightforward, since
	\[
	\lVert \Sigma_X \rVert_2 = \lambda_{\max}(\Sigma_X)	\leq \tr(\Sigma_X)=\sigma^2\sum_{t=1}^{n}\tr(\Gamma_t)\leq n\sigma^2 \tr(\Gamma_n)\leq d n \sigma^2 \lambda_{\max}(\Gamma_n).
	\]
	
	However, showing that  $\lVert \Sigma_X \rVert_2$ is bounded by a fixed constant proportional to $\sigma^2$ under Assumption \ref{assum:stable2} requires  a much more delicate argument. This is largely because $\lVert \Sigma_X \rVert_2$ is affected by not only the growing diagonal blocks $\sigma^2\Gamma_1, \dots, \sigma^2\Gamma_n$ but also the growing off-diagonal blocks; note that for any $1\leq t,s\leq n$, the $(t,s)$th block of $\Sigma_X$ is
	\begin{equation}\label{eq:lem_Sigma_X_eq1}
	E(X_tX_s^\T)=\begin{cases}
	\sigma^2 \Gamma_t (A_*^\T)^{s-t}, & \text{ if } t< s\\
	\sigma^2 A_*^{t-s} \Gamma_s, & \text{ if } t\geq s
	\end{cases}.
	\end{equation}
	
	To overcome this difficulty, under Assumption \ref{assum:stable2}, we consider the following `coupled' stable $\VAR(1)$ process $\{\widetilde{X}_t\}$ with  independent and identically distributed  innovations $\{\eta_t\}$ such that $E(\eta_t)=0$ and $\var(\eta_t)=\sigma^2 I_d$, but assuming that $\widetilde{X}_t$ starts from $t=-\infty$:
	\begin{equation}\label{eq:coupled}
	\widetilde{X}_{t+1}=A_* \widetilde{X}_{t}+\eta_t, \quad t\in\mathbb{Z}.
	\end{equation}
	Unlike $\{X_t\}_{t\geq 0}$ in the main paper, this process is weakly stationary. Indeed, for any $t\in\mathbb{Z}$, it holds $E(\widetilde{X}_t)=0$ and 
	\begin{equation}\label{eq:lem_Sigma_X_eq2}
	E(\widetilde{X}_t \widetilde{X}_{t+k}^\T)=\begin{cases}
	\sigma^2 \Gamma_{\infty} (A_*^\T)^{k}, & \text{ if } k>0\\
	\sigma^2 A_*^{k} \Gamma_{\infty}, & \text{ if } k\leq 0
	\end{cases},
	\end{equation}
	where $\Gamma_{\infty}=\lim_{n\rightarrow\infty}\Gamma_{n}=\sum_{s=0}^{\infty} A_*^s (A_*^\T)^s<\infty$. Analogously to $\Sigma_X$, let $\widetilde{\Sigma}_X$ be the symmetric $dn\times dn$ matrix with its  $(t,s)$th $d\times d$ block being $E(\widetilde{X}_t \widetilde{X}_{s}^\T)$ for $1\leq t,s\leq n$. In other words, $\widetilde{\Sigma}_X$ is the covariance matrix of the  $d n\times 1$ vector $\vect(\widetilde{X}^\T)=(\widetilde{X}_1^\T, \dots, \widetilde{X}_n^\T)^\T$. Note that in contrast to $\Sigma_X$, the blocks of $\widetilde{\Sigma}_X$ do not grow in the diagonal direction, in the sense that all $E(\widetilde{X}_t \widetilde{X}_{s}^\T)$'s share the same factor matrix $\Gamma_{\infty}$. By \cite{Basu_Michailidis2015}, for the weakly stationary $\VAR(1)$ process $\{\widetilde{X}_t\}$ in \eqref{eq:coupled}, it holds
	\begin{equation}\label{eq:lem_Sigma_X_eq3}
	\lVert \widetilde{\Sigma}_X \rVert_2 \leq \frac{\sigma^2}{\mu_{\min}(\mathcal{A})}\leq \frac{\sigma^2}{\mu_1},
	\end{equation}
	where $\mu_{\min}(\mathcal{A})\geq \mu_1>0$ is defined as in Assumption \ref{assum:stable2}.
	
	In view of \eqref{eq:lem_Sigma_X_eq3} and the triangle inequality
	\begin{equation}\label{eq:lem_Sigma_X_eq4}
	\lVert \Sigma_X \rVert_2 \leq \lVert \widetilde{\Sigma}_X \rVert_2 + \lVert \widetilde{\Sigma}_X - \Sigma_X \rVert_2,
	\end{equation}
	it remains to prove that $\lVert \widetilde{\Sigma}_X - \Sigma_X \rVert_2 \lesssim \sigma^2$. To this end, for any $1\leq t,s\leq n$, consider the difference between the $(t,s)$th blocks of $ \widetilde{\Sigma}_X$ and $\Sigma_X$:
	\begin{align}\label{eq:lem_Sigma_X_eq5}
	E(\widetilde{X}_t\widetilde{X}_s^\T)-E(X_tX_s^\T)
	&=\begin{cases}
	\sigma^2 (\Gamma_{\infty}-\Gamma_t) (A_*^\T)^{s-t}, & \text{ if } t < s\\
	\sigma^2 A_*^{t-s}(\Gamma_{\infty}-\Gamma_s), & \text{ if } t\geq s
	\end{cases} \notag\\
	&=\sigma^2 A_*^t\Gamma_\infty (A_*^\T)^s.
	\end{align}
	Note that under Assumption \ref{assum:stable2}, $\lVert \Gamma_{\infty} \rVert_2\leq  \sum_{t=0}^{\infty} \lVert A_*^t\rVert_{2}^2\leq C^2 \sum_{t=0}^{\infty} \varrho^{2t} = \frac{C^2}{1-\varrho^2}$, where $\varrho\in(0,1)$. This, together with \eqref{eq:lem_Sigma_X_eq5}, implies that for any $1\leq t,s\leq n$,
	\begin{equation*}
	\lVert E(\widetilde{X}_t\widetilde{X}_s^\T)-E(X_tX_s^\T) \rVert_{2} \leq \sigma^2 \lVert A_*^t \rVert_{2} \lVert \Gamma_{\infty} \rVert_2\lVert A_*^s\rVert_{2} \leq \frac{C^2 \sigma^2 \varrho^{t+s}}{1-\varrho^2}.
	\end{equation*}
	Consequently, for any $u=(u_1^\T, \dots, u_n^\T)^\T \in\mathcal{S}^{dn-1}$ with $u_t\in\mathbb{R}^d$, we have
	\begin{align*}
	u^\T (\widetilde{\Sigma}_X - \Sigma_X) u &= \sum_{t=1}^{n}\sum_{s=1}^{n} u_t^\T \{E(\widetilde{X}_t\widetilde{X}_s^\T)-E(X_tX_s^\T)\} u_s\\
	&\leq \sum_{t=1}^{n}\sum_{s=1}^{n} \frac{u_t^\T \{E(\widetilde{X}_t\widetilde{X}_s^\T)-E(X_tX_s^\T)\}  u_s}{\lVert u_t \rVert \lVert u_s \rVert}\\
	&\leq \sum_{t=1}^{n}\sum_{s=1}^{n} \lVert E(\widetilde{X}_t\widetilde{X}_s^\T)-E(X_tX_s^\T) \rVert_{2} \\
	&\leq \frac{C^2 \sigma^2}{1-\varrho^2} \sum_{t=1}^{n}\sum_{s=1}^{n} \varrho^{t+s} \leq \frac{C^2\sigma^2 \varrho^2}{(1-\varrho^2)(1-\varrho)^2}.
	\end{align*}
	Thus,
	\begin{equation}\label{eq:lem_Sigma_X_eq6}
	\lVert \widetilde{\Sigma}_X - \Sigma_X \rVert_2 \leq \frac{C^2\sigma^2 \varrho^2}{(1-\varrho^2)(1-\varrho)^2}.
	\end{equation}
	Combining \eqref{eq:lem_Sigma_X_eq3}, \eqref{eq:lem_Sigma_X_eq4} and \eqref{eq:lem_Sigma_X_eq6}, the proof of this lemma is complete.
\end{proof}

\section{Proof of Theorem \ref{thm3}\label{sec:S4}}
We will prove claim (i) of Theorem \ref{thm3} only, as claims (ii) and (iii) can be proved by a method similar to that for (i). 

First, by an argument similar to that in \citet[p.~199]{Lutkepohl2005}, we can show that 
\[
R \left \{R^\T (I_d\otimes \Gamma_k) R\right \}^{-1} R^\T \preceq I_d\otimes \Gamma_{k}^{-1}. 
\]
In addition, note that 
\[
\lambda_{\min}(\Gamma_{k}) \geq \sum_{s=0}^{k-1} \lambda_{\min}\{A_*^s (A_*^\T)^s \}
= \sum_{s=0}^{k-1} \sigma_{\min}^s(A_* A_*^\T) 
\geq \sum_{s=0}^{k-1}\sigma_{\min}^{2s}(A_*).
\]
As a result,
\begin{equation}\label{eq:thm3_aeq1}
\lambda_{\max}\left [R \left \{R^\T (I_d\otimes \Gamma_k) R\right \}^{-1} R^\T\right ] \leq \lambda_{\max}(\Gamma_{k}^{-1})=\frac{1}{\lambda_{\min}(\Gamma_{k})}\leq  \frac{1}{\sum_{s=0}^{k-1}\sigma_{\min}^{2s}(A_*)}.
\end{equation}

Now we prove the rate in \eqref{eq:slow1} under condition \eqref{eq:sigmin_slow1}. By the existence condition of $k$ in \eqref{eq:k_sufficient}, we can choose
\begin{equation}\label{eq:max_k}
k=\frac{c_0 n}{m\left [\log\{d\cond(S)/\delta\}+ b_{\max} \log n \right ]},
\end{equation}
where $c_0>0$ is a universal constant. Then, \eqref{eq:sigmin_slow1} can be written as
\begin{equation}\label{eq:thm3_aeq2}
\sigma_{\min}(A_*) \leq 1-c_2 /k,
\end{equation}
where $c_2=c_1c_0>0$.
Since
\[
\frac{1}{\sum_{s=0}^{k-1}\sigma_{\min}^{2s}(A_*)}=\frac{1-\sigma_{\min}^{2}(A_*)}{1-\sigma_{\min}^{2k}(A_*)},
\]
by Theorem \ref{thm2}(i) and \eqref{eq:thm3_aeq1}, to prove the rate for $\lVert \widehat{\beta} - \beta_* \rVert$ in \eqref{eq:slow1}, it suffices to show that there exists a universal constant $c_3\in(0,1)$ such that
\begin{equation}\label{eq:thm3_aeq3}
1-\sigma_{\min}^{2k}(A_*)\geq c_3.
\end{equation}
Moreover, by \eqref{eq:thm3_aeq2}, we can show that \eqref{eq:thm3_aeq3} is satisfied if 
\begin{equation}\label{eq:thm3_aeq4}
-2k \log(1-c_2 / k ) \geq -\log(1-c_3).
\end{equation}
Note that the function $f(k)=-2k \log(1-c_2 / k )$ monotonically deceases to $2c_2$ as $k\rightarrow\infty$. Thus, by choosing $c_3$ such that $-\log(1-c_3)=2c_2$, i.e., $c_3=1-\exp(-2c_2) \in(0,1)$, we accomplish the proof of \eqref{eq:slow1}.

Next we prove the rate in \eqref{eq:fast1} when the opposite of \eqref{eq:sigmin_slow1} is true, i.e., when
\begin{equation}\label{eq:sigmin_fast1}
\sigma_{\min}(A_*)\geq 1- \frac{c_1 m\left [\log\{d\cond(S)/\delta\}+ b_{\max} \log n \right ]}{n}.
\end{equation}	
Again, we choose $k$ in \eqref{eq:max_k}, and then \eqref{eq:sigmin_fast1} becomes
\begin{equation*}
\sigma_{\min}(A_*) \geq 1-c_2 /k,
\end{equation*}
where $c_2$ is defined as in \eqref{eq:thm3_aeq2}.
Thus,
\begin{equation}\label{eq:thm3_aeq5}
\sum_{s=0}^{k-1}\sigma_{\min}^{2s}(A_*) \geq \sum_{s=0}^{k-1} (1-c_2 /k)^{2s} \geq k (1-c_2 /k)^{2k}.
\end{equation}

In view of Theorem \ref{thm2}(i), \eqref{eq:thm3_aeq1}, \eqref{eq:max_k} and \eqref{eq:thm3_aeq5}, to prove the rate for $\lVert \widehat{\beta} - \beta_* \rVert$ in \eqref{eq:fast1}, we only need to show that there exists a universal constant $c_4\in(0,1)$ such that
\begin{equation}
(1-c_2 /k)^{2k} \geq c_4.
\end{equation}
By the choice of $k$ in \eqref{eq:max_k},  we have $k > c_2$. Hence, there exists $\epsilon>0$ such that $k\geq c_2+\epsilon$. Moreover, notice that the function $g(k)=(1-c_2 /k)^{2k}$ is monotonically increasing in $k$. As a result, by choosing $c_4=g(c_2+\epsilon)$, we complete the proof of \eqref{eq:fast1}.

\section{Proofs of Theorem \ref{thm_lower} and Corollary \ref{cor_lower} \label{sec:S5}}
\subsection{Two Auxiliary Lemmas}
The proof of Theorem \ref{thm_lower} is based upon Lemmas \ref{lem_A4} and \ref{lem_A5} below.
Denote by $\kl(\mathbb{Q},\mathbb{P})$ the Kullback-Leibler divergence between two probability measures $\mathbb{P}$ and $\mathbb{Q}$  on the same measurable space.

\begin{lemma}\label{lem_A4} Fix $\delta\in(0,1/2)$, $\epsilon>0$ and $R\in\mathbb{R}^{N\times m}$. Suppose that $\mathcal{N}$ is a finite subset of $\mathbb{R}^{m}$ such that $\lVert R(\theta_1-\theta_2) \rVert\geq 2\epsilon, \forall \theta_1\neq \theta_2\in\mathcal{N}$. 
	If 
	\begin{equation}\label{eq:lem_F1_condition}
	\inf_{\widehat{\theta}}\sup_{\theta\in \mathcal{N}} \P_{\theta}^{(n)} \left \{ \lVert R( \widehat{\theta}-\theta) \rVert \geq \epsilon \right \} \leq \delta,
	\end{equation}
	where the infimum is taken over all estimators of $\theta$ which are $\mathcal{F}_{n+1}$-measurable,
	then
	\begin{equation*}
	\inf_{\theta_0\in\mathcal{N}} \sup_{\theta\in\mathcal{N}\setminus \{\theta_0\}} \kl(\P_{\theta}^{(n)}, \P_{\theta_0}^{(n)}) \geq (1-2\delta)\log\frac{\vert \mathcal{N}\vert-1}{2\delta}.
	\end{equation*}
\end{lemma}

\begin{proof}[Proof of Lemma~\ref{lem_A4}]
	For any $\mathcal{F}_{n+1}$-measurable estimator $\widehat{\theta}$, let $\mathcal{E}_{\theta}=\{ \lVert R(\widehat{\theta}-\theta) \rVert < \epsilon\}$ for $\theta\in\mathcal{N}$. Since $\mathcal{N}$ is a $2\epsilon$-packing of $\mathbb{R}^{m}$, the events $\mathcal{E}_{\theta}$'s with $\theta\in\mathcal{N}$ are pairwise disjoint in $\mathcal{F}_{n+1}$. By \eqref{eq:lem_F1_condition}, there exists a $\widehat{\theta}$ such that $\sup_{\theta\in \mathcal{N}}\P_{\theta}^{(n)}(\mathcal{E}_{\theta}^c)\leq \delta <1/2$, i.e., $\inf_{\theta\in \mathcal{N}}\P_{\theta}^{(n)}(\mathcal{E}_{\theta})\geq 1-\delta >1/2$. Applying Birg\'{e}'s inequality \citep[Theorem 4.21]{Boucheron_Lugosi_Massart2013} and an argument similar to that for Lemma F.1 in \cite{Simchowitz2018}, we can readily prove that for any $\theta_0\in\mathcal{N}$,
	\[
	\sup_{\theta\in \mathcal{N}\setminus\{\theta_0\}}\kl(\P_{\theta}^{(n)},\P_{\theta_0}^{(n)})\geq (1-2\delta) \log\frac{(1-\delta)(\vert\mathcal{N}\vert-1)}{\delta}\geq (1-2\delta)\log\frac{\vert \mathcal{N}\vert-1}{2\delta}.
	\]
	Taking the infimum over $\theta_0\in\mathcal{N}$, we accomplish the proof of this lemma.
\end{proof}

\begin{lemma}\label{lem_A5}
	For the linearly restricted vector autoregressive model, under the conditions of Theorem \ref{thm_lower}, for any $\theta, \theta_0\in\mathbb{R}^{m}$, we have
	\[
	\kl(\P_{\theta}^{(n)}, \P_{\theta_0}^{(n)})=\frac{1}{2} (\theta-\theta_0)^\T \Gamma_{R,n}(\theta) (\theta-\theta_0),
	\]  
	where
	$\Gamma_{R,n}(\theta)=\sum_{i=1}^{d} R_i^\T \sum_{t=1}^{n}  \Gamma_{t}(\theta) R_i =R^\T \{I_d\otimes\sum_{t=1}^{n}  \Gamma_{t}(\theta)\} R$.
\end{lemma}

\begin{proof}[Proof of Lemma~\ref{lem_A5}]	
	Without loss of generality, we assume that $\gamma=0$, so that $\beta=R\theta$. Let $X_{i,t}$ be the $i$th entry of $X_t$, and denote $Z_{i,t}=R_i^\T X_t$. For any $\theta\in\mathbb{R}^{m}$, under $\P_{\theta}^{(n)}$ we have $X_{i,t+1}\mid \mathcal{F}_t \sim N(\theta^\T Z_{i,t}, \sigma^2)$, where $0\leq t \leq n$ and $\mathcal{F}_0=\emptyset$. Hence, the log-likelihood of $(X_1,\ldots, X_{n+1})$ under $\P_{\theta}^{(n)}$ is
	\begin{align*}
	&\log \prod_{t=0}^{n} \prod_{i=1}^{d}\frac{1}{(2\pi)^{1/2} \sigma} \exp\left \{-\frac{(X_{i,t+1}-\theta^\T Z_{i,t})^2}{2\sigma^2}\right \}\\ &\hspace{5mm}= -(n+1)d\log((2\pi)^{1/2} \sigma)-\frac{1}{2\sigma^2}\sum_{t=0}^{n}\sum_{i=1}^{d}(X_{i,t+1}-\theta^\T Z_{i,t})^2.
	\end{align*}
	As a result,
	\begin{align*}
	\kl(\P_{\theta}^{(n)}, \P_{\theta_0}^{(n)})=E_{\P_{\theta}^{(n)}} \left ( \log \frac{d \P_{\theta}^{(n)}}{d \P_{\theta_0}^{(n)}} \right )
	&=\frac{1}{2}\sum_{t=0}^{n} \sum_{i=1}^{d} E_{\P_{\theta}^{(n)}}\left[\{ \eta_{i,t}+(\theta-\theta_0)^\T Z_{i,t}\}^2- \eta_{i,t}^2\right]\\
	&=\frac{1}{2} (\theta-\theta_0)^\T \sum_{t=1}^{n} \sum_{i=1}^{d} E_{\P_{\theta}^{(n)}}(Z_{i,t}Z_{i,t}^\T) (\theta-\theta_0)\\
	&=\frac{1}{2} (\theta-\theta_0)^\T \Gamma_{R,n}(\theta) (\theta-\theta_0),
	\end{align*}
	where the last equality is because of $E_{\P_{\theta}^{(n)}}(Z_{i,t}Z_{i,t}^\T)=R_i^\T \Gamma_{t}(\theta) R_i$. The proof is complete.
\end{proof}

\subsection{Proof of Theorem \ref{thm_lower}}
Without loss of generality, we assume that $\gamma=0$, so that $\beta=R\theta$.
Define the ellipsoid $E=\{\theta\in \mathbb{R}^m: \lVert R\theta\rVert\leq \bar{\rho} \} =\{ (R^\T R)^{-1/2}\omega: \omega\in B(0,\bar{\rho})\}$,
where $B(0,r)$ denotes the Euclidean ball in $\mathbb{R}^m$ with center zero and radius $r$. Since  $\rho\{A(\theta)\} \leq \lVert A(\theta) \rVert_F=\left (\sum_{i=1}^d\lVert R_i\theta \rVert^2\right )^{1/2}=\lVert R\theta \rVert$, we have $E\subseteq \Theta(\bar{\rho})$.

For any $\epsilon\in(0,\bar{\rho}/4]$, let $\mathcal{N}_1$ be a maximal $2\epsilon$-packing of $B(0,4\epsilon)$ in $\mathbb{R}^m$, and define 	$\mathcal{N}=\{(R^\T R)^{-1/2}\omega: \omega\in\mathcal{N}_1\}$.
Then, $\mathcal{N}$ is a $2\epsilon$-packing of $E$ in the norm $\lVert (R^\T R)^{1/2}(\cdot)\rVert$. As a result, $2\epsilon \leq \lVert R(\theta-\theta_0) \rVert\leq 8\epsilon$ for all $\theta\neq \theta_0\in\mathcal{N}$. In addition, by a standard volumetric argument, we have $\vert\mathcal{N} \vert=\vert \mathcal{N}_1\vert\geq 2^m$.	
By Lemma \ref{lem_A4}, for any $\delta\in(0,1/2)$, this theorem holds if
\begin{equation}\label{kl_condition1}
\inf_{\theta_0\in\mathcal{N}} \sup_{\theta\in\mathcal{N}\setminus \{\theta_0\}} \kl(\P_{\theta}^{(n)}, \P_{\theta_0}^{(n)}) < (1-2\delta)\log\frac{\vert \mathcal{N}\vert-1}{2\delta}.
\end{equation}

Since $\sum_{t=1}^{n} \Gamma_{t}(\theta) \preceq n \Gamma_{n}(\theta)$ for any $\theta\in\mathbb{R}^m$, and
\[
\sup_{\theta\in\Theta(\bar{\rho})}\lambda_{\max}\{\Gamma_{n}(\theta)\} \leq  \sum_{s=0}^{n-1} \lambda_{\max}[ A^s(\theta) \{A^\T(\theta)\}^s] \leq  \sum_{s=0}^{n-1}\bar{\rho}^{2s}
=\gamma_n(\bar{\rho}),
\]
it follows from Lemma \ref{lem_A5} that
\begin{align*}
\max_{\theta,\theta_0\in\mathcal{N}}\kl(\P_{\theta}^{(n)}, \P_{\theta_0}^{(n)}) 
&\leq \frac{1}{2}  \max_{\theta,\theta_0\in\mathcal{N}}(\theta-\theta_0)^\T \Gamma_{R,n}(\theta) (\theta-\theta_0)\\
&\leq \frac{n}{2} \max_{\theta,\theta_0\in\mathcal{N}}(\theta-\theta_0)^\T R^\T \{I_d\otimes \Gamma_{n}(\theta)\} R (\theta-\theta_0)\\
&\leq \frac{n}{2} \max_{\theta,\theta_0\in\mathcal{N}} \lVert R (\theta-\theta_0) \rVert^2 \sup_{\theta\in\Theta(\bar{\rho})}\lambda_{\max}\{\Gamma_{n}(\theta)\}\\
&\leq 32\epsilon^2 n \gamma_n(\bar{\rho})
\end{align*}
As a result, a sufficient condition for \eqref{kl_condition1} is 
\begin{equation}\label{kl_condition2}
n \gamma_n(\bar{\rho}) < \frac{(1-2\delta)}{32\epsilon^2}\log\frac{2^m}{4\delta}.
\end{equation}
In particular, for any $\delta\in(0,1/4)$, we can show that there exists a universal constant $c>0$ such that the right-hand side of \eqref{kl_condition2} is bounded below by $c\{m+\log(1/\delta)\}/\epsilon^2$, i.e., the conclusion of this theorem follows.

\subsection{Proof of Corollary \ref{cor_lower}}
Under the conditions of Theorem \ref{thm_lower}, we have
\[
\inf_{\widehat{\beta}}\sup_{\theta\in \Theta(\bar{\rho})} \P_{\theta}^{(n)} \left[ \lVert \widehat{\beta}-\beta\rVert \geq C \left \{\frac{m+\log(1/\delta)}{n \gamma_n(\bar{\rho})}\right \}^{1/2} \right] \geq \delta,
\]
where $C>0$ is fixed. It then suffices to derive lower bounds of $1/\gamma_n(\bar{\rho})$ for $\bar{\rho}\in(0,\infty)$.

First, suppose that $\bar{\rho}\in(0,1)$. Then we have $\gamma_n(\bar{\rho})=\sum_{s=0}^{n-1}\bar{\rho}^{2s}=(1-\bar{\rho}^{2n})/(1-\bar{\rho}^2)<\min\{n, (1-\bar{\rho}^2)^{-1}\}$, and therefore
\begin{equation}\label{eq:cor_lower_aeq1}
\frac{1}{\gamma_n(\bar{\rho})}>
\begin{cases}
1-\bar{\rho}^2, & \text{if}\quad \bar{\rho}\in (0, (1-1/n)^{1/2})\\
1/n, & \text{if}\quad \bar{\rho}\in [(1-1/n)^{1/2}, 1)
\end{cases}.
\end{equation}

Next, suppose that $\bar{\rho}\in[1, 1+c/n]$ for a fixed $c>0$. Then
\[
\frac{\gamma_n(\bar{\rho})}{n}=\frac{1}{n}\sum_{s=0}^{n-1}\bar{\rho}^{2s} \leq \frac{1}{n}\sum_{s=0}^{n-1} (1+c/n)^{2s}\leq (1+c/n)^{2n}.
\]
Since $(1+c/n)^{2n}$ monotonically increases to $\exp(2c)$ as $n\rightarrow\infty$, there exists a constant $C_2>0$ free of $n$ such that $\gamma_n(\bar{\rho})/n$ is uniformly bounded above by $C_2$, i.e., 
\begin{equation}\label{eq:cor_lower_aeq2}
\frac{1}{\gamma_n(\bar{\rho})} \geq \frac{1}{C_2}n^{-1} \quad\text{if}\quad \bar{\rho}\in[1, 1+c/n].
\end{equation}

Moreover, for any $\bar{\rho}\in(1,\infty)$, we have
\begin{equation}\label{eq:cor_lower_aeq3}
\frac{1}{\gamma_n(\bar{\rho})} =\frac{\bar{\rho}^2-1}{\bar{\rho}^{2n}-1}> \frac{\bar{\rho}^2-1}{\bar{\rho}^{2n}}.
\end{equation}

Combining \eqref{eq:cor_lower_aeq1}--\eqref{eq:cor_lower_aeq3}, we have
\[
\left \{\frac{m}{n \gamma_n(\bar{\rho})}\right \}^{1/2} \geq 
\begin{cases}
\left \{{(1-\bar{\rho}^2)m}/{n}\right \}^{1/2}, & \text{if}\quad \bar{\rho}\in (0, (1-1/n)^{1/2})\\
m^{1/2}/n, & \text{if}\quad \bar{\rho}\in [(1-1/n)^{1/2}, 1+c/n]\\
\bar{\rho}^{-n}\left \{{(\bar{\rho}^2-1)m}/{n}\right \}^{1/2}, & \text{if}\quad \bar{\rho}\in(1+c/n, \infty)
\end{cases},
\]
and this completes the proof of Corollary \ref{cor_lower}.


\clearpage

\end{document}